\documentclass[11pt,reqno]{amsart}

\usepackage{amsmath,amssymb,mathrsfs}
\usepackage{graphicx,cite,times}
\usepackage{hyperref}
\hypersetup{
    colorlinks=true,
    linkcolor=blue,
    filecolor=gray,
    urlcolor=blue,
    citecolor=blue,
}

\usepackage[doipre={doi:~}]{uri}
\usepackage{cases}
 \usepackage{dsfont}
 \usepackage[mathscr]{eucal}
 \usepackage{color}
\newtheorem{theo}{Theorem}[section]
\newtheorem{coro}[theo]{Corollary}
\newtheorem{lemm}[theo]{Lemma}
\newtheorem{prop}[theo]{Proposition}
\newtheorem{rema}[theo]{Remark}
\newtheorem{defi}[theo]{Definition}

\renewcommand{\Im}{\operatorname{Im}}
\numberwithin{equation}{section}

\begin{document}

\title[Modal approximation for elastic quasiparticles in time-domain]{
Modal approximation for time-domain elastic scattering from metamaterial quasiparticles}

\author{Bochao Chen}
\address{School of
Mathematics and Statistics, Center for Mathematics and
Interdisciplinary Sciences, Northeast Normal University, Changchun, Jilin 130024, P.R.China. }
\email{chenbc758@nenu.edu.cn}

\author{Yixian Gao}
\address{School of
Mathematics and Statistics, Center for Mathematics and
Interdisciplinary Sciences, Northeast Normal University, Changchun, Jilin 130024, P.R.China. }
\email{gaoyx643@nenu.edu.cn}

\author{Hongyu Liu}
\address{Department of Mathematics, City University of Hong Kong, Kowloon, Hong Kong, China }
\email{hongyu.liuip@gmail.com, hongyliu@cityu.edu.hk}

\thanks{The research of BC was supported in part by  NSFC grants 11901232. The research of YG was  supported by NSFC grants 11871140, 12071065 and National Key R\&D Program of China 2020YFA0714102.
 The research
of HL was supported by Hong Kong RGC General Research Funds (project numbers,
11300821, 12301218 and 12302919) and the NSFC-RGC Joint Research Grant (project number,
 N\_CityU101/21). }



\begin{abstract}
This paper aims at quantitatively understanding the elastic wave scattering due to negative metamaterial structures under wide-band signals in the time domain. Specifically, we establish the modal expansion for the time-dependent field scattered by metamaterial quasiparticles in elastodynamics. By Fourier transform, we first analyze the modal expansion in the time-harmonic regime. With the presence of quasiparticles, we validate such an expansion in the static regime via quantitatively analyzing the spectral properties of the Neumann-Poincar\'{e} operator associated with  the elastostatic system. We then approximate the incident field with a finite number of modes and apply perturbation theory to obtain such an expansion in the perturbative regime. In addition, we give polariton resonances as simple poles for the elastic system with non-zero frequency. Finally, we show that  the low-frequency part of the scattered field in the time domain can be well approximated by using the resonant modal expansion with sharp error estimates.

\end{abstract}

\maketitle


\section{Introduction}\label{sec:1}

Due to the rapid technological advancement, elastic metamaterials with negative bulk moduli and/or negative density have been realised via many means \cite{WLZ11,LHHS11,LLZ21}. Moreover, striking applications have been proposed and realised by using elastic metamaterial structures including super-focusing and superlensing \cite{ZLHSH14,C11, LGBLLPA} and invisibility cloaking \cite{LGBLLPA,AKKY17,DLL20b,LLL18,LL17,LL16,LLZ22}. The fundamental basis for those applications is the polariton resonances induced by the metamaterial structures (in the subwavelength scale), which have also received considerable attention in the mathematical literature \cite{DLL20a,DLL20b,LLL18}. We would also like to mention in passing some related mathematical studies on plasmon resonances for optical metamaterial structures; see \cite{ADM,ACKLM13,APRYZ19,BLLW20,DLZ21,DLZ22,LLL15,LLLW19,LL18,MN06,RS19} and a recent survey paper \cite{MM20} as well as the references cited therein.

In quantitatively understanding the metamaterial structures, most of the studies in the literature are devoted to the time-harmonic scattering, especially on the resonant frequencies and resonant fields. In contrast, the study of time-dependent wave propagation due to negative metamaterial structures is much less touched. However, it is of both practical and theoretical importance to quantitatively understand the performance of negative metamaterial structures under wide-band signals in the time domain. It is the aim of the present article to quantitatively study the time-dependent field scattered by metamaterial quasiparticles in elastodynamics. To that purpose, we shall adopt the modal analysis, which is a powerful tool to understand complex wave physics; see \cite{ramm1982mathematical} for a general presentation of resonance expansions. Notably, modal analysis for wave propagation in an unbounded domain \cite{popov1999distribution,popov1999resnances} as well as from negative metamaterial structures \cite{ammari2022modal,Baldassari2021modal,cassier2017spectrum,cassier2017mathematical} present significant challenges in either scenario, since many classical results from spectral theory do not apply, say e.g. the corresponding wave operator cannot be diagonalised with the classical spectral theorem. In our study, we consider the Kelvin-Voigt model for modelling viscoelastic materials  \cite{Oestreicher1951field,demchenko2019on} and assume the presence of a metamaterial structure with an asymptotically small characteristic size, which is referred to as a metamaterial quasiparticle. We establish a polariton resonance expansion with sharp error estimates for the low-frequency part of the elastic field scattered by the aforementioned elastic quasiparticle. Our study shares a similar spirit to the recent works \cite{ammari2022modal,Baldassari2021modal} which established the plasmon resonance expansion of the electromagnetic field scattered by nanoparticles with dispersive material parameters placed in a homogeneous medium in the low-frequency regime. However, it is remarked that the elastic waves in solid media where the coupled longitudinal and transverse waves bring challenging but richer wave phenomena. Indeed, in our modal analysis, a key wave operator, known as the static elastic Neumann-Poincar\'e operator, is non-compact in any geometric setup and is self-adjoint only in the radial geometry \cite{Deng2019spectral}. It is also pointed out that the non-static elastic Neumann-Poincar\'e operator is both non-compact and non-self-adjoint. This brings significant challenges in our spectral analysis and hence we shall mainly confine within the radial geometry. However, such a geometric simplification is both practically and theoretically unobjectionable. In fact, from a practical point of view, many metamaterial applications are based on collective behaviours of certain building elements which are usually of radial shape \cite{LLZ21}, and moreover from a theoretical point of view, even in the radial geometric setup, the modal analysis presents sufficient technicalities and subtleties. We also refer to Section~\ref{sec:conclusion} for more relevant discussion and remark.

Next, we introduce the mathematical setup and the elastic scattering problem for our study. Denote by $D$ an elastic quasiparticle in our study, which is of the form:
\begin{equation}\label{eq:D}
D=\delta B+\boldsymbol z,
\end{equation}
where $\delta\in\mathbb{R}_+$, the point $\boldsymbol z\in\mathbb{R}^3$ and $B$ is a unit ball containing the origin in $\mathbb{R}^3$. That is, $D$ is located at $\boldsymbol z$ with a characteristic size $\delta\ll1$. The elastic medium configuration in $\mathbb{R}^3$ is characterized by the Lam\'{e} parameters $\hat{\lambda}$ and $\hat{\mu}$:
\begin{equation*}\label{eq:lp1}
(\mathbb{R}^3; \hat\lambda, \hat\mu)=(D; \lambda_1, \mu_1)\cup (\mathbb{R}^3\backslash\overline{D}; \lambda, \mu).
\end{equation*}
Here  $(\lambda,\mu)$ denotes the Lam\'{e} parameter of the homogeneous background medium and satisfies the following strong convexity conditions:
\begin{align}\label{convexity}
\mathrm{(i)}~\mu>0 \qquad\text{and}\qquad\mathrm{(ii)}~3\lambda+2\mu>0.
\end{align}
It is assumed that $(\lambda_1, \mu_1)=-\alpha\cdot(\lambda, \mu)$ with $\alpha\in\mathbb{R}\setminus\{0\}$ to be more definitely specified in what follows. Introduce the isotropic elasticity tensor $\widehat{\mathds{C}}=(\widehat{C}_{ijkl})^3_{i,j,k,l=1}$ as follows:
\begin{align*}
\widehat{C}_{ijkl}:=\hat{\lambda} \boldsymbol\delta_{ij}\boldsymbol\delta_{kl}+\hat{\mu}(\boldsymbol\delta_{ik}\boldsymbol\delta_{jl}+\boldsymbol\delta_{il}\boldsymbol\delta_{jk}),
\end{align*}
where $\boldsymbol\delta$ signifies the Kronecker delta.

Denote by $\hat{\mathbf{u}}(\boldsymbol x,t)=(\hat{u}_i(\boldsymbol x, t))_{i=1}^3$  the elastic wave field with
$\boldsymbol x=(x_1, x_2, x_3)\in\mathbb{R}^3$ and $t\in\mathbb{R}$.
Let $\nabla^s$ denote the symmetric gradient, i.e.,
\begin{align*}
\nabla^s\hat{\mathbf{u}}:=\frac12(\nabla \hat{\mathbf{u}}+\nabla \hat{\mathbf{u}}^\top),
\end{align*}
where $\nabla \hat{\mathbf{u}}$ stands for the matrix $(\partial_{x_j}\hat{u}_i)_{i,j=1}^3$ and the superscript $\top$ signifies the matrix transport.
The elastostatic operator and the conormal derivative on $\partial D$ are, respectively, defined as
\begin{align*}
\mathcal{L}_{\hat{\lambda},\hat{\mu}}\hat{\mathbf{u}}:=&\nabla\cdot\widehat{\mathds{C}}\nabla^s \hat{\mathbf{u}}=\hat{\mathbf{\mu}}\Delta\hat{\mathbf{u}}+(\hat{\lambda}+\hat{\mu})\nabla\nabla\cdot\hat{\mathbf{u}}  && \text{in}~~\mathbb R^3,\\
\frac{\partial}{\partial{\boldsymbol{\nu}}}\hat{\mathbf{u}}:=&\hat{\lambda}(\nabla \cdot \hat{\mathbf{u}})\boldsymbol{\nu}+2\hat{\mu}(\nabla^s \hat{\mathbf{u}})\boldsymbol{\nu}&&\text{on}~~\partial D,
\end{align*}
where $\boldsymbol{\nu}$ signifies the exterior unit normal vector to the boundary $\partial D$.
We consider the following time-domain Lam\'{e} system:
\begin{align}\label{S:Lame1}
\left\{ \begin{aligned}
&\mathcal{L}_{\lambda_1,\mu_1}\hat{\mathbf{u}}(\boldsymbol x,t)
+\mathcal{L}_{\lambda_2,\mu_2}\partial_t\hat{\mathbf{u}}(\boldsymbol x,t)
-\partial^2_{t}\hat{\mathbf{u}}(\boldsymbol x,t)=0,&&\boldsymbol x\in D,\\
&\mathcal{L}_{\lambda,\mu}\hat{\mathbf{u}}(\boldsymbol x,t)-\partial^2_{t}\hat{\mathbf{u}}(\boldsymbol x,t)=\delta_{\boldsymbol x}(\mathbf{s})\hat{f}(t)\mathbf{p},&&\boldsymbol x\in\mathbb{R}^3\backslash\overline{D},\\
&\hat{\mathbf{u}}(\boldsymbol x,t)|_+=\hat{\mathbf{u}}(\boldsymbol x,t)|_-,&&\boldsymbol x\in \partial D,\\
&\frac{\partial \hat{\mathbf{u}}(\boldsymbol x,t)}{\partial \boldsymbol{\nu}}\big|_+=\frac{\partial \hat{\mathbf{u}}(\boldsymbol x,t)}{\partial \boldsymbol{\nu}}\big|_-,&&\boldsymbol x\in\partial D,\\
&\hat{\mathbf{u}}(\boldsymbol x,t) \text{ satisfies the radiation condition},
\end{aligned}\right.
\end{align}
where $(\lambda_2, \mu_2)$ is the viscoelastic counterpart (``viscosity") of the medium in the quasiparticle $D$, $\mathbf p=(p_j)_{j=1}^3\in\mathbb{R}^3$ is a polarisation vector, and $\pm$ stand for the traces on $\partial D$ taken from outside and inside of the domain $D$, respectively. The source intensity
$\hat{f}:t\mapsto \hat{f}(t)\in C^{\infty}_0([0,C_1])$ with $C_1>0$ represents a wide-band signal and $\mathbf{s}\in\mathbb{R}^3$ is the source location. The viscosity $(\lambda_2, \mu_2)$ characterises the loss or radiation inside the medium, and we assume that $(\lambda_2, \mu_2)=\beta\cdot(\lambda,\mu)$ with $\beta\in\mathbb{R}\setminus\{0\}$ to be more definitely specified in what follows.
 The first equation in \eqref{S:Lame1} is known as the Kelvin-Voigt model which describes the viscoelastic wave propagation (cf. \cite{Oestreicher1951field,demchenko2019on}). The radiation condition in \eqref{S:Lame1} designates the following
two asymptotic relations as $|\boldsymbol x|\rightarrow+\infty$:
 \begin{equation*}
 \begin{split}
&(\nabla\times\nabla\times\hat{\mathbf{u}})(\boldsymbol x, t)\times\frac{\boldsymbol x}{|\boldsymbol x|}-\frac{1}{c_s}\nabla\times\partial_t\hat{\mathbf{u}}(\boldsymbol x, t)=\mathcal{O}(|\boldsymbol x|^{-2}),\\
&\frac{\boldsymbol x}{|\boldsymbol x|}\cdot (\nabla(\nabla\cdot\hat{\mathbf{u}}))(\boldsymbol x, t)-\frac{1}{c_p}\nabla\partial_t\hat{\mathbf{u}}(\boldsymbol x, t)=\mathcal{O}(|\boldsymbol x|^{-2}),
\end{split}
\end{equation*}
where
\begin{align}\label{E:speed}
c_s:=\sqrt{\mu},\quad c_p:=\sqrt{\lambda+2\mu},
\end{align}
signify the shear and compressional velocities of the elastic wave propagation, respectively.

%

As an intermediate step to analysing the scattering behaviour of \eqref{S:Lame1}, we need convert it into the frequency domain via Fourier analysis. To that end, we introduce the temporal Fourier transform and inverse Fourier transform as follows:
\begin{align*}
&(\mathcal{F}\hat{\mathbf{u}})(\boldsymbol x,\omega):=\frac{1}{2\pi}\int^{\infty}_{\infty}\hat{\mathbf{u}}(\boldsymbol x,t)e^{\mathrm{i}\omega t}\mathrm{d}t,\\
&(\mathcal{F}^{-1}\mathbf{u})(\boldsymbol x,t):=\int^{\infty}_{\infty}\mathbf{u}(\boldsymbol x,\omega)e^{-\mathrm{i}\omega t}\mathrm{d}\omega,
\end{align*}
where $\mathrm{i}:=\sqrt{-1}$ is the imaginary unit. Assume that $\hat{\mathbf{u}}(\boldsymbol x,\cdot)\in L^2(\mathbb{R})^3$ for fixed $\boldsymbol x$, and that the wave and its first derivative with respect to time decay sufficiently fast as time tends to infinity, namely, $\hat{\mathbf{u}}(\boldsymbol x,\cdot)\rightarrow0$ and $\frac{\partial}{\partial t}\hat{\mathbf{u}}(\boldsymbol x,\cdot)\rightarrow0$ as $|t|\rightarrow\infty$. Denote by $\mathbf{u}$ and $f$ the Fourier transforms of $\hat{\mathbf u}$ and $\hat{f}$, respectively. By applying temporal Fourier transform to \eqref{S:Lame1}, we obtain the following time-harmonic Lam\'{e} system:
\begin{align}\label{SS:lame}
\left\{ \begin{aligned}
&\mathcal{L}_{\tilde{\lambda},\tilde{\mu}}\mathbf{u}(\boldsymbol x, \omega)+\omega^2\mathbf{u}(\boldsymbol x, \omega)=0,&&\boldsymbol x\in D,\\
&\mathcal{L}_{\tilde{\lambda},\tilde{\mu}}\mathbf{u}(\boldsymbol x, \omega)+\omega^2\mathbf{u}(\boldsymbol x, \omega)=\delta_{\boldsymbol x}(\mathbf{s}){f}(\omega)\mathbf{p},&&\boldsymbol x\in\mathbb{R}^3\backslash\overline{D},\\
&\mathbf{u}(\boldsymbol x,\omega)|_+=\mathbf{u}(\boldsymbol x,\omega)|_-,&&\boldsymbol x\in \partial D,\\
&\frac{\partial \mathbf{u}(\boldsymbol x, \omega)}{\partial \boldsymbol{\nu}}\big|_+=\frac{\partial \mathbf{u}(\boldsymbol x, \omega)}{\partial \boldsymbol{\nu}}\big|_-,&&\boldsymbol x\in\partial D,\\
&\mathbf{u}(\boldsymbol x, \omega) \text{ satisfies the radiation condition}.
\end{aligned}\right.
\end{align}
In the following, for brevity, we no longer write the dependence of $\omega$ for the
 function $\mathbf{u}$. Here,  $(\tilde{\lambda},\tilde{\mu})=a(\boldsymbol x, \omega)(\lambda,\mu)$ with $(\lambda,\mu)$ satisfying conditions \eqref{convexity} and
\begin{align}\label{E:A}
a(\boldsymbol x, \omega)=
\left\{ \begin{aligned}
&c(\omega), \quad&&\boldsymbol x\in D,\\
&1, \quad&&\boldsymbol x\in\mathbb{R}^3\backslash\overline{D},
\end{aligned}\right.
\end{align}
where
\begin{align}\label{c:omega}
c(\omega)=-\alpha+\mathrm{i}\beta\omega.
\end{align}
If $\omega$ is in a bounded interval, then there exists some constant $C>0$ such that
\begin{align*}
 |c(\omega)|\leq C, \quad |\frac{1}{\sqrt{c(\omega)}}|\leq C.
 \end{align*}
The radiation condition in \eqref{SS:lame} designates the following asymptotic relations as $|\boldsymbol x|\rightarrow+\infty$,
\begin{align*}
&(\nabla\times\nabla\times\mathbf{u})(\boldsymbol x)\times\frac{\boldsymbol x}{|\boldsymbol x|}+\mathrm{i}k_s\nabla\times\mathbf{u}(\boldsymbol x)=\mathcal{O}(|\boldsymbol x|^{-2}),\\
&\frac{\boldsymbol x}{|\boldsymbol x|}\cdot (\nabla(\nabla\cdot\mathbf{u}))(\boldsymbol x)+\mathrm{i}k_p\nabla\mathbf{u}(\boldsymbol x)=\mathcal{O}(|\boldsymbol x|^{-2}),
\end{align*}
where  $k_s$ and $k_p$ are, respectively, $s$-wavenumber and $p$-wavenumber with
\begin{align*}
k_s=\frac{\omega}{c_s},\quad k_p=\frac{\omega}{c_p}.
\end{align*}
Our goal of this paper is to give the modal decomposition of the scattered field of system \eqref{SS:lame} in the frequency
domain, and then to establish a resonance expansion for the low-frequency part of the scattered field of system
\eqref{S:Lame1} in the time domain.

The paper is organized as follows. Section \ref{sec:2} is devoted to preliminary knowledge of the layer potentials. Using layer potential techniques, we reduce the time-harmonic problem to a singularly perturbed system of integral equations.  In Section \ref{sec:3}, we analyze the spectral properties of the Neumann-Poincar\'{e} operator associated with  the elastostatic system and
establish  the corresponding modal expansion. We then approximate the incident field with a finite number of modes. Thanks to the spectral perturbation theory, we obtain the modal expansion of the field scattered by an elastic quasiparticle in the frequency domain. The purpose of Section \ref{sec:4} is to give a polariton resonance expansion for the low-frequency part of the field scattered by  the elastic quasiparticle in the time domain. The paper is concluded in Section~\ref{sec:conclusion} with relevant remarks.

\section{Mathematical setup}\label{sec:2}

In this section, we focus on applying the potential theory to derive the integral representation of the solution
to the Lam\'{e} system \eqref{SS:lame}.

First, we introduce the potential operators for the following Lam\'{e} system
\begin{align*}
\mathcal{L}_{{\lambda},{\mu}}\mathbf{u}(\boldsymbol x)+\omega^2\mathbf{u}(\boldsymbol x)=0.
\end{align*}
The corresponding fundamental solution is the Kupradze matrix $\boldsymbol\Gamma^{\omega}=(\Gamma^{\omega}_{ij})^{3}_{i,j=1}$ with $\omega\neq0$ given by \begin{align}\label{Fundamental2}
\boldsymbol\Gamma^{\omega}(\boldsymbol x) =-\frac{e^{\mathrm{i}\frac{\omega}{c_s}|\boldsymbol x|}}{4\pi\mu|\boldsymbol x|}\boldsymbol{\mathcal{I}}+\frac{1}{4\pi\omega^2}\nabla\nabla\big(\frac{e^{\mathrm{i}\frac{\omega}{c_p}|\boldsymbol x|}
-e^{\mathrm{i}\frac{\omega}{c_s}|\boldsymbol x|}}{|\boldsymbol x|}\big),
\end{align}
where $\boldsymbol{\mathcal{I}}$ is the $3\times 3$ identity matrix. For $\omega=0$, the fundamental solution to the elastostatic system is the Kelvin matrix $\boldsymbol\Gamma^{0}=(\Gamma^{0}_{ij})^{3}_{i,j=1}$ defined by
\begin{align*}
 \boldsymbol\Gamma^{0}(\boldsymbol x)=-\frac{\gamma_1}{4\pi}\frac{1}{|\boldsymbol x|}\boldsymbol{\mathcal{I}}-\frac{\gamma_2}{4\pi}\frac{\boldsymbol x {\boldsymbol x}^\top}{|\boldsymbol x|^3},
\end{align*}
 where 
\begin{align*}
\gamma_1=\frac{1}{2}\big(\frac{1}{c^2_s}+\frac{1}{c^2_p}\big),\quad \gamma_2=\frac{1}{2}\big(\frac{1}{c^2_s}-\frac{1}{c^2_p}\big).
\end{align*}
Let $\boldsymbol\Gamma^{\omega}(\boldsymbol x,\boldsymbol y):=\boldsymbol\Gamma^{\omega}(\boldsymbol x-\boldsymbol y)$.
For $\boldsymbol\varphi\in L^{2}(\partial D)^3$, we define, respectively, the single layer potential associated with the fundamental solution $\boldsymbol\Gamma^{\omega}$ by
\begin{align*}
\mathbf{S}^{\omega}_{\partial D}[\boldsymbol\varphi](\boldsymbol x)=&\int_{\partial D}\boldsymbol\Gamma^{\omega}(\boldsymbol x,\boldsymbol y)\boldsymbol\varphi(\boldsymbol y)\mathrm{d}\sigma(\boldsymbol y),\quad \boldsymbol x\in\mathbb{R}^3,
\end{align*}
and the Neumann-Poincar\'{e} operator associated with the fundamental solution $\boldsymbol\Gamma^{\omega}$ by
\begin{align*}
\mathbf{K}^{\omega,*}_{\partial D}[\boldsymbol\varphi](\boldsymbol x)=&\mathrm{p.v.}\int_{\partial D} \frac{\partial}{\partial\boldsymbol{\nu}_{\boldsymbol x}}\boldsymbol\Gamma^{\omega}(\boldsymbol x,\boldsymbol y)\boldsymbol\varphi(\boldsymbol y)\mathrm{d}\sigma(\boldsymbol y),\quad \boldsymbol x\in\partial{D},
\end{align*}
where $\mathrm{p.v.}$ represents the Cauchy principle value. The conormal derivative of the single layer potential enjoys the jump formula
\begin{align}\label{Jump2}
\frac{\partial}{\partial\boldsymbol{\nu}}\mathbf{S}^{\omega}_{\partial D}[\boldsymbol\varphi]\big|_{\pm}(\boldsymbol x)=\big(\pm\frac{1}{2}\mathbf{I}+\mathbf{K}^{\omega,*}_{\partial D}\big)[\boldsymbol\varphi](\boldsymbol x),\quad \boldsymbol x\in\partial{D},
\end{align}
where $\mathbf{I}$ is the identity operator. Moreover,  the adjoint of the operator $\mathbf{K}^{\omega,*}_{\partial D}[\boldsymbol\varphi](\boldsymbol x)$ is defined by
\begin{align*}
\mathbf{K}^{\omega}_{\partial D}[\boldsymbol\varphi](\boldsymbol x)=\mathrm{p.v.}\int_{\partial D} \frac{\partial}{\partial\boldsymbol{\nu}_{\boldsymbol y}}\boldsymbol\Gamma^{\omega}(\boldsymbol x,\boldsymbol y)\boldsymbol\varphi(\boldsymbol y)\mathrm{d}\sigma(\boldsymbol y),\quad \boldsymbol x\in\partial D.
\end{align*}
For brevity,  we denote  the  elastostatic operators by  $ \mathbf{S}^{0}_{\partial D},\mathbf{K}^{0,*}_{\partial D},\mathbf{K}^{0}_{\partial D}$ by $\mathbf{S}_{\partial D},\mathbf{K}^{*}_{\partial D},\mathbf{K}_{\partial D}$.

The following lemma corresponds to a fundamental result associated with the Neumann-Poincar\'{e} operator as seen in \cite{chang2007spectral,Dahlbrg1988boundary,Irina1999spectral}.
\begin{lemm}\label{le:known}
One has

$\mathrm{(i)}$ $\mathbf{K}_{\partial D}$ and $\mathbf{K}^*_{\partial D}$ are bounded on $L^{2}(\partial D)^3$.

$\mathrm{(ii)}$ The spectrum of $\mathbf{K}^{*}_{\partial D}$ on $L^{2}(\partial D)^3$ lies in $(-1/2,1/2]$.

$\mathrm{(iii)}$  $\mathbf{S}_{\partial D}$ as an operator defined on $\partial D$ is bounded from $L^{2}(\partial D)^3$ into $L^{2}(\partial D)^3$.

$\mathrm{(iv)}$ $\mathbf{S}_{\partial D}:L^{2}(\partial D)^3\rightarrow L^{2}(\partial D)^3$ is invertible in three dimensions.
\end{lemm}

We emphasize that $\mathbf{K}^*_{\partial D}$ is not self-adjoint in general. In particular, $\mathbf{K}^*_{\partial D}$ is self-adjoint if $\partial D$ is a sphere. Moreover, we compute the following asymptotic expansions of the single layer potential and  the Neumann-Poincar\'{e} operator. The following lemma addresses  the series expansion for the fundamental solution.

\begin{lemm}\label{le:series}
If $\omega\ll1$,  then for $\boldsymbol x\in\mathbb{R}^3$,
\begin{align}\label{E-Fundamental2}
\boldsymbol\Gamma^{\omega}(\boldsymbol x)=&-\frac{1}{4\pi}\sum^{+\infty}_{j=0}\frac{\mathrm{i}^j}{(j+2)j!}\big(\frac{j+1}{c_s^{j+2}}+\frac{1}{c_p^{j+2}}\big)\omega^{j}|\boldsymbol x|^{j-1}\boldsymbol{\mathcal{I}}\nonumber\\
&+\frac{1}{4\pi}\sum^{+\infty}_{j=0}\frac{\mathrm{i}^j(j-1)}{(j+2)j!}\big(\frac{1}{c_s^{j+2}}-\frac{1}{c_p^{j+2}}\big)\omega^j|\boldsymbol x|^{j-3}\boldsymbol x\boldsymbol x^\top.
\end{align}
\end{lemm}
\begin{proof}
Formula \eqref{E-Fundamental2} follows from \eqref{Fundamental2} and Taylor series expansion with respect to $\omega$.
\end{proof}

In view of Lemma \ref{le:series}, we can obtain the following Lemmas \ref{le:S-series} and \ref{le:K-series}.
\begin{lemm}\label{le:S-series}
Let $\boldsymbol\varphi\in L^{2}(\partial D)^3$. For $\omega\ll1$, one has the following asymptotic expansion in three dimensions:
\begin{align}\label{ES-series}
\mathbf{S}^{\omega}_{\partial D}=\mathbf{S}_{\partial D}+\omega \mathbf{R}_{\partial D}+\omega^2\mathbf{P}_{\partial D}+\mathcal{O}(\omega^3),
\end{align}
where
\begin{align}\label{ri}
\mathbf{R}_{\partial D}[\boldsymbol\varphi](\boldsymbol x)=\gamma_3\int_{\partial D}\boldsymbol\varphi(\boldsymbol x)\mathrm{d}\sigma(\boldsymbol x),\quad
\mathbf{P}_{\partial D}[\boldsymbol\varphi](\boldsymbol x)=\int_{\partial D}\boldsymbol\Lambda(\boldsymbol x-\boldsymbol y)\boldsymbol\varphi(\boldsymbol y)\mathrm{d}\sigma(\boldsymbol y)
\end{align}
with
\begin{align}\label{E:lambda}
\gamma_3=\frac{-\mathrm{i}}{12\pi}\big(\frac{2}{c_s^{3}}+\frac{1}{c_p^{3}}\big),\quad\boldsymbol\Lambda(\boldsymbol x)=\frac{1}{16\pi}\gamma_4|\boldsymbol x|\boldsymbol{\mathcal{I}}-\frac{1}{16\pi}\gamma_5\frac{\boldsymbol x\boldsymbol x^\top}{|\boldsymbol x|},
\end{align}
where 
\begin{align*}
\gamma_4=\frac{1}{2}\big(\frac{3}{c_s^4}+\frac{1}{c^4_p}\big),\quad\gamma_5=\frac{1}{2}\big(\frac{1}{c^4_s}-\frac{1}{c^4_p}\big).
\end{align*}
\end{lemm}

\begin{lemm}\label{le:K-series}
Let $\boldsymbol\varphi\in L^{2}(\partial D)^3$. For $\omega\ll1$, one obtains the following asymptotic expansion in three dimensions:
\begin{align*}
\mathbf{K}^{\omega,*}_{\partial D}=\mathbf{K}^*_{\partial D}+\omega^2\mathbf{Q}_{\partial D}+\mathcal{O}(\omega^3)
\end{align*}
with
\begin{align*}
\mathbf{Q}_{\partial D}[\boldsymbol\varphi](\boldsymbol x)=\int_{\partial D}\frac{\partial}{\partial\boldsymbol{\nu}_{\boldsymbol x}}\boldsymbol\Lambda(\boldsymbol x-\boldsymbol y)\boldsymbol\varphi(\boldsymbol y)\mathrm{d}\sigma(\boldsymbol y),
\end{align*}
where $\boldsymbol\Lambda(\boldsymbol x)$ is given in \eqref{E:lambda}.
\end{lemm}

In addition, we show the invertibility of the three-dimensional single layer potential.

\begin{lemm}\label{le:IS-series}
If $\omega\ll1$, then the three-dimensional single layer potential $\mathbf{S}^{\omega}_{\partial D}: L^{2}(\partial D)^3\rightarrow L^{2}(\partial D)^3$ is invertible satisfying
\begin{align*}
&(\mathbf{S}^{\omega}_{\partial D})^{-1}=\mathbf{S}^{-1}_{\partial D}+\omega\mathbf R_{\partial D,1}+\omega^2\mathbf P_{\partial D,1}+\mathcal{O}(\omega^3)
\end{align*}
with
\begin{align}\label{ri1}
\mathbf{R}_{\partial D,1}=-\mathbf{S}_{\partial D}^{-1}\mathbf{R}_{\partial D}\mathbf{S}_{\partial D}^{-1},\quad \mathbf{P}_{\partial D,1}=-\mathbf{S}_{\partial D}^{-1}\mathbf{P}_{\partial D}\mathbf{S}_{\partial D}^{-1}+\mathbf{S}_{\partial D}^{-1}\mathbf{R}_{\partial D}\mathbf{S}_{\partial D}^{-1}\mathbf{R}_{\partial D}\mathbf{S}_{\partial D}^{-1},
\end{align}
where $\mathbf{R}_{\partial D}$ and $ \mathbf{P}_{\partial D}$ are given in \eqref{ri}.
\end{lemm}
\begin{proof}
It follows from \eqref{ES-series} and Lemma \ref{le:known}  that
\begin{align*}
\mathbf{S}^{\omega}_{\partial D}=\mathbf{S}_{\partial D}\left(\mathrm{\mathbf{I}}+\omega\mathbf{S}^{-1}_{\partial D}\mathbf{R}_{\partial D}+\omega^2\mathbf{S}^{-1}_{\partial D}\mathbf{P}_{\partial D}+\mathcal{O}(\omega^3)\right).
\end{align*}
If $\omega\ll1$, then the single layer potential $\mathbf{S}^{\omega}_{\partial D}: L^{2}(\partial D)^3\rightarrow L^{2}(\partial D)^3$ is invertible by Newmann series. Moreover,
\begin{align*}
(\mathbf{S}^{\omega}_{\partial D})^{-1}=&\left(\mathrm{\mathbf{I}}+\omega\mathbf{S}^{-1}_{\partial D}\mathbf{R}_{\partial D}+\omega^2\mathbf{S}^{-1}_{\partial D}\mathbf{P}_{\partial D}+\mathcal{O}(\omega^3)\right)^{-1}\mathbf{S}_{\partial D}^{-1}\\
=&\sum^{+\infty}_{j=0}\left(-\omega\mathbf{S}^{-1}_{\partial D}\mathbf{R}_{\partial D}-\omega^2\mathbf{S}^{-1}_{\partial D}\mathbf{P}_{\partial D}+\mathcal{O}(\omega^3)\right)^j\mathbf{S}_{\partial D}^{-1}\\
=&\mathbf{S}_{\partial D}^{-1}-\omega \mathbf{S}_{\partial D}^{-1}\mathbf{R}_{\partial D}\mathbf{S}_{\partial D}^{-1}+\omega^2(-\mathbf{S}_{\partial D}^{-1}\mathbf{P}_{\partial D}\mathbf{S}_{\partial D}^{-1}+\mathbf{S}_{\partial D}^{-1}\mathbf{R}_{\partial D}\mathbf{S}_{\partial D}^{-1}\mathbf{R}_{\partial D}\mathbf{S}_{\partial D}^{-1})+\mathcal{O}(\omega^3),
\end{align*}
which readily implies the conclusion of the lemma.

The proof is complete.
\end{proof}

With the above preparations, one has the following integral representation of the solution to the elastic system \eqref{SS:lame} as follows
\begin{align}\label{Solution}
\mathbf{u}(\boldsymbol x)=
\left\{ \begin{aligned}
&\mathbf{S}^{\omega_1}_{\partial D}[\boldsymbol{\varphi}](\boldsymbol x),\quad&&\boldsymbol x\in D,\\
&\mathbf{S}^{\omega_2}_{\partial D}[\boldsymbol{\psi}](\boldsymbol x)+\mathbf{u}^{\mathrm{in}}(\boldsymbol x),\quad && \boldsymbol x\in\mathbb{R}^3\backslash\overline{D}
\end{aligned}\right.
\end{align}
for $(\boldsymbol\varphi,\boldsymbol\psi)\in L^{2}(\partial D)^3\times L^{2}(\partial D)^3$,
where
\begin{align}
&\omega_1=\frac{\omega}{\sqrt{c(\omega)}},\label{Omega1}\\
&\omega_2=\omega.\label{Omega2}
\end{align}
By applying the Fourier transform as well as using the fundamental solution \eqref{Fundamental2}, we introduce an incident field to system \eqref{SS:lame} as
\begin{align}\label{E:incident}
\mathbf{u}^{\mathrm{in}}(\boldsymbol x)=f(\omega)\boldsymbol\Gamma^{\omega}(\boldsymbol x,\mathbf{s})\mathbf{p},\quad \boldsymbol x\in\mathbb{R}^3\setminus\overline{D},
\end{align}
which is the frequency-domain counterpart to the time-dependent source in \eqref{S:Lame1}. Moreover, we define
\begin{align*}
\mathbf{u}^{\mathrm{sca}}(\boldsymbol x):=\mathbf{u}(\boldsymbol x)-\mathbf{u}^{\mathrm{in}}(\boldsymbol x),\quad \boldsymbol x\in\mathbb{R}^3\setminus\overline{D}.
\end{align*}
By the transmission conditions in \eqref{SS:lame} across the boundary $\partial D$, the solution $\mathbf{u}(\boldsymbol x)$ in \eqref{Solution} should satisfy
\begin{align*}
\left\{ \begin{aligned}
&(\mathbf{S}^{\omega_2}_{\partial D}[\boldsymbol{\psi}](\boldsymbol x)+\mathbf{u}^{\mathrm{in}}(\boldsymbol x))|_+=\mathbf{S}^{\omega_1}_{\partial D}[\boldsymbol{\varphi}](\boldsymbol x)|_-,\quad&&\boldsymbol x\in \partial D,\\
&\frac{\partial (\mathbf{S}^{\omega_2}_{\partial D}[\boldsymbol{\psi}](\boldsymbol x)+\mathbf{u}^{\mathrm{in}}(\boldsymbol x))}{\partial \boldsymbol{\nu}}\big|_+=\frac{\partial \mathbf{S}^{\omega_1}_{\partial D}[\boldsymbol{\varphi}](\boldsymbol x)}{\partial \boldsymbol{\nu}}\big|_-,\quad&&\boldsymbol x\in\partial D.
\end{aligned}\right.
\end{align*}
With the help of the jump relation \eqref{Jump2}, the pair $(\boldsymbol{\varphi},\boldsymbol{\psi})$ is the solution to the following system of integral equations
\begin{align}\label{equa2}
\left\{ \begin{aligned}
&\mathbf{S}^{\omega_1}_{\partial D}[\boldsymbol\varphi](\boldsymbol x)-\mathbf{S}^{\omega_2}_{\partial D}[\boldsymbol\psi](\boldsymbol x)=\mathbf{F}_1,\quad&&\boldsymbol x\in\partial D,\\
&c(\omega)\big(-\frac{1}{2}\mathbf{I}+\mathbf{K}^{\omega_1,*}_{\partial D}\big)[\boldsymbol\varphi](\boldsymbol x)-\big(\frac{1}{2}\mathbf{I}+\mathbf{K}^{\omega_2,*}_{\partial D}\big)[\boldsymbol\psi]=\mathbf{F}_2,\quad&&\boldsymbol x\in\partial{D},
\end{aligned}\right.
\end{align}
where
\begin{align*}
\mathbf{F}_1={\mathbf u}^{\mathrm{in}}(\boldsymbol x),\quad\mathbf{F}_2=\frac{\partial}{\partial\boldsymbol\nu}\mathbf{u}^{\mathrm{in}}(\boldsymbol x).
\end{align*}

To see the scaling clearly, we introduce the transform $\boldsymbol x=\boldsymbol z+\delta \boldsymbol X$, where $\boldsymbol x\in\partial D,\boldsymbol X\in\partial B$. Then two integral equations in \eqref{equa2} will be studied in $L^{2}(\partial B)^3$ instead of $L^{2}(\partial D)^3$. For each function $\boldsymbol\varphi$ defined on $\partial D$, we define a corresponding function on $\partial B$ by $\tilde{\boldsymbol\varphi}(\boldsymbol X)=\boldsymbol\varphi(\boldsymbol z+\delta \boldsymbol X)$.

\begin{lemm}\label{le:scale}
One has 
\begin{align*}
&\mathbf{S}^{\omega}_{\partial D}[\boldsymbol\varphi](\boldsymbol x)=\delta \mathbf{S}^{\omega\delta}_{\partial B}[\tilde{\boldsymbol\varphi}](\boldsymbol X),\quad
\mathbf{K}^{\omega,*}_{\partial D}[\boldsymbol\varphi](\boldsymbol x)=\mathbf{K}^{\omega\delta,*}_{\partial B}[\tilde{\boldsymbol\varphi}](\boldsymbol X).
\end{align*}
\end{lemm}
\begin{lemm}\label{le:scalar}
For $\mathbf{f},\mathbf{g}$ defined on $\partial D$, corresponding to $\tilde{\mathbf{f}},\tilde{\mathbf{g}}$ defined on $\partial B$, respectively,  one has
\begin{align*}
&\langle \mathbf{f},\mathbf{g}\rangle_{{L^{2}(\partial D)^3}}=\delta^2\langle \tilde{\mathbf{f}},\tilde{\mathbf{g}}\rangle_{{L^{2}(\partial B)^3}},\\
&\|\mathbf{f}\|_{{L^{2}(\partial D)^3}}=\delta\|\tilde{\mathbf{f}}\|_{{L^{2}(\partial B)^3}}.
\end{align*}
\end{lemm}
\begin{proof}
It is straightforward to verify that
\begin{align*}
\langle \mathbf{f},\mathbf{g}\rangle_{{L^{2}(\partial D)^3}}=&\int_{\partial D}\mathbf{f}(\boldsymbol x)\cdot\mathbf{g}(\boldsymbol x)\mathrm{d}\sigma(\boldsymbol x)\\
=&\delta^2\int_{\partial B}\mathbf{f}(\boldsymbol z+\delta\boldsymbol X)\cdot\mathbf{g}(\boldsymbol z+\delta\boldsymbol X)\mathrm{d}\sigma(\boldsymbol X)=\delta^2\langle \tilde{\mathbf{f}},\tilde{\mathbf{g}}\rangle_{{L^{2}(\partial B)^3}},
\end{align*}
which then gives $\|\mathbf{f}\|_{{L^{2}(\partial D)^3}}=\delta\|\tilde{\mathbf{f}}\|_{{L^{2}(\partial B)^3}}$ and readily completes the proof.
\end{proof}

In view of Lemma \ref{le:scale},  the solution $\mathbf{u}$ in \eqref{Solution} can be rewritten as
\begin{align}\label{Solution2}
\tilde{\mathbf{u}}(\boldsymbol X)=
\left\{ \begin{aligned}
&\delta \mathbf{S}^{\omega_1\delta}_{\partial B}[\tilde{\boldsymbol{\varphi}}](\boldsymbol X),\quad&&\boldsymbol X\in B,\\
&\delta\mathbf{S}^{\omega_2\delta}_{\partial B}[\tilde{\boldsymbol{\psi}}](\boldsymbol X)+\tilde{\mathbf{u}}^{\mathrm{in}}(\boldsymbol X),\quad &&\boldsymbol X\in\mathbb{R}^3\backslash\overline{B},
\end{aligned}\right.
\end{align}
where $\tilde{\mathbf{u}}^{\mathrm{in}}(\boldsymbol X):=\mathbf{u}^{\mathrm{in}}(\boldsymbol z+\delta \boldsymbol X)$.  The density $(\tilde{\boldsymbol{\varphi}},\tilde{\boldsymbol{\psi}})$ is the unique solution to
\begin{align}\label{Equivalent}
\left\{ \begin{aligned}
&\mathbf{S}^{\omega_1\delta}_{\partial B}[\tilde{\boldsymbol{\varphi}}](\boldsymbol X)-\mathbf{S}^{\omega_2\delta}_{\partial B}[\tilde{\boldsymbol{\psi}}](\boldsymbol X)=\frac{1}{\delta}\tilde{\mathbf{F}}_1(\boldsymbol X),\quad &\boldsymbol X\in\partial{B},\\
&c(\omega)\big(-\frac{1}{2}\mathbf{I}+\mathbf{K}^{\omega_1\delta,*}_{\partial B}\big)[\tilde{\boldsymbol\varphi}](\boldsymbol X)-\big(\frac{1}{2}\mathbf{I}+\mathbf{K}^{\omega_2\delta,*}_{\partial B}\big)[\tilde{\boldsymbol\psi}]=\frac{1}{\delta}\tilde{\mathbf{F}}_2(\boldsymbol X),\quad&\boldsymbol X\in\partial{B},
\end{aligned}\right.
\end{align}
where
\begin{align*}
\tilde{\mathbf{F}}_1(\boldsymbol X)=\mathbf{u}^{\mathrm{in}}(\boldsymbol z+\delta \boldsymbol X),\quad\tilde{\mathbf{F}}_2(\boldsymbol X)=\frac{\partial}{\partial\boldsymbol\nu_{\boldsymbol X}}\mathbf{u}^{\mathrm{in}}(\boldsymbol z+\delta \boldsymbol X).
\end{align*}

From the invertibility of the operator $\mathbf{S}^{\omega_1\delta}_{\partial B}$ in Lemma \ref{le:IS-series}, system \eqref{Equivalent} is equivalent to
\begin{align}\label{AO}
\boldsymbol{\mathcal{A}}^{\omega\delta}_{\partial B}[\tilde{\boldsymbol{\psi}}](\boldsymbol X)=\frac{1}{\delta}\tilde{\mathbf{F}}^{\omega\delta}(\boldsymbol X),\quad \boldsymbol X\in\partial{B},
\end{align}
where
\begin{align}
&\boldsymbol{\mathcal{A}}^{\omega\delta}_{\partial B}=c(\omega)\big(-\frac{1}{2}\mathbf{I}+\mathbf{K}^{\omega_1\delta,*}_{\partial B}\big)(\mathbf{S}^{\omega_1\delta}_{\partial B})^{-1}\mathbf{S}^{\omega_2\delta}_{\partial B}-\big(\frac{1}{2}\mathbf{I}+\mathbf{K}^{\omega_2\delta,*}_{\partial B}\big),\label{AOD}\\
&\tilde{\mathbf{F}}^{\omega\delta}(\boldsymbol X)=\tilde{\mathbf{F}}_2(\boldsymbol X)-c(\omega)\big(-\frac{1}{2}\mathbf{I}+\mathbf{K}^{\omega_1\delta,*}_{\partial B}\big)(\mathbf{S}^{\omega_1\delta}_{\partial B})^{-1}\tilde{\mathbf{F}}_1(\boldsymbol X).\label{tildeF}
\end{align}
In order to give the modal decomposition of the scattered field for the Lam\'{e} system \eqref{SS:lame}, we should study the singularly perturbed system of integral equations equivalently.


\section{Modal decomposition of the field}\label{sec:3}

In this section, we aim at establishing the modal decomposition of the scattered field for the Lam\'{e} system in the frequency domain. In the boundary integral equation \eqref{AO}, since the operator $\boldsymbol{\mathcal{A}}^{\omega\delta}_{\partial B}$ is not self-adjoint in $L^{2}(\partial B)^3$,  it can not be diagonalised directly to solve equation \eqref{AO}. However, in the static regime, the operator $\boldsymbol{\mathcal{A}}^{0}_{\partial B}(\omega\delta=0)$ can be expressed simply with $\mathbf{K}^{*}_{\partial B}$ which is self-adjoint when $\partial B$ is a unit sphere. We show the eigenvalues and the  asymptotic properties of the Neumann-Poincar\'{e} operator associated with  the elastostatic system, and then
give the corresponding  modal expansion. Our next goal is to approximate the incident field with a finite number of modes. Finally,  we obtain the finite modal expansion of the scattered field  in the linear  elasticity. The proof is based on the classic perturbation
theory.


\subsection{Scattered field for the elastostatic system}
This subsection is devoted to establishing  the modal decomposition of the scattered field for the elastostatic system in the frequency domain.

For $\omega\delta=0$, one has $\omega_1\delta=\omega_2\delta=0$.  We consider the following  elastostatic system
\begin{align}\label{E:static}
\left\{ \begin{aligned}
&\mathbf{S}_{\partial B}[\tilde{\boldsymbol{\varphi}}](\boldsymbol X)-\mathbf{S}_{\partial B}[\tilde{\boldsymbol{\psi}}](\boldsymbol X)=\frac{1}{\delta}\tilde{\mathbf{F}}_1(\boldsymbol X),\quad &&\boldsymbol X\in\partial{B},\\
&c(\omega)\big(-\frac{1}{2}\mathbf{I}+\mathbf{K}^{*}_{\partial B}\big)[\tilde{\boldsymbol\varphi}](\boldsymbol X)-\big(\frac{1}{2}\mathbf{I}+\mathbf{K}^{*}_{\partial B}\big)[\tilde{\boldsymbol\psi}]=\frac{1}{\delta}\tilde{\mathbf{F}}_2(\boldsymbol X),\quad&&\boldsymbol X\in\partial{B}.
\end{aligned}\right.
 \end{align}
Using the invertibility of the operator $\mathbf{S}_{\partial B}$ given by Lemma \ref{le:known} yields that
\begin{align}\label{EO:static}
\boldsymbol{\mathcal{A}}^0_{\partial B}[\tilde{\boldsymbol{\psi}}](\boldsymbol X)=\frac{1}{\delta}\tilde{\mathbf{F}}(\boldsymbol X),\quad\boldsymbol X\in\partial{B},
\end{align}
where
\begin{align}
&\boldsymbol{\mathcal{A}}^0_{\partial B}=-\big(\frac{c(\omega)}{2}+\frac{1}{2}\big)\mathbf{I}+(c(\omega)-1)\mathbf{K}^*_{\partial B},\label{A-Zero}\\
&\tilde{\mathbf{F}}(\boldsymbol X)=\tilde{\mathbf{F}}_2(\boldsymbol X)-c(\omega)\big(-\frac{1}{2}\mathbf{I}+\mathbf{K}^*_{\partial B}\big)\mathbf{S}^{-1}_{\partial B}\tilde{\mathbf{F}}_1(\boldsymbol X).\label{F-Tilde}
\end{align}

In the following,  we shall define orthogonal vectorial polynomials which will be quite important in the analysis of the spectral of the Neumann-Poincar\'{e} operator $\mathbf{K}^*_{\partial B}$.  Moreover, to avoid confusion, we introduce the notation $\hat{\boldsymbol X}\in\partial B$ instead of $\boldsymbol X\in\partial B $. Let $r=|\boldsymbol X|$. Denote by $\tilde{Y}^m_n(\hat{\boldsymbol X}),-n\leq m\leq n$  spherical harmonics on the unit sphere $\partial B$. Define three vectorial polynomials as
\begin{align*}
&\tilde{{\mathfrak{T}}}^{m}_n(\boldsymbol X)=\nabla(r^n\tilde{Y}^m_n(\hat{\boldsymbol X}))\times\boldsymbol X, &&n\geq1,-n\leq m\leq n,
\\&\tilde{{\mathfrak{M}}}^{m}_n(\boldsymbol X)=\nabla(r^n\tilde{Y}^m_n(\hat{\boldsymbol X})), &&n\geq1,-n\leq m\leq n,
\end{align*}
and
\begin{align*}
\tilde{{\mathfrak{N}}}^{m}_{n}(\boldsymbol X)=a^m_nr^{n-1}\tilde{Y}^m_{n-1}(\hat{\boldsymbol X})\boldsymbol X+\big(1-\frac{a^m_n}{2n-1}-r^2\big)\nabla(r^{n-1}\tilde{Y}^m_{n-1}(\hat{\boldsymbol X})),
\end{align*}
where
\begin{align*}
a^m_n=\frac{2(n-1)\lambda+2(3n-2)\mu}{(n+2)\lambda+(n+4)\mu},\quad n\geq1,-(n-1)\leq m\leq n-1.
\end{align*}
Denote by $\tilde{\mathbf{T}}^m_n,\tilde{\mathbf{M}}^m_n$ and $\tilde{\mathbf{N}}^m_n$ the traces of $\tilde{{\mathfrak{T}}}^{m}_n,\tilde{{\mathfrak{M}}}^{m}_n$ and $\tilde{{\mathfrak{N}}}^{m}_{n}$ on the unit sphere $\partial B$, respectively. By directly applying the trace theorem, these traces have the form
\begin{align}
&\tilde{\mathbf{T}}^m_n(\boldsymbol X)=\nabla_{S}\tilde{Y}^m_n(\hat{\boldsymbol X})\times\boldsymbol{\nu}_{\boldsymbol X},\label{BJMN}\\
&\tilde{\mathbf{M}}^m_n(\boldsymbol X)=\nabla_{S}\tilde{Y}^m_n(\hat{\boldsymbol X})+n\tilde{Y}^m_n(\hat{\boldsymbol X})\boldsymbol{\nu}_{\boldsymbol X},\label{BMMN}\\
&\tilde{\mathbf{N}}^m_n(\boldsymbol X)=\frac{a^m_n}{2n-1}\big(-\nabla_{S}\tilde{Y}^m_{n-1}(\hat{\boldsymbol X})+n\tilde{Y}^m_{n-1}(\hat{\boldsymbol X}){\boldsymbol\nu}_{\boldsymbol X}\big).\label{BNMN}
\end{align}

The following lemma and theorem correspond to some facts to be used in the sequel, proofs of which can be found in \cite{Jean2001acoustic,Deng2019spectral}.
\begin{lemm}\label{le:solution}
The trace
$(\tilde{\mathbf{T}}^m_n,\tilde{\mathbf{M}}^m_n,\tilde{\mathbf{N}}^m_n)$ defined in \eqref{BJMN}--\eqref{BNMN} forms an  orthogonal basis on $L^2(\partial B)^3$.
\end{lemm}
From now on, without loss of generality, we assume that $(\tilde{\mathbf{T}}^m_n,\tilde{\mathbf{M}}^m_n,\tilde{\mathbf{N}}^m_n)$ is the normalized orthogonal basis on $L^2(\partial B)^3$.
\begin{theo}\label{spectrum}
If the domain $B$ is a central ball of radius $1$, then the eigenvalues of the operator $\mathbf{K}^*_{\partial B}$ are given by
\begin{align*}
&\lambda_{\mathbf{T},n}=\frac{3}{2(2n+1)},
\\&\lambda_{\mathbf{M},n}=\frac{3\lambda-2\mu(2n^2-2n-3)}{2(\lambda+2\mu)(4n^2-1)},
\\&\lambda_{\mathbf{N},n}=\frac{-3\lambda+2\mu(2n^2+2n-3)}{2(\lambda+2\mu)(4n^2-1)},
\end{align*}
where $n\geq1$ are nature numbers, and the corresponding eigenfunctions are $\tilde{\mathbf{T}}^m_n,\tilde{\mathbf{M}}^m_n,\tilde{\mathbf{N}}^m_n$, respectively.
\end{theo}

Owing to Lemma \ref{le:solution}, the function $\tilde{\mathbf{F}}\in L^{2}(\partial B)^3$ (see \eqref{F-Tilde}) can be expanded into
\begin{align}\label{F:expansion}
\tilde{\mathbf{F}}=&\sum_{n=1}^{\infty}\sum_{m=-n}^{n}\langle\tilde{\mathbf{F}},\tilde{\mathbf{T}}^{m}_{n}\rangle_{L^{2}(\partial B)^3}\tilde{\mathbf{T}}^{m}_{n}+\sum_{n=1}^{\infty}\sum_{m=-n}^{n}\langle\tilde{\mathbf{F}},\tilde{\mathbf{M}}^{m}_{n}\rangle_{L^{2}(\partial B)^3}\tilde{\mathbf{M}}^{m}_{n}\nonumber\\
&+\sum_{n=1}^{\infty}\sum_{m=-(n-1)}^{n-1}\langle\tilde{\mathbf{F}},\tilde{\mathbf{N}}^{m}_{n}\rangle_{L^{2}(\partial B)^3}\tilde{\mathbf{N}}^{m}_{n}.
\end{align}
Define
\begin{align*}
&L^{2}_{\mathbf{T}}(\partial B)^3:=\left\{\tilde{\boldsymbol{\psi}}\in L^2(\partial B)^3:\langle \tilde{\boldsymbol{\psi}},\tilde{\mathbf{M}}^m_n\rangle_{L^2(\partial B)^3}=0,\langle\tilde{\boldsymbol{\psi}},\tilde{\mathbf{N}}^m_n\rangle_{L^2(\partial B)^3}=0\right\},\\
&L^{2}_{\mathbf{M}}(\partial B)^3:=\left\{\tilde{\boldsymbol{\psi}}\in L^2(\partial B)^3:\langle \tilde{\boldsymbol{\psi}},\tilde{\mathbf{T}}^m_n\rangle_{L^2(\partial B)^3}=0,\langle\tilde{\boldsymbol{\psi}},\tilde{\mathbf{N}}^m_n\rangle_{L^2(\partial B)^3}=0\right\},\\
&L^{2}_{\mathbf{N}}(\partial B)^3:=\left\{\tilde{\boldsymbol{\psi}}\in L^2(\partial B)^3:\langle \tilde{\boldsymbol{\psi}},\tilde{\mathbf{T}}^m_n\rangle_{L^2(\partial B)^3}=0,\langle\tilde{\boldsymbol{\psi}},\tilde{\mathbf{M}}^m_n\rangle_{L^2(\partial B)^3}=0\right\}.
\end{align*}
In addition, we decompose
\begin{align*}
\tilde{\boldsymbol{\psi}}=\tilde{\boldsymbol{\psi}}_{\mathbf{T}}+\tilde{\boldsymbol{\psi}}_{\mathbf{M}}+\tilde{\boldsymbol{\psi}}_{\mathbf{N}},\quad \tilde{\mathbf{F}}=\tilde{\mathbf{F}}_{\mathbf{T}}+\tilde{\mathbf{F}}_{\mathbf{M}}+\tilde{\mathbf{F}}_{\mathbf{N}},
\end{align*}
where
$\tilde{\boldsymbol{\psi}}_{\mathbf{T}},\tilde{\mathbf{F}}_{\mathbf{T}}\in L^{2}_{\mathbf{T}}(\partial B)^3,\tilde{\boldsymbol{\psi}}_{\mathbf{M}},\tilde{\mathbf{F}}_{\mathbf{M}}\in L^{2}_{\mathbf{M}}(\partial B)^3$ and $\tilde{\boldsymbol{\psi}}_{\mathbf{N}},\tilde{\mathbf{F}}_{\mathbf{N}}\in L^{2}_{\mathbf{N}}(\partial B)^3$.
Then
\begin{align*}
&\tilde{\mathbf{F}}_{\mathbf{T}}=\sum_{n=1}^{\infty}\sum_{m=-n}^{n}\langle\tilde{\mathbf{F}},\tilde{\mathbf{T}}^{m}_{n}\rangle_{L^{2}(\partial B)^3}\tilde{\mathbf{T}}^{m}_{n},\\
&\tilde{\mathbf{F}}_{\mathbf{M}}=\sum_{n=1}^{\infty}\sum_{m=-n}^{n}\langle\tilde{\mathbf{F}},\tilde{\mathbf{M}}^{m}_{n}\rangle_{L^{2}(\partial B)^3}\tilde{\mathbf{M}}^{m}_{n},\\
&\tilde{\mathbf{F}}_{\mathbf{N}}=\sum_{n=1}^{\infty}\sum_{m=-(n-1)}^{n-1}\langle\tilde{\mathbf{F}},\tilde{\mathbf{N}}^{m}_{n}\rangle_{L^{2}(\partial B)^3}\tilde{\mathbf{N}}^{m}_{n}.
\end{align*}
Denote by $\Pi_{i}$ the projection operators as follows
\begin{align*}
\Pi_{i}:L^{2}(\partial B)^3\longrightarrow L^{2}_{i}(\partial B)^3,\quad i=\mathbf{T},\mathbf{M},\mathbf{N}.
\end{align*}
As a result, equation \eqref{EO:static} is equivalent to the following system:
\begin{align}
&\Pi_{\mathbf{T}}\boldsymbol{\mathcal{A}}^0_{\partial B}[\tilde{\boldsymbol{\psi}}_{\mathbf{T}}+\tilde{\boldsymbol{\psi}}_{\mathbf{M}}+\tilde{\boldsymbol{\psi}}_{\mathbf{N}}](\boldsymbol X)=\frac{1}{\delta}\sum_{n=1}^{\infty}\sum_{m=-n}^{n}\langle\tilde{\mathbf{F}},\tilde{\mathbf{T}}^{m}_{n}\rangle_{L^{2}(\partial B)^3}\tilde{\mathbf{T}}^{m}_{n}(\boldsymbol X),\label{E:T}\\
&\Pi_{\mathbf{M}}\boldsymbol{\mathcal{A}}^0_{\partial B}[\tilde{\boldsymbol{\psi}}_{\mathbf{T}}+\tilde{\boldsymbol{\psi}}_{\mathbf{M}}+\tilde{\boldsymbol{\psi}}_{\mathbf{N}}](\boldsymbol X)=\frac{1}{\delta}\sum_{n=1}^{\infty}\sum_{m=-n}^{n}\langle\tilde{\mathbf{F}},\tilde{\mathbf{M}}^{m}_{n}\rangle_{L^{2}(\partial B)^3}\tilde{\mathbf{M}}^{m}_{n}(\boldsymbol X),\label{E:M}\\
&\Pi_{\mathbf{N}}\boldsymbol{\mathcal{A}}^0_{\partial B}[\tilde{\boldsymbol{\psi}}_{\mathbf{T}}+\tilde{\boldsymbol{\psi}}_{\mathbf{M}}+\tilde{\boldsymbol{\psi}}_{\mathbf{N}}](\boldsymbol X)=\frac{1}{\delta}\sum_{n=1}^{\infty}\sum_{m=-(n-1)}^{n-1}\langle\tilde{\mathbf{F}},\tilde{\mathbf{N}}^{m}_{n}\rangle_{L^{2}(\partial B)^3}\tilde{\mathbf{N}}^{m}_{n}(\boldsymbol X).\label{E:N}
\end{align}

\begin{prop}\label{prop:tau}
The operator  $\boldsymbol{\mathcal{A}}^0_{\partial B}$ has the spectral decomposition:
\begin{align*}
&\boldsymbol{\mathcal{A}}^0_{\partial B}[\tilde{\boldsymbol{\psi}}_{\mathbf{T}}]=\sum_{n=1}^{\infty}\sum_{m=-n}^{n}\tau_{\mathbf{T},n}\langle\tilde{\boldsymbol{\psi}}_{\mathbf{T}},\tilde{\mathbf{T}}^{m}_{n}\rangle_{L^{2}(\partial B)^3}\tilde{\mathbf{T}}^{m}_{n},\\
&\boldsymbol{\mathcal{A}}^0_{\partial B}[\tilde{\boldsymbol{\psi}}_{\mathbf{M}}]=\sum_{n=1}^{\infty}\sum_{m=-n}^{n}\tau_{\mathbf{M},n}\langle\tilde{\boldsymbol{\psi}}_{\mathbf{M}},\tilde{\mathbf{M}}^{m}_{n}\rangle_{L^{2}(\partial B)^3}\tilde{\mathbf{M}}^{m}_{n},\\
&\boldsymbol{\mathcal{A}}^0_{\partial B}[\tilde{\boldsymbol{\psi}}_{\mathbf{N}}]=\sum_{n=1}^{\infty}\sum_{m=-(n-1)}^{n-1}\tau_{\mathbf{N},n}\langle\tilde{\boldsymbol{\psi}}_{\mathbf{N}},\tilde{\mathbf{N}}^{m}_{n}\rangle_{L^{2}(\partial B)^3}\tilde{\mathbf{N}}^{m}_{n},
\end{align*}
where
\begin{align}\label{tau1}
\tau_{i,n}=-\big(\frac{c(\omega)}{2}+\frac{1}{2}\big)+(c(\omega)-1)\lambda_{i,n},\quad i=\mathbf{T},\mathbf{M},\mathbf{N}.
\end{align}
\end{prop}
\begin{proof}
The proof of the proposition is a direct result of Theorem \ref{spectrum}.
\end{proof}

The following corollary corresponds to the modal decomposition of the scattered field for the elastostatic system.
\begin{coro}
Provided that the incident wave $\tilde{\mathbf{u}}^{\mathrm{in}}$ satisfies \eqref{F:expansion}, the spectral approximation of the solution for the elastostatic system \eqref{E:static} $(\omega\delta=0)$ satisfies that for $\boldsymbol X\in\mathbb{R}^3\backslash \overline{B}$,
\begin{align}\label{F-Static}
\tilde{\mathbf{u}}^{\mathrm{sca}}(\boldsymbol X)
=&\sum_{n=1}^{\infty}\sum_{m=-n}^{n}\frac{1}{\tau_{\mathbf{T},n}}\langle\tilde{\mathbf{F}},\tilde{\mathbf{T}}^{m}_{n}\rangle_{L^{2}(\partial B)^3}\mathbf{S}_{\partial B}[\tilde{\mathbf{T}}^{m}_{n}](\boldsymbol X)\nonumber\\
&+\sum_{n=1}^{\infty}\sum_{m=-n}^{n}\frac{1}{\tau_{\mathbf{M},n}}\langle\tilde{\mathbf{F}},\tilde{\mathbf{M}}^{m}_{n}\rangle_{L^{2}(\partial B)^3}\mathbf{S}_{\partial B}[\tilde{\mathbf{M}}^{m}_{n}](\boldsymbol X)\nonumber\\
&+\sum_{n=1}^{\infty}\sum_{m=-(n-1)}^{n-1}\frac{1}{\tau_{\mathbf{N},n}}\langle\tilde{\mathbf{F}},\tilde{\mathbf{N}}^{m}_{n}\rangle_{L^{2}(\partial B)^3}\mathbf{S}_{\partial B}[\tilde{\mathbf{N}}^{m}_{n}](\boldsymbol X),
\end{align}
where $\tilde{\mathbf{u}}^{\mathrm{sca}}(\boldsymbol X):=\mathbf{u}^{\mathrm{sca}}(\boldsymbol z+\delta\boldsymbol X)$ and $\tau_{i,n},i=\mathbf{T},\mathbf{M},\mathbf{N}$ are defined in  \eqref{tau1}.
\end{coro}
\begin{proof}
Thanks to \eqref{E:T}--\eqref{E:N} and Proposition \ref{prop:tau}, one can show that
\begin{align*}
&\sum_{n=1}^{\infty}\sum_{m=-n}^{n}\tau_{\mathbf{T},n}\langle\tilde{\boldsymbol{\psi}}_{\mathbf{T}},\tilde{\mathbf{T}}^{m}_{n}\rangle_{L^{2}(\partial B)^3}\tilde{\mathbf{T}}^{m}_{n}=\frac{1}{\delta} \sum_{n=1}^{\infty}\sum_{m=-n}^{n}\langle\tilde{\mathbf{F}},\tilde{\mathbf{T}}^{m}_{n}\rangle_{L^{2}(\partial B)^3}\tilde{\mathbf{T}}^{m}_{n},\\
&\sum_{n=1}^{\infty}\sum_{m=-n}^{n}\tau_{\mathbf{M},n}\langle\tilde{\boldsymbol{\psi}}_{\mathbf{M}},\tilde{\mathbf{M}}^{m}_{n}\rangle_{L^{2}(\partial B)^3}\tilde{\mathbf{M}}^{m}_{n}=\frac{1}{\delta} \sum_{n=1}^{\infty}\sum_{m=-n}^{n}\langle\tilde{\mathbf{F}},\tilde{\mathbf{M}}^{m}_{n}\rangle_{L^{2}(\partial B)^3}\tilde{\mathbf{M}}^{m}_{n},\\
&\sum_{n=1}^{\infty}\sum_{m=-(n-1)}^{n-1}\tau_{\mathbf{N},n}\langle\tilde{\boldsymbol{\psi}}_{\mathbf{N}},\tilde{\mathbf{N}}^{m}_{n}\rangle_{L^{2}(\partial B)^3}\tilde{\mathbf{N}}^{m}_{n}=\frac{1}{\delta} \sum_{n=1}^{\infty}\sum_{m=-(n-1)}^{n-1}\langle\tilde{\mathbf{F}},\tilde{\mathbf{N}}^{m}_{n}\rangle_{L^{2}(\partial B)^3}\tilde{\mathbf{N}}^{m}_{n},
\end{align*}
which then gives
\begin{align*}
&\langle\tilde{\boldsymbol{\psi}}_{\mathbf{T}},\tilde{\mathbf{T}}^{m}_{n}\rangle_{L^{2}(\partial B)^3}=
\frac{1}{\delta\tau_{\mathbf{T},n}}\langle\tilde{\mathbf{F}},\tilde{\mathbf{T}}^{m}_{n}\rangle_{L^{2}(\partial B)^3},\\
&\langle\tilde{\boldsymbol{\psi}}_{\mathbf{M}},\tilde{\mathbf{M}}^{m}_{n}\rangle_{L^{2}(\partial B)^3}=\frac{1}{\delta\tau_{\mathbf{M},n}}\langle\tilde{\mathbf{F}},\tilde{\mathbf{M}}^{m}_{n}\rangle_{L^{2}(\partial B)^3},\\
&\langle\tilde{\boldsymbol{\psi}}_{\mathbf{N}},\tilde{\mathbf{N}}^{m}_{n}\rangle_{L^{2}(\partial B)^3}=\frac{1}{\delta\tau_{\mathbf{N},n}}\langle\tilde{\mathbf{F}},\tilde{\mathbf{N}}^{m}_{n}\rangle_{L^{2}(\partial B)^3}.
\end{align*}
Hence it follows from the second expression in \eqref{Solution2} that \eqref{F-Static} holds. This ends the proof of the lemma.
\end{proof}


\subsection{Scattered field for the Lam\'{e} system}

In the previous subsection, we have studied the limiting problem \eqref{EO:static} of \eqref{AO}. This subsection focuses on applying perturbation theory tools to express solutions to \eqref{AO} by the eigenvectors of $\mathbf{K}^*_{\partial B}$ which appear in the spectral decomposition of \eqref{EO:static}. Moreover, we need to work with a finite number of modes in order to obtain a meaningful expansion  of the scattered field. 
As in classical Fourier analysis, the decay with $n,m$ of Fourier coefficients  will be determined by the regularity of the function
and the number of modes to consider will depend on the spatial variations of $\tilde{\mathbf{F}}^{\omega\delta}$ over $\partial B$. If  the incoming field $\tilde {\mathbf{F}}^{\omega\delta}$ in an homogeneous medium is smooth, then we can see a fast decay of Fourier coefficients. In view of perturbation theory, we can  replace $\tau_{i,n}$ by a perturbed value $\tau_{i,n}(\omega\delta)$ and then validate an expansion with a finite number of modes in the perturbative regime.
It has been empirically reported that only a few modes actually are conductive to the scattered field. The number of modes  to consider increasing as the source gets closer to the particle. By the above analysis, we provide a mathematical explanation of this phenomenon.


The following lemma corresponds to the decay estimate of the eigenvalues of  $\mathbf{K}^*_{\partial B}$.

\begin{lemm}\label{coro:spectrum}
Let $\kappa_0:=\frac{\mu}{2(\lambda+2\mu)}.$
If the domain $B$ is a central ball of radius $1$, then the following facts hold:

$\mathrm{(i)}$ The eigenvalues $\lambda_{\mathbf{T},n}$ of the operator $\mathbf{K}^*_{\partial B}$ associated with the eigenfunctions $\tilde{\mathbf{T}}^m_n$ satisfy
\begin{align*}
&\lambda_{\mathbf{T},n}>0, \quad\forall n\geq1,\\
&\lambda_{\mathbf{T},n}\sim\frac{3}{4n} \quad\text {as $n\rightarrow\infty$}.
\end{align*}

$\mathrm{(ii)}$ The eigenvalues $\lambda_{\mathbf{M},n}+\kappa_0$ of the operator $\mathbf{K}^*_{\partial B}+\kappa_0 \mathbf{I}$ associated with the eigenfunctions $\tilde{\mathbf{M}}^m_n$ satisfy
\begin{align*}
&\lambda_{\mathbf{M},n}+\kappa_0>0,\quad\forall n\geq1,\\
&\lambda_{\mathbf{M},n}+\kappa_0\sim\frac{\mu}{2(\lambda+2\mu)n} \quad\text {as $n\rightarrow\infty$}.
\end{align*}

$\mathrm{(iii)}$ The eigenvalues $\lambda_{\mathbf{N},n}-\kappa_0$ of the operator $\mathbf{K}^*_{\partial B}-\kappa_0 \mathbf{I}$ associated with the eigenfunctions $\tilde{\mathbf{N}}^m_n$ satisfy
\begin{align*}
&\lambda_{\mathbf{N},n}-\kappa_0>0,\quad \forall  n\geq\lfloor\frac{3\lambda+5\mu}{4\mu}\rfloor+1,\\
&\lambda_{\mathbf{N},n}-\kappa_0\sim\frac{\mu}{2(\lambda+2\mu)n} \quad\text {as $n\rightarrow\infty$},
\end{align*}
where the symbol $\lfloor\cdot\rfloor$ stands for the integer part.
\end{lemm}
\begin{proof}
By means of \eqref{convexity} and Theorem \ref{spectrum}, one has
\begin{align*}
\lambda_{\mathbf{M},n}+\kappa_0=\frac{3\lambda+5\mu+4\mu n}{2(\lambda+2\mu)(4n^2-1)}>0, \quad&\forall n\geq1,\\
\lambda_{\mathbf{N},n}+\kappa_0=\frac{-(3\lambda+5\mu)+4\mu n}{2(\lambda+2\mu)(4n^2-1)}>0,\quad&\forall n\geq\lfloor\frac{3\lambda+5\mu}{4\mu}\rfloor+1.
\end{align*}
Moreover, from Theorem \ref{spectrum},  it follows that
\begin{align*}
&\lim_{n\rightarrow\infty}\frac{\lambda_{\mathbf{T},n}}{\frac{3}{4n} }=\lim_{n\rightarrow\infty}\frac{\frac{3}{2(2n+1)}}{\frac{3}{4n} }=\lim_{n\rightarrow\infty}\frac{4n}{4n+2}=\lim_{n\rightarrow\infty}\frac{1}{1+\frac{1}{2n}}=1,\\
&\lim_{n\rightarrow\infty}\frac{\lambda_{\mathbf{M},n}+\kappa_0}{\frac{\mu}{2(\lambda+2\mu)n} }=\lim_{n\rightarrow\infty}\frac{(3\lambda+5\mu)n+4\mu n^2}{4\mu n^2-\mu}=\lim_{n\rightarrow\infty}\frac{\frac{3\lambda+5\mu}{n}+4\mu }{4\mu-\frac{\mu}{n^2}}=1,\\
&\lim_{n\rightarrow\infty}\frac{\lambda_{\mathbf{N},n}-\kappa_0}{\frac{\mu}{2(\lambda+2\mu)n} }=\lim_{n\rightarrow\infty}\frac{-(3\lambda+5\mu)n+4\mu n^2}{4\mu n^2-\mu}=\lim_{n\rightarrow\infty}\frac{-\frac{3\lambda+5\mu}{n}+4\mu }{4\mu-\frac{\mu}{n^2}}=1.
\end{align*}
We have obtained the conclusion of the lemma.
\end{proof}

Let $B$ be a  central ball with radius $1$. Provided that  $\tilde{\boldsymbol{\psi}}\in H^s(\partial B)^3$ with
\begin{align*}
\tilde{\boldsymbol{\psi}}=&\sum^{+\infty}_{n=1}\sum^n_{m=-n}\langle\tilde{\boldsymbol{\psi}},\tilde{\mathbf{T}}^{m}_{n}\rangle_{L^{2}(\partial B)^3}\tilde{\mathbf{T}}^{m}_{n}+\sum^{+\infty}_{n=1}\sum^n_{m=-n}\langle\tilde{\boldsymbol{\psi}},\tilde{\mathbf{M}}^{m}_{n}\rangle_{L^{2}(\partial B)^3}\tilde{\mathbf{M}}^{m}_{n}\\
&+\sum^{+\infty}_{n=1}\sum^{n-1}_{m=-(n-1)}\langle\tilde{\boldsymbol{\psi}},\tilde{\mathbf{N}}^{m}_{n}\rangle_{L^{2}(\partial B)^3}\tilde{\mathbf{N}}^{m}_{n},
\end{align*}
the corresponding $H^s$-norm is given by (cf. \cite{liu2013enhanced})
\begin{align*}
\|\tilde{\boldsymbol{\psi}}\|^2_{H^s(\partial B)^3}=&\sum^{\infty}_{n=1}\sum^n_{m=-n}(1+n(n+1))^s\big(|\langle\tilde{\boldsymbol{\psi}},\tilde{\mathbf{T}}^{m}_{n}\rangle_{L^{2}(\partial B)^3}|^2+|\langle\tilde{\boldsymbol{\psi}},\tilde{\mathbf{M}}^{m}_{n}\rangle_{L^{2}(\partial B)^3}|^2\big)\\
&+\sum^{\infty}_{n=1}\sum^{n-1}_{m=-(n-1)}(1+n(n+1))^s|\langle\tilde{\boldsymbol{\psi}},\tilde{\mathbf{N}}^{m}_{n}\rangle_{L^{2}(\partial B)^3}|^2.
\end{align*}
Moreover, we define Sobolev spaces  $H^{s}_{\mathbf{T}}(\partial B)^3, H^{s}_{\mathbf{M}}(\partial B)^3$ and $H^{s}_{\mathbf{T}}(\partial B)^3$ by
\begin{align*}
&H^{s}_{\mathbf{T}}(\partial B)^3=\left\{\tilde{\boldsymbol{\psi}}\in H^s(\partial B)^3:\langle \tilde{\boldsymbol{\psi}},\tilde{\mathbf{M}}^m_n\rangle_{L^2(\partial B)^3}=0,\langle\tilde{\boldsymbol{\psi}},\tilde{\mathbf{N}}^m_n\rangle_{L^2(\partial B)^3}=0\right\},\\
&H^{s}_{\mathbf{M}}(\partial B)^3=\left\{\tilde{\boldsymbol{\psi}}\in H^s(\partial B)^3:\langle \tilde{\boldsymbol{\psi}},\tilde{\mathbf{T}}^m_n\rangle_{L^2(\partial B)^3}=0,\langle\tilde{\boldsymbol{\psi}},\tilde{\mathbf{N}}^m_n\rangle_{L^2(\partial B)^3}=0\right\},\\
&H^{s}_{\mathbf{N}}(\partial B)^3=\left\{\tilde{\boldsymbol{\psi}}\in H^s(\partial B)^3:\langle \tilde{\boldsymbol{\psi}},\tilde{\mathbf{T}}^m_n\rangle_{L^2(\partial B)^3}=0,\langle\tilde{\boldsymbol{\psi}},\tilde{\mathbf{M}}^m_n\rangle_{L^2(\partial B)^3}=0\right\}.
\end{align*}
For each real number $s\geq0$, there exist $c_i>0,C_i>0,i=\mathbf{T},\mathbf{M},\mathbf{N}$ depending on $\lambda,\mu$ at most such that for $\tilde{\boldsymbol{\psi}}_{\mathbf{T}}\in H^{s}_{\mathbf{T}}(\partial B)^3, \tilde{\boldsymbol{\psi}}_{\mathbf{M}}\in H^{s}_{\mathbf{M}}(\partial B)^3$ and $\tilde{\boldsymbol{\psi}}_{\mathbf{N}}\in H^{s}_{\mathbf{N}}(\partial B)^3$,
\begin{align}\label{E:estimate}
&c_{\mathbf{T}}\|\tilde{\boldsymbol{\psi}}_{\mathbf{T}}\|_{H^{s}(\partial B)^3}\leq\|\mathbf{K}^{*}_{\partial B}[\tilde{\boldsymbol{\psi}}_{\mathbf{T}}]\|_{H^{s+1}(\partial B)^3}\leq C_{\mathbf{T}}\|\tilde{\boldsymbol{\psi}}_{\mathbf{T}}\|_{H^{s}(\partial B)^3},\\
&c_{\mathbf{M}}\|\tilde{\boldsymbol{\psi}}_{\mathbf{M}}\|_{H^{s}(\partial B)^3}\leq\|(\mathbf{K}^{*}_{\partial B}+\kappa_0\mathbf{I})[\tilde{\boldsymbol{\psi}}_{\mathbf{M}}]\|_{H^{s+1}(\partial B)^3}\leq C_\mathbf{M}\|\tilde{\boldsymbol{\psi}}_{\mathbf{M}}\|_{H^{s}(\partial B)^3},\label{E:estimate1}\\
&c_\mathbf{N}\|\tilde{\boldsymbol{\psi}}_{\mathbf{N}}\|_{H^{s}(\partial B)^3}\leq\|(\mathbf{K}^{*}_{\partial B}-\kappa_0\mathbf{I})[\tilde{\boldsymbol{\psi}}_{\mathbf{N}}]\|_{H^{s+1}(\partial B)^3}\leq C_\mathbf{N}\|\tilde{\boldsymbol{\psi}}_{\mathbf{N}}\|_{H^{s}(\partial B)^3}.\label{E:estimate2}
\end{align}
Similarly, we decompose that for $\tilde{\mathbf{F}}^{\omega\delta}\in H^{s}(\partial B)^3$,
\begin{align*} \tilde{\mathbf{F}}^{\omega\delta}=\tilde{\mathbf{F}}^{\omega\delta}_{\mathbf{T}}+\tilde{\mathbf{F}}^{\omega\delta}_{\mathbf{M}}+\tilde{\mathbf{F}}^{\omega\delta}_{\mathbf{N}}
\end{align*}
with
$\tilde{\mathbf{F}}^{\omega\delta}_{\mathbf{T}}\in H^{s}_{\mathbf{T}}(\partial B)^3,\tilde{\mathbf{F}}^{\omega\delta}_{\mathbf{M}}\in H^{s}_{\mathbf{M}}(\partial B)^3$ and $\tilde{\mathbf{F}}^{\omega\delta}_{\mathbf{N}}\in H^{s}_{\mathbf{N}}(\partial B)^3$, where
\begin{align*}
&\tilde{\mathbf{F}}^{\omega\delta}_{\mathbf{T}}=\sum_{n=1}^{\infty}\sum_{m=-n}^{n}\langle\tilde{\mathbf{F}}^{\omega\delta},\tilde{\mathbf{T}}^{m}_{n}\rangle_{L^{2}(\partial B)^3}\tilde{\mathbf{T}}^{m}_{n},\\
&\tilde{\mathbf{F}}^{\omega\delta}_{\mathbf{M}}=\sum_{n=1}^{\infty}\sum_{m=-n}^{n}\langle\tilde{\mathbf{F}}^{\omega\delta},\tilde{\mathbf{M}}^{m}_{n}\rangle_{L^{2}(\partial B)^3}\tilde{\mathbf{M}}^{m}_{n},\\
&\tilde{\mathbf{F}}^{\omega\delta}_{\mathbf{N}}=\sum_{n=1}^{\infty}\sum_{m=-(n-1)}^{n-1}\langle\tilde{\mathbf{F}}^{\omega\delta},\tilde{\mathbf{N}}^{m}_{n}\rangle_{L^{2}(\partial B)^3}\tilde{\mathbf{N}}^{m}_{n}.
\end{align*}

In the next proposition we establish the decay estimates of Fourier coefficients.

\begin{prop}
Let the domain $B$ be a central ball of radius $1$. There exists an integer $N>1$ large enough such that if $\tilde{\mathbf{F}}^{\omega\delta}\in H^{N}(\partial B)^3$, then the corresponding Fourier coefficients  satisfy
\begin{align}
&\langle\tilde{\mathbf{F}}^{\omega\delta},\tilde{\mathbf{T}}^{m}_{n}\rangle_{L^{2}(\partial B)^3}=o(n^{-\frac{N}{2}})\quad\text{as $n\rightarrow+\infty$},\label{E:decay1}\\
&\langle\tilde{\mathbf{F}}^{\omega\delta},\tilde{\mathbf{M}}^{m}_{n}\rangle_{L^{2}(\partial B)^3}=o(n^{-\frac{N}{2}})\quad\text{as $n\rightarrow+\infty$},\label{E:decay2}\\
&\langle\tilde{\mathbf{F}}^{\omega\delta},\tilde{\mathbf{N}}^{m}_{n}\rangle_{L^{2}(\partial B)^3}=o(n^{-\frac{N}{2}})\quad\text{as $n\rightarrow+\infty$}.\label{E:decay3}
\end{align}
\end{prop}
\begin{proof}
Observe that
\begin{align*}
&\langle\tilde{\mathbf{F}}^{\omega\delta}_{\mathbf{T}},\tilde{\mathbf{T}}^{m}_{n}\rangle_{L^{2}(\partial B)^3}=\langle\tilde{\mathbf{F}}^{\omega\delta},\tilde{\mathbf{T}}^{m}_{n}\rangle_{L^{2}(\partial B)^3},\\
&\langle\tilde{\mathbf{F}}^{\omega\delta}_{\mathbf{M}},\tilde{\mathbf{M}}^{m}_{n}\rangle_{L^{2}(\partial B)^3}=\langle\tilde{\mathbf{F}}^{\omega\delta},\tilde{\mathbf{M}}^{m}_{n}\rangle_{L^{2}(\partial B)^3},\\
&\langle\tilde{\mathbf{F}}^{\omega\delta}_{\mathbf{N}},\tilde{\mathbf{N}}^{m}_{n}\rangle_{L^{2}(\partial B)^3}=\langle\tilde{\mathbf{F}}^{\omega\delta},\tilde{\mathbf{N}}^{m}_{n}\rangle_{L^{2}(\partial B)^3}.
\end{align*}
Let us first check \eqref{E:decay1}. It is straightforward to see that
\begin{align*}
H^{s}_{\mathbf{T}}(\partial B)^3=\mathbf{K}^*_{\partial B}(H^{s}_{\mathbf{T}}(\partial B)^3)\oplus \ker(\mathbf{K}^*_{\partial B}),
\end{align*}
where $\ker(\mathbf{K}^*_{\partial B})$ denotes the kernel of $\mathbf{K}^*_{\partial B}$. The symbol $\oplus $ is understood in the sense of $L^2$ scalar product.  Since $\tilde{\mathbf{F}}^{\omega\delta}_{\mathbf{T}}\in H^{N}_{\mathbf{T}}(\partial B)^3$, one deduces
\begin{align}\label{E:decomposition}
\tilde{\mathbf{F}}^{\omega\delta}_{\mathbf{T}}=\mathbf{K}^*_{\partial B}(\tilde{\mathbf{G}}^{(1)}_{\mathbf{T}})+\tilde{\mathbf{F}}^{(1)}_{{\mathbf{T}},\ker}
\end{align}
with $\tilde{\mathbf{G}}^{(1)}_{\mathbf{T}}\in H^{N-1}_{\mathbf{T}}(\partial B)^3$. Then for $n\geq1,-n\leq m\leq n$,
\begin{align*}
\langle\tilde{\mathbf{F}}^{\omega\delta}_{\mathbf{T}},\tilde{\mathbf{T}}^{m}_{n}\rangle_{L^{2}(\partial B)^3}=&\langle\mathbf{K}^*_{\partial B}(\tilde{\mathbf{G}}^{(1)}_{\mathbf{T}}),\tilde{\mathbf{T}}^{m}_{n}\rangle_{L^{2}(\partial B)^3}+\langle\tilde{\mathbf{F}}^{(1)}_{{\mathbf{T}},\ker},\tilde{\mathbf{T}}^{m}_{n}\rangle_{L^{2}(\partial B)^3}\\
=&\lambda_{\mathbf{T},n}\langle\tilde{\mathbf{G}}^{(1)}_{\mathbf{T}},\tilde{\mathbf{T}}^{m}_{n}\rangle_{L^{2}(\partial B)^3}+\langle\tilde{\mathbf{F}}^{(1)}_{{\mathbf{T}},\ker},\tilde{\mathbf{T}}^{m}_{n}\rangle_{L^{2}(\partial B)^3}\\
=&\lambda_{\mathbf{T},n}\langle\tilde{\mathbf{G}}^{(1)}_{\mathbf{T}},\tilde{\mathbf{T}}^{m}_{n}\rangle_{L^{2}(\partial B)^3}.
\end{align*}
Similarly, we can write $\tilde{\mathbf{G}}^{(1)}_{\mathbf{T}}=\mathbf{K}^*_{\partial B}(\tilde{\mathbf{G}}^{(2)}_{\mathbf{T}})+\tilde{\mathbf{F}}^{(2)}_{{\mathbf{T}},\ker}$, which then gives
\begin{align*}
\langle\tilde{\mathbf{G}}^{(1)}_{\mathbf{T}},\tilde{\mathbf{T}}^{m}_{n}\rangle_{L^{2}(\partial B)^3}
=\lambda_{\mathbf{T},n}\langle\tilde{\mathbf{G}}^{(2)}_{\mathbf{T}},\tilde{\mathbf{T}}^{m}_{n}\rangle_{L^{2}(\partial B)^3}.
\end{align*}
As a result,
\begin{align*}
\langle\tilde{\mathbf{F}}^{\omega\delta}_{\mathbf{T}},\tilde{\mathbf{T}}^{m}_{n}\rangle_{L^{2}(\partial B)^3}=(\lambda_{\mathbf{T},n})^{N}\langle\tilde{\mathbf{G}}^{(N)}_{\mathbf{T}},\tilde{\mathbf{T}}^{m}_{n}\rangle_{L^{2}(\partial B)^3}.
\end{align*}

In order to prove \eqref{E:decay1}, our aim is to control the $L^2$-norm of $\tilde{\mathbf{G}}^{(N)}_{\mathbf{T}}$. By using  \eqref{E:decomposition}, one arrives at
\begin{align*}
\tilde{\mathbf{F}}^{\omega\delta}_{\mathbf{T}}=(\mathbf{K}^*_{\partial B})^{N}(\tilde{\mathbf{G}}^{(N)}_{\mathbf{T}})+\tilde{\mathbf{F}}^{(1)}_{{\mathbf{T}},\ker},
\end{align*}
which leads to
\begin{align*}
\mathbf{K}^*_{\partial B}(\tilde{\mathbf{F}}^{\omega\delta}_{\mathbf{T}})=(\mathbf{K}^*_{\partial B})^{N+1}(\tilde{\mathbf{G}}^{(N)}_{\mathbf{T}}).
\end{align*}
It follows from \eqref{E:estimate} that
\begin{align*}
\|(\mathbf{K}^*_{\partial B})^{N+1}(\tilde{\mathbf{G}}^{(N)}_{\mathbf{T}})\|_{H^{N+1}(\partial B)^3}=\|\mathbf{K}^*_{\partial B}(\tilde{\mathbf{F}}^{\omega\delta}_{\mathbf{T}})\|_{H^{N+1}(\partial B)^3}\leq C_{\mathbf{T}}\|\tilde{\mathbf{F}}^{\omega\delta}_{\mathbf{T}}\|_{H^{N}(\partial B)^3}.
\end{align*}
On the other hand, one has
\begin{align*}
\|\tilde{\mathbf{G}}^{(N)}_{\mathbf{T}}\|_{L^{2}(\partial B)^3}\stackrel{\eqref{E:estimate}}{\leq}&\frac{1}{c^{N+1}_{\mathbf{T}}}\|(\mathbf{K}^*_{\partial B})^{N+1}(\tilde{\mathbf{G}}^{(N)}_{\mathbf{T}})\|_{{H^{N+1}}(\partial B)^3}
\leq \frac{C_\mathbf{T}}{c^{N+1}_\mathbf{T}} \|\tilde{\mathbf{F}}^{\omega\delta}_{\mathbf{T}}\|_{H^{N}(\partial B)^3}.
\end{align*}
According to $\mathrm{(i)}$ in Corollary \ref{coro:spectrum}, we can conclude that
\begin{align*}
|\langle\tilde{\mathbf{F}}^{\omega\delta}_{\mathbf{T}},\tilde{\mathbf{T}}^{m}_{n}\rangle_{L^{2}(\partial B)^3}|\leq&(\lambda_{\mathbf{T},n})^{N}\|\tilde{\mathbf{G}}^{(N)}_{\mathbf{T}}\|_{L^{2}(\partial B)^3}\|\tilde{\mathbf{T}}^{m}_{n}\|_{L^{2}(\partial B)^3}\\
\leq& \frac{C_\mathbf{T}}{c^{N+1}_\mathbf{T}}(\lambda_{\mathbf{T},n})^{N}\|\tilde{\mathbf{F}}^{\omega\delta}_{\mathbf{T}}\|_{H^{N}(\partial B)^3}\\
\leq& Cn^{-N}\|\tilde{\mathbf{F}}^{\omega\delta}\|_{H^{N}(\partial B)^3}\\
=&o(n^{-\frac{N}{2}})
\end{align*}
as $n\rightarrow\infty$. Thus formula \eqref{E:decay1} holds.

The remainder of the proof is to verify \eqref{E:decay2} ({\rm{resp.}} \eqref{E:decay3}). We can proceed with a similar procedure as above. The proof is based on $\mathrm{(ii)}$ ({\rm{resp.}}  $\mathrm{(iii)}$) in Corollary \ref{coro:spectrum} and formula \eqref{E:estimate1} ({\rm{resp.}} \eqref{E:estimate2}). The main differences include the  space  $H^{s}_{\mathbf{M}}(\partial B)^3$ ({\rm{resp.}} $H^{s}_{\mathbf{N}}(\partial B)^3$) and  the operator $\mathbf{K}^*_{\partial B}+\kappa_0\mathbf{I}$ ({\rm{resp.}} $\mathbf{K}^*_{\partial B}-\kappa_0\mathbf{I}$) instead of the space  $H^{s}_{\mathbf{T}}(\partial B)^3$ and  the operator $\mathbf{K}^*_{\partial B}$, respectively. Moreover, we consider $\tilde{\mathbf{N}}^{m}_{n}$ for $n\geq\lfloor\frac{3\lambda+5\mu}{4\mu}\rfloor+1,-(n-1)\leq m\leq n-1$ in the proof of formula \eqref{E:decay3}. The proof of this proposition is now complete.
\end{proof}

In addition, we give the asymptotic expansion for the operator $\boldsymbol{\mathcal{A}}^{\omega\delta}_{\partial B}$  in the next lemma.

\begin{lemm}\label{le:perturbation}
If $\omega\delta\ll1$, then the operator $\boldsymbol{\mathcal{A}}^{\omega\delta}_{\partial B}$ defined in \eqref{AOD} admits the following asymptotic expansion
\begin{align*}
\boldsymbol{\mathcal{A}}^{\omega\delta}_{\partial B}=&\boldsymbol{\mathcal{A}}^0_{\partial B}+(c(\omega)-1)(\omega\delta)^2
\boldsymbol{\mathcal{A}}_{\partial B,1}+\mathcal {O}((\omega\delta)^3),
\end{align*}
where $\boldsymbol{\mathcal{A}}^0_{\partial B}$ is as seen in \eqref{A-Zero} and
\begin{align}\label{E:AB1}
\boldsymbol{\mathcal{A}}_{\partial B,1}=\big(-\frac{1}{2}\mathbf{I}+\mathbf{K}^{*}_{\partial B}\big)\mathbf{S}^{-1}_{\partial B}\mathbf{P}_{\partial B}.
\end{align}
\end{lemm}

Before starting the corresponding proof,  we introduce the following auxiliary result.

\begin{lemm}\label{le:constant}
Let $\mathbf{C}$ be a constant vector. For $\tilde{\boldsymbol\psi}\in{L^2}(\partial B)^3$, if we assume that $\mathbf{S}_{\partial B}[\tilde{\boldsymbol{\psi}}]=\mathbf{C}$ holds on $\partial B$,  then
\begin{align*}
\big(-\frac{1}{2}\mathbf{I}+\mathbf{K}^{*}_{\partial B}\big)\mathbf{S}^{-1}_{\partial B}[\mathbf{C}]=0.
\end{align*}
\end{lemm}
\begin{proof}
Let us  consider  the Dirichlet boundary value problem
\begin{align*}
\left\{ \begin{aligned}
&\mathcal{L}_{\lambda,\mu}\tilde{\mathbf{u}}(\boldsymbol X)=0,\quad&&\boldsymbol X\in B,\\
&\tilde{\mathbf{u}}(\boldsymbol X)=\mathbf{C},\quad&&\boldsymbol X\in\partial B,
\end{aligned}\right.
\end{align*}
which has a unique solution $\tilde{\mathbf{u}}(\boldsymbol X)=\mathbf{C}$. Consequently, if $\mathbf{S}_{\partial B}[\tilde{\boldsymbol{\psi}}]=\mathbf{C}$ on $\partial B$, then $\mathbf{S}_{\partial B}[\tilde{\boldsymbol{\psi}}]=\mathbf{C}$ in the domain $B$.  As a result,
\begin{align*}
\frac{\partial}{\partial{\boldsymbol{\nu}}}\mathbf{S}_{\partial B}[\tilde{\boldsymbol{\psi}}]\big|_{-}(\boldsymbol X)=0,\quad \boldsymbol X\in\partial {B}.
\end{align*}
Because of  the jump condition \eqref{Jump2} with $\omega=0$, we obtain
\begin{align*}
\big(-\frac{1}{2}\mathbf{I}+\mathbf{K}^{*}_{\partial B}\big)[\tilde{\boldsymbol{\psi}}](\boldsymbol X)=0,\quad \boldsymbol X\in \partial B.
\end{align*}
Since $\mathbf{S}_{\partial B}$ is invertible in $L^{2}(\partial B)^3$ by Lemma \ref{le:known}, we decuce $\tilde{\boldsymbol{\psi}}=\mathbf{S}^{-1}_{\partial B}[\mathbf{C}]$. This ends the proof of the lemma.
 \end{proof}

Let us complete the proof of Lemma \ref{le:perturbation}.

\begin{proof}[Proof of Lemma \ref{le:perturbation}]
It follows from the expressions of $\omega_1,\omega_2$ (see \eqref{Omega1}, \eqref{Omega2}) and Lemmas \ref{le:S-series}--\ref{le:IS-series} that for $\omega\delta\ll1$,
\begin{align*}
&\mathbf{S}^{\omega_2\delta}_{\partial B}
=\mathbf{S}_{\partial B}+\omega_2\delta\mathbf{R}_{\partial B}+(\omega_2\delta)^2\mathbf{P}_{\partial B}+\mathcal{O}((\omega\delta)^3),\\
&(\mathbf{S}^{\omega_1\delta}_{\partial B})^{-1}
=\mathbf{S}^{-1}_{\partial B}+\omega_1\delta\mathbf{R}_{\partial B,1}+(\omega_1\delta)^2\mathbf{P}_{\partial B,1}+\mathcal{O}((\omega\delta)^3),\\
&\mathbf{K}^{\omega_i\delta,*}_{\partial B}=\mathbf{K}^*_{\partial B}+(\omega_i\delta)^2\mathbf{Q}_{\partial B}+\mathcal{O}((\omega\delta)^3),\quad i=1,2.
\end{align*}
It is clear that
\begin{align*}
(\mathbf{S}^{\omega_1\delta}_{\partial B})^{-1}\mathbf{S}^{\omega_2\delta}_{\partial B}=&\mathbf{I}+\omega_2\delta \mathbf{S}^{-1}_{\partial B}\mathbf{R}_{\partial B}+\omega_1\delta \mathbf{R}_{\partial B,1}\mathbf{S}_{\partial B}+(\omega_2\delta)^2\mathbf{S}^{-1}_{\partial B}\mathbf{P}_{\partial B}\\
&
+(\omega_1\delta)(\omega_2\delta)\mathbf{R}_{\partial B,1}\mathbf{R}_{\partial B}+(\omega_1\delta)^2\mathbf{P}_{\partial B,1}\mathbf{S}_{\partial B}+\mathcal {O}((\omega\delta)^3),
\end{align*}
which leads to
\begin{align*}
\big(-\frac{1}{2}\mathbf{I}+\mathbf{K}^{\omega_1\delta,*}_{\partial B}\big)(\mathbf{S}^{\omega_1\delta}_{\partial B})^{-1}\mathbf{S}^{\omega_2\delta}_{\partial B}=&\big(-\frac{1}{2}\mathbf{I}+\mathbf{K}^{*}_{\partial B}\big)+\omega_2\delta \big(-\frac{1}{2}\mathbf{I}+\mathbf{K}^{*}_{\partial B}\big)\mathbf{S}^{-1}_{\partial B}\mathbf{R}_{\partial B}\\
&+\omega_1\delta \big(-\frac{1}{2}\mathbf{I}+\mathbf{K}^{*}_{\partial B}\big)\mathbf{R}_{\partial B,1}\mathbf{S}_{\partial B}+(\omega_2\delta)^2\big(-\frac{1}{2}\mathbf{I}+\mathbf{K}^{*}_{\partial B}\big)\mathbf{S}^{-1}_{\partial B}\mathbf{P}_{\partial B}\\
&
+(\omega_1\delta)(\omega_2\delta)\big(-\frac{1}{2}\mathbf{I}+\mathbf{K}^{*}_{\partial B}\big)\mathbf{R}_{\partial B,1}\mathbf{R}_{\partial B}\\
&+(\omega_1\delta)^2\big(-\frac{1}{2}\mathbf{I}+\mathbf{K}^{*}_{\partial B}\big)\mathbf{P}_{\partial B,1}\mathbf{S}_{\partial B}+(\omega_1\delta)^2\mathbf{Q}_{\partial B}+\mathcal {O}((\omega\delta)^3).
\end{align*}
According to the definitions of $\mathbf{R}_{\partial B},\mathbf{R}_{\partial B,1},\mathbf{P}_{\partial B,1}$ in \eqref{ri} and \eqref{ri1},  using Lemma \ref{le:constant} yields that for all $\tilde{\boldsymbol{\psi}}\in{L^2}(\partial B)^3$,
\begin{align*}
&\big(-\frac{1}{2}\mathbf{I}+\mathbf{K}^{*}_{\partial B}\big)\mathbf{S}^{-1}_{\partial B}\mathbf{R}_{\partial B}[\tilde{\boldsymbol{\psi}}](\boldsymbol X)=0,\\
&\big(-\frac{1}{2}\mathbf{I}+\mathbf{K}^{*}_{\partial B}\big)\mathbf{R}_{\partial B,1}\mathbf{S}_{\partial B}[\tilde{\boldsymbol{\psi}}](\boldsymbol X)=\big(-\frac{1}{2}\mathbf{I}+\mathbf{K}^{*}_{\partial B}\big)\mathbf{S}^{-1}_{\partial B}\mathbf{R}_{\partial B}[\tilde{\boldsymbol{\varphi}}](\boldsymbol X)=0,\\
&\big(-\frac{1}{2}\mathbf{I}+\mathbf{K}^{*}_{\partial B}\big)\mathbf{R}_{\partial B,1}\mathbf{R}_{\partial B}[\tilde{\boldsymbol{\psi}}](\boldsymbol X)=\big(-\frac{1}{2}\mathbf{I}+\mathbf{K}^{*}_{\partial B}\big)\mathbf{S}^{-1}_{\partial B}\mathbf{R}_{\partial B}[\mathbf{S}^{-1}_{\partial B}\mathbf{R}_{\partial B}[\tilde{\boldsymbol{\psi}}](\boldsymbol X)]=0,\\
&\big(-\frac{1}{2}\mathbf{I}+\mathbf{K}^{*}_{\partial B}\big)\mathbf{P}_{\partial B,1}\mathbf{S}_{\partial B}[\tilde{\boldsymbol{\psi}}](\boldsymbol X)
=-\big(-\frac{1}{2}\mathbf{I}+\mathbf{K}^{*}_{\partial B}\big)\mathbf{S}_{\partial B}^{-1}\mathbf{P}_{\partial B}[\tilde{\boldsymbol{\psi}}](\boldsymbol X).
\end{align*}
Hence,
\begin{align*}
\big(-\frac{1}{2}\mathbf{I}+\mathbf{K}^{\omega_1\delta,*}_{\partial B}\big)(\mathbf{S}^{\omega_1\delta}_{\partial B})^{-1}\mathbf{S}^{\omega_2\delta}_{\partial B}=&\big(-\frac{1}{2}\mathbf{I}+\mathbf{K}^{*}_{\partial B}\big)+(\omega_1\delta)^2\mathbf{Q}_{\partial B}\\
&+((\omega_2\delta)^2-(\omega_1\delta)^2)(-\frac{1}{2}\mathbf{I}+\mathbf{K}^{*}_{\partial B}\big)\mathbf{S}^{-1}_{\partial B}\mathbf{P}_{\partial B}+\mathcal {O}((\omega\delta)^3).
\end{align*}
Moreover,
\begin{align*}
-\big(\frac{1}{2}\mathbf{I}+\mathbf{K}^{\omega_2\delta,*}_{\partial B}\big)=-\big(\frac{1}{2}\mathbf{I}+\mathbf{K}^{*}_{\partial B}\big)-(\omega_2\delta)^2\mathbf{Q}_{\partial B}+\mathcal {O}((\omega\delta)^3).
\end{align*}
Since $\omega_1=\frac{\omega}{\sqrt{c(\omega)}},\omega_2=\omega$ (see \eqref{Omega1}, \eqref{Omega2}), we obtain
\begin{align*}
\boldsymbol{\mathcal{A}}^{\omega\delta}_{\partial B}=&c(\omega)\big(-\frac{1}{2}\mathbf{I}+\mathbf{K}^{*}_{\partial B}\big)-\big(\frac{1}{2}\mathbf{I}+\mathbf{K}^{*}_{\partial B}\big)+(c(\omega)(\omega_1\delta)^2-(\omega_2\delta)^2)\mathbf{Q}_{\partial B}\\
&+c(\omega)((\omega_2\delta)^2-(\omega_1\delta)^2)\big(-\frac{1}{2}\mathbf{I}+\mathbf{K}^{*}_{\partial B}\big)\mathbf{S}^{-1}_{\partial B}\mathbf{P}_{\partial B}+\mathcal {O}((\omega\delta)^3)\\
=&c(\omega)\big(-\frac{1}{2}\mathbf{I}+\mathbf{K}^{*}_{\partial B}\big)-\big(\frac{1}{2}\mathbf{I}+\mathbf{K}^{*}_{\partial B}\big)+(c(\omega)-1)(\omega\delta)^2\big(-\frac{1}{2}\mathbf{I}+\mathbf{K}^{*}_{\partial B}\big)\mathbf{S}^{-1}_{\partial B}\mathbf{P}_{\partial B}+\mathcal {O}((\omega\delta)^3),
\end{align*}
which implies the conclusion of the lemma.

The proof is complete.
\end{proof}

The following lemma corresponds to the spectral approximation of the scattered field for the Lam\'{e} system  in the frequency domain.
\begin{prop}
Let $\tilde{\mathbf{F}}^{\omega\delta}\in H^{N}(\partial B)^3$ for an integer $N>1$ large enough. Suppose that
\begin{align*}
\tilde{\mathbf{F}}^{\omega\delta}=\tilde{\mathbf{F}}^{\omega\delta}_{\mathbf{T}}+\tilde{\mathbf{F}}^{\omega\delta}_{\mathbf{M}}+\tilde{\mathbf{F}}^{\omega\delta}_{\mathbf{N}}
\end{align*}
with
$\tilde{\mathbf{F}}^{\omega\delta}_{\mathbf{T}}\in H^{N}_{\mathbf{T}}(\partial B)^3,\tilde{\mathbf{F}}^{\omega\delta}_{\mathbf{M}}\in H^{N}_{\mathbf{M}}(\partial B)^3$ and $\tilde{\mathbf{F}}^{\omega\delta}_{\mathbf{N}}\in H^{N}_{\mathbf{N}}(\partial B)^3$, where
\begin{align*}
&\tilde{\mathbf{F}}^{\omega\delta}_{\mathbf{T}}=\sum_{n=1}^{N}\sum_{m=-n}^{n}\langle\tilde{\mathbf{F}}^{\omega\delta},\tilde{\mathbf{T}}^{m}_{n}\rangle_{L^{2}(\partial B)^3}\tilde{\mathbf{T}}^{m}_{n},\\
&\tilde{\mathbf{F}}^{\omega\delta}_{\mathbf{M}}=\sum_{n=1}^{N}\sum_{m=-n}^{n}\langle\tilde{\mathbf{F}}^{\omega\delta},\tilde{\mathbf{M}}^{m}_{n}\rangle_{L^{2}(\partial B)^3}\tilde{\mathbf{M}}^{m}_{n},\\
&\tilde{\mathbf{F}}^{\omega\delta}_{\mathbf{N}}=\sum_{n=1}^{N}\sum_{m=-(n-1)}^{n-1}\langle\tilde{\mathbf{F}}^{\omega\delta},\tilde{\mathbf{N}}^{m}_{n}\rangle_{L^{2}(\partial B)^3}\tilde{\mathbf{N}}^{m}_{n}.
\end{align*}
If $\omega\delta\ll1$, then the spectral approximation of the scattered field satisfies that for $\boldsymbol X\in\mathbb{R}^3\backslash \overline{B}$,
\begin{align*}
\tilde{\mathbf{u}}^{\mathrm{sca}}(\boldsymbol X)
=&\sum_{n=1}^{N}\sum_{m=-n}^{n}\frac{1}{\tau_{\mathbf{T},n}(\omega\delta)}\langle\tilde{\mathbf{F}}^{\omega\delta},
\tilde{\mathbf{T}}^{m}_{n}\rangle_{L^{2}(\partial B)^3}\mathbf{S}^{\omega_2\delta}_{\partial B}[\tilde{\mathbf{T}}^{m}_{n}](\boldsymbol X)\\
&+\sum_{n=1}^{N}\sum_{m=-n}^{n}\frac{1}{\tau_{\mathbf{M},n}(\omega\delta)}\langle\tilde{\mathbf{F}}^{\omega\delta},\tilde{\mathbf{M}}^{m}_{n}\rangle_{L^{2}(\partial B)^3}\mathbf{S}^{\omega_2\delta}_{\partial B}[\tilde{\mathbf{M}}^{m}_{n}](\boldsymbol X)\\
&+\sum_{n=1}^{N}\sum_{m=-(n-1)}^{n-1}\frac{1}{\tau_{\mathbf{N},n}(\omega\delta)}\langle\tilde{\mathbf{F}}^{\omega\delta},\tilde{\mathbf{N}}^{m}_{n}\rangle_{L^{2}(\partial B)^3}\mathbf{S}^{\omega_2\delta}_{\partial B}[\tilde{\mathbf{N}}^{m}_{n}](\boldsymbol X),
\end{align*}
where, for $i=\mathbf{T},\mathbf{M},\mathbf{N}$,
\begin{align}\label{E:tau}
\tau_{i,n}(\omega\delta)=&\tau_{i,n}+(c(\omega)-1)(\omega\delta)^2\varrho_{i,n}+\mathcal {O}((\omega\delta)^3)
\end{align}
with
\begin{align}
&\varrho_{\mathbf{T},n}=\langle\boldsymbol{\mathcal{A}}_{\partial B,1}[\tilde{\mathbf{T}}^{m}_{n}],\tilde{\mathbf{T}}^{m}_{n}\rangle_{L^{2}(\partial B)^3},\label{E:tau1n1}\\
&\varrho_{\mathbf{M},n}=\langle\boldsymbol{\mathcal{A}}_{\partial B,1}[\tilde{\mathbf{M}}^{m}_{n}],\tilde{\mathbf{M}}^{m}_{n}\rangle_{L^{2}(\partial B)^3},\label{E:tau1n2}\\
&\varrho_{\mathbf{N},n}=\langle\boldsymbol{\mathcal{A}}_{\partial B,1}[\tilde{\mathbf{N}}^{m}_{n}],\tilde{\mathbf{N}}^{m}_{n}\rangle_{L^{2}(\partial B)^3},\label{E:tau1n3}
\end{align}
where $\boldsymbol{\mathcal{A}}_{\partial B,1}$ is given in \eqref{E:AB1}.
\end{prop}
\begin{proof}
First, we note that \eqref{AO} is equivalent to the following system:
\begin{align*}
&\Pi_{\mathbf{T}}\boldsymbol{\mathcal{A}}^{\omega\delta}_{\partial B}[\tilde{\boldsymbol{\psi}}_{\mathbf{T}}+\tilde{\boldsymbol{\psi}}_{\mathbf{M}}+\tilde{\boldsymbol{\psi}}_{\mathbf{N}}](\boldsymbol X)=\frac{1}{\delta}\sum_{n=1}^{N}\sum_{m=-n}^{n}\langle\tilde{\mathbf{F}}^{\omega\delta},\tilde{\mathbf{T}}^{m}_{n}\rangle_{L^{2}(\partial B)^3}\tilde{\mathbf{T}}^{m}_{n}(\boldsymbol X),\\
&\Pi_{\mathbf{M}}\boldsymbol{\mathcal{A}}^{\omega\delta}_{\partial B}[\tilde{\boldsymbol{\psi}}_{\mathbf{T}}+\tilde{\boldsymbol{\psi}}_{\mathbf{M}}+\tilde{\boldsymbol{\psi}}_{\mathbf{N}}](\boldsymbol X)=\frac{1}{\delta}\sum_{n=1}^{N}\sum_{m=-n}^{n}\langle\tilde{\mathbf{F}}^{\omega\delta},\tilde{\mathbf{M}}^{m}_{n}\rangle_{L^{2}(\partial B)^3}\tilde{\mathbf{M}}^{m}_{n}(\boldsymbol X),\\
&\Pi_{\mathbf{N}}\boldsymbol{\mathcal{A}}^{\omega\delta}_{\partial B}[\tilde{\boldsymbol{\psi}}_{\mathbf{T}}+\tilde{\boldsymbol{\psi}}_{\mathbf{M}}+\tilde{\boldsymbol{\psi}}_{\mathbf{N}}](\boldsymbol X)=\frac{1}{\delta}\sum_{n=1}^{N}\sum_{m=-(n-1)}^{n-1}\langle\tilde{\mathbf{F}}^{\omega\delta},\tilde{\mathbf{N}}^{m}_{n}\rangle_{L^{2}(\partial B)^3}\tilde{\mathbf{N}}^{m}_{n}(\boldsymbol X).
\end{align*}
Furthermore,
\begin{align*}
\boldsymbol{\mathcal{A}}^{\omega\delta}_{\partial B}=(\boldsymbol{\mathcal{A}}^{0}_{\partial B}+\boldsymbol{\mathcal{A}}^{\omega\delta}_{\partial B}-\boldsymbol{\mathcal{A}}^{0}_{\partial B}).
\end{align*}
It follows from Proposition \ref{prop:tau} and the  spectral perturbation theory that
\begin{align*}
&\sum_{n=1}^{\infty}\sum_{m=-n}^{n}\big(\tau_{\mathbf{T},n}+\langle(\boldsymbol{\mathcal{A}}^{\omega\delta}_{\partial B}-\boldsymbol{\mathcal{A}}^{0}_{\partial B})[\tilde{\mathbf{T}}^{m}_{n}],\tilde{\mathbf{T}}^{m}_{n}\rangle_{L^{2}(\partial B)^3}\big)\langle\tilde{\boldsymbol{\psi}}_{\mathbf{T}},\tilde{\mathbf{T}}^{m}_{n}\rangle_{L^{2}(\partial B)^3}\tilde{\mathbf{T}}^{m}_{n}\\&=\frac{1}{\delta}\sum_{n=1}^{N}\sum_{m=-n}^{n}\langle\tilde{\mathbf{F}}^{\omega\delta},\tilde{\mathbf{T}}^{m}_{n}\rangle_{L^{2}(\partial B)^3}\tilde{\mathbf{T}}^{m}_{n},\\
&\sum_{n=1}^{\infty}\sum_{m=-n}^{n}\big(\tau_{\mathbf{M},n}+\langle(\boldsymbol{\mathcal{A}}^{\omega\delta}_{\partial B}-\boldsymbol{\mathcal{A}}^{0}_{\partial B})[\tilde{\mathbf{M}}^{m}_{n}],\tilde{\mathbf{M}}^{m}_{n}\rangle_{L^{2}(\partial B)^3}\big)\langle\tilde{\boldsymbol{\psi}}_{\mathbf{M}},\tilde{\mathbf{M}}^{m}_{n}\rangle_{L^{2}(\partial B)^3}\tilde{\mathbf{M}}^{m}_{n}\\&=\frac{1}{\delta}\sum_{n=1}^{N}\sum_{m=-n}^{n}\langle\tilde{\mathbf{F}}^{\omega\delta},\tilde{\mathbf{M}}^{m}_{n}\rangle_{L^{2}(\partial B)^3}\tilde{\mathbf{M}}^{m}_{n},\\
&\sum_{n=1}^{\infty}\sum_{m=-(n-1)}^{n-1}\big(\tau_{\mathbf{N},n}+\langle(\boldsymbol{\mathcal{A}}^{\omega\delta}_{\partial B}-\boldsymbol{\mathcal{A}}^{0}_{\partial B})[\tilde{\mathbf{N}}^{m}_{n}],\tilde{\mathbf{N}}^{m}_{n}\rangle_{L^{2}(\partial B)^3}\big)\langle\tilde{\boldsymbol{\psi}}_{\mathbf{N}},\tilde{\mathbf{N}}^{m}_{n}\rangle_{L^{2}(\partial B)^3}\tilde{\mathbf{N}}^{m}_{n}\\&
=\frac{1}{\delta}\sum_{n=1}^{N}\sum_{m=-(n-1)}^{n-1}\langle\tilde{\mathbf{F}}^{\omega\delta},\tilde{\mathbf{N}}^{m}_{n}\rangle_{L^{2}(\partial B)^3}\tilde{\mathbf{N}}^{m}_{n}.
\end{align*}
Therefore,
\begin{align*}
&\langle\tilde{\boldsymbol{\psi}}_{\mathbf{T}},\tilde{\mathbf{T}}^{m}_{n}\rangle_{L^{2}(\partial B)^3}=
\frac{1}{\delta}\frac{1}{\tau_{\mathbf{T},n}+\langle(\boldsymbol{\mathcal{A}}^{\omega\delta}_{\partial B}-\boldsymbol{\mathcal{A}}^{0}_{\partial B})[\tilde{\mathbf{T}}^{m}_{n}],\tilde{\mathbf{T}}^{m}_{n}\rangle_{L^{2}(\partial B)^3}}\langle\tilde{\mathbf{F}}^{\omega\delta},\tilde{\mathbf{T}}^{m}_{n}\rangle_{L^{2}(\partial B)^3},\\
&\text{if}\quad1\leq n\leq N,|m|\leq n,\\
&\langle\tilde{\boldsymbol{\psi}}_{\mathbf{T}},\tilde{\mathbf{T}}^{m}_{n}\rangle_{L^{2}(\partial B)^3}=0,\quad\text{if}\quad n>N,|m|\leq n,\\
&\langle\tilde{\boldsymbol{\psi}}_{\mathbf{M}},\tilde{\mathbf{M}}^{m}_{n}\rangle_{L^{2}(\partial B)^3}=
\frac{1}{\delta}\frac{1}{\tau_{\mathbf{M},n}+\langle(\boldsymbol{\mathcal{A}}^{\omega\delta}_{\partial B}-\boldsymbol{\mathcal{A}}^{0}_{\partial B})[\tilde{\mathbf{M}}^{m}_{n}],\tilde{\mathbf{M}}^{m}_{n}\rangle_{L^{2}(\partial B)^3}}\langle\tilde{\mathbf{F}}^{\omega\delta},\tilde{\mathbf{M}}^{m}_{n}\rangle_{L^{2}(\partial B)^3},\\
&\text{if}\quad1\leq n\leq N,|m|\leq n,\\
&\langle\tilde{\boldsymbol{\psi}}_{\mathbf{M}},\tilde{\mathbf{M}}^{m}_{n}\rangle_{L^{2}(\partial B)^3}=0,\quad\text{if}\quad n>N,|m|\leq n,\\
&\langle\tilde{\boldsymbol{\psi}}_{\mathbf{N}},\tilde{\mathbf{N}}^{m}_{n}\rangle_{L^{2}(\partial B)^3}=
\frac{1}{\delta}\frac{1}{\tau_{\mathbf{N},n}+\langle(\boldsymbol{\mathcal{A}}^{\omega\delta}_{\partial B}-\boldsymbol{\mathcal{A}}^{0}_{\partial B})[\tilde{\mathbf{N}}^{m}_{n}],\tilde{\mathbf{N}}^{m}_{n}\rangle_{L^{2}(\partial B)^3}}\langle\tilde{\mathbf{F}}^{\omega\delta},\tilde{\mathbf{N}}^{m}_{n}\rangle_{L^{2}(\partial B)^3},\\
& \text{if}\quad1\leq n\leq N,|m|\leq n-1,\\
&\langle\tilde{\boldsymbol{\psi}}_{\mathbf{N}},\tilde{\mathbf{N}}^{m}_{n}\rangle_{L^{2}(\partial B)^3}=0,\quad\text{if}\quad n>N,|m|\leq n-1.
\end{align*}
In addition,  it is evident from Lemma \ref{le:perturbation} that
\begin{align*}
&\langle(\boldsymbol{\mathcal{A}}^{\omega\delta}_{\partial B}-\boldsymbol{\mathcal{A}}^{0}_{\partial B})[\tilde{\mathbf{T}}^{m}_{n}],\tilde{\mathbf{T}}^{m}_{n}\rangle_{L^{2}(\partial B)^3}=(c(\omega)-1)(\omega\delta)^2\langle\boldsymbol{\mathcal{A}}_{\partial B,1}[\tilde{\mathbf{T}}^{m}_{n}],\tilde{\mathbf{T}}^{m}_{n}\rangle_{L^{2}(\partial B)^3}+\mathcal {O}((\omega\delta)^3),\\
&\langle(\boldsymbol{\mathcal{A}}^{\omega\delta}_{\partial B}-\boldsymbol{\mathcal{A}}^{0}_{\partial B})[\tilde{\mathbf{M}}^{m}_{n}],\tilde{\mathbf{M}}^{m}_{n}\rangle_{L^{2}(\partial B)^3}=(c(\omega)-1)(\omega\delta)^2\langle\boldsymbol{\mathcal{A}}_{\partial B,1}[\tilde{\mathbf{M}}^{m}_{n}],\tilde{\mathbf{M}}^{m}_{n}\rangle_{L^{2}(\partial B)^3}+\mathcal {O}((\omega\delta)^3),\\
&\langle(\boldsymbol{\mathcal{A}}^{\omega\delta}_{\partial B}-\boldsymbol{\mathcal{A}}^{0}_{\partial B})[\tilde{\mathbf{N}}^{m}_{n}],\tilde{\mathbf{N}}^{m}_{n}\rangle_{L^{2}(\partial B)^3}=(c(\omega)-1)(\omega\delta)^2\langle\boldsymbol{\mathcal{A}}_{\partial B,1}[\tilde{\mathbf{N}}^{m}_{n}],\tilde{\mathbf{N}}^{m}_{n}\rangle_{L^{2}(\partial B)^3}+\mathcal {O}((\omega\delta)^3).
\end{align*}
Then it follows from \eqref{Solution2} that for $\boldsymbol X\in\mathbb{R}^3\backslash\overline{B}$,
\begin{align*}
\tilde{\mathbf{u}}^{\mathrm{sca}}(\boldsymbol X)
=
\delta\mathbf{S}^{\omega_2\delta}_{\partial B}[\tilde{\boldsymbol{\psi}}_{\mathbf{T}}+\tilde{\boldsymbol{\psi}}_{\mathbf{M}}+\tilde{\boldsymbol{\psi}}_{\mathbf{N}}](\boldsymbol X)=&\delta\sum_{n=1}^{N}\sum_{m=-n}^{n}\langle\tilde{\boldsymbol{\psi}}_{\mathbf{T}},\tilde{\mathbf{T}}^{m}_{n}\rangle_{L^{2}(\partial B)^3}\mathbf{S}^{\omega_2\delta}_{\partial B}[\tilde{\mathbf{T}}^{m}_{n}](\boldsymbol X)\\
&+\delta\sum_{n=1}^{N}\sum_{m=-n}^{n}\langle\tilde{\boldsymbol{\psi}}_{\mathbf{M}},\tilde{\mathbf{M}}^{m}_{n}\rangle_{L^{2}(\partial B)^3}\mathbf{S}^{\omega_2\delta}_{\partial B}[\tilde{\mathbf{M}}^{m}_{n}](\boldsymbol X)\\
&+\delta\sum_{n=1}^{N}\sum_{m=-(n-1)}^{n-1}\langle\tilde{\boldsymbol{\psi}}_{\mathbf{N}},\tilde{\mathbf{N}}^{m}_{n}\rangle_{L^{2}(\partial B)^3}\mathbf{S}^{\omega_2\delta}_{\partial B}[\tilde{\mathbf{N}}^{m}_{n}](\boldsymbol X),
\end{align*}
which readily completes the proof of the proposition.
\end{proof}

For each eigenfunction of $\mathbf{K}^*_{\partial B}$, we consider the function on $\partial D$ as follows
\begin{align*}
\mathbf{T}^{m}_{n}(\boldsymbol x)=\tilde{\mathbf{T}}^{m}_{n}\big(\frac{\boldsymbol x-\boldsymbol z}{\delta}\big),\quad \mathbf{M}^{m}_{n}(\boldsymbol x)=\tilde{\mathbf{M}}^{m}_{n}\big(\frac{\boldsymbol x-\boldsymbol z}{\delta}\big),\quad \mathbf{N}^{m}_{n}(\boldsymbol x)=\tilde{\mathbf{N}}^{m}_{n}\big(\frac{\boldsymbol x-\boldsymbol z}{\delta}\big).
\end{align*}
On account of Lemmas \ref{le:scale} and \ref{le:scalar}, one obtains:
\begin{align*}
&\mathbf{S}^{\omega_2\delta}_{\partial B}[\tilde{\mathbf{T}}^{m}_{n}](\boldsymbol X)=\frac{1}{\delta}\mathbf{S}^{\omega_2}_{\partial D}[\mathbf{T}^{m}_{n}](\boldsymbol x),\quad\langle\tilde{\mathbf{F}}^{\omega\delta}, \tilde{\mathbf{T}}^{m}_{n}\rangle_{L^{2}(\partial B)^3}=\frac{1}{\delta^2}\langle\mathbf{F}^{\omega\delta},\mathbf{T}^{m}_{n}\rangle_{L^{2}(\partial D)^3},\quad\|\mathbf{T}^{m}_{n}\|_{{L^{2}(\partial D)^3}}=\delta,\\
&\mathbf{S}^{\omega_2\delta}_{\partial B}[\tilde{\mathbf{M}}^{m}_{n}](\boldsymbol X)=\frac{1}{\delta}\mathbf{S}^{\omega_2}_{\partial D}[\mathbf{M}^{m}_{n}](\boldsymbol x),\quad\langle\tilde{\mathbf{F}}^{\omega\delta}, \tilde{\mathbf{M}}^{m}_{n}\rangle_{L^{2}(\partial B)^3}=\frac{1}{\delta^2}\langle\mathbf{F}^{\omega\delta},\mathbf{M}^{m}_{n}\rangle_{L^{2}(\partial D)^3},\quad\|\mathbf{M}^{m}_{n}\|_{{L^{2}(\partial D)^3}}=\delta,\\
&\mathbf{S}^{\omega_2\delta}_{\partial B}[\tilde{\mathbf{N}}^{m}_{n}](\boldsymbol X)=\frac{1}{\delta}\mathbf{S}^{\omega_2}_{\partial D}[\mathbf{N}^{m}_{n}](\boldsymbol x),\quad\langle\tilde{\mathbf{F}}^{\omega\delta}, \tilde{\mathbf{N}}^{m}_{n}\rangle_{L^{2}(\partial B)^3}=\frac{1}{\delta^2}\langle\mathbf{F}^{\omega\delta},\mathbf{N}^{m}_{n}\rangle_{L^{2}(\partial D)^3},\quad\|\mathbf{N}^{m}_{n}\|_{{L^{2}(\partial D)^3}}=\delta.
\end{align*}
Then we have the following proposition corresponding to the original unscaled problem.

\begin{prop}\label{prop:approximation}
 For an integer $N>1$ large enough, if $\omega\delta\ll1$, then the spectral approximation of the scattered field for the Lam\'{e} system \eqref{SS:lame}  satisfies that for $\boldsymbol x\in\mathbb{R}^3\backslash \overline{D}$,
\begin{align*}
\mathbf{u}^{\mathrm{sca}}(\boldsymbol x)
=&\frac{1}{\delta}\sum_{n=1}^{N}\sum_{m=-n}^{n}\frac{1}{\tau_{\mathbf{T},n}(\omega\delta)}\langle\mathbf{F}^{\omega\delta},\breve{\mathbf{T}}^{m}_{n}\rangle_{L^{2}(\partial D)^3}\mathbf{S}^{\omega_2}_{\partial D}[\breve{\mathbf{T}}^{m}_{n}](\boldsymbol x)\\
&+\frac{1}{\delta}\sum_{n=1}^{N}\sum_{m=-n}^{n}\frac{1}{\tau_{\mathbf{M},n}(\omega\delta)}\langle\mathbf{F}^{\omega\delta},\breve{\mathbf{M}}^{m}_{n}\rangle_{L^{2}(\partial D)^3}\mathbf{S}^{\omega_2}_{\partial D}[\breve{\mathbf{M}}^{m}_{n}](\boldsymbol x)\\
&+\frac{1}{\delta}\sum_{n=1}^{N}\sum_{m=-(n-1)}^{n-1}\frac{1}{\tau_{\mathbf{N},n}(\omega\delta)}\langle\mathbf{F}^{\omega\delta},\breve{\mathbf{N}}^{m}_{n}\rangle_{L^{2}(\partial D)^3}\mathbf{S}^{\omega_2}_{\partial D}[\breve{\mathbf{N}}^{m}_{n}](\boldsymbol x),
\end{align*}
where $\tau_{i,n}(\omega\delta),i=\mathbf{T},\mathbf{M},\mathbf{N}$ are as seen in \eqref{E:tau} and
\begin{align*}
\breve{\mathbf{T}}^{m}_{n}=\frac{1}{\delta}\mathbf{T}^{m}_{n},\quad\breve{\mathbf{M}}^{m}_{n}=\frac{1}{\delta}\mathbf{M}^{m}_{n},\quad\breve{\mathbf{N}}^{m}_{n}=\frac{1}{\delta}\mathbf{N}^{m}_{n}.
\end{align*}
\end{prop}
In view  of this proposition, our next goal aims at giving a resonance expansion for the low-frequency part of the scattered field in the time domain.

\section{Time domain approximation of the scattered field}\label{sec:4}

This section is devoted to establishing  a resonance expansion for the low-frequency part of the corresponding scattered field for the Lam\'{e} system \eqref{S:Lame1} in the time domain. There are a particular class of metamaterials that allows the presence of negative material parameters. Due to the negativity of material parameters, it leads to the non-ellipticity
of the partial differential operator which induces resonance.




\subsection{Polariton resonances}

This subsection is devoted to calculating size and frequency dependent polariton resonances.
\begin{defi}
Let $n\in\{1,2,\cdots,N\}$ for some positive integer $N$. The frequency $\omega$ is called a static polariton resonance if $|\tau_{i,n}|=0,i=\mathbf{T},\mathbf{M},\mathbf{N}$, where $\tau_{i,n}$ are given in \eqref{tau1}.
\end{defi}
\begin{defi}
Let $n\in\{1,2,\cdots,N\}$ for some positive integer $N$. The frequency $\omega$ is called a first-order corrected polariton resonance  if $|\tau_{i,n}+(c(\omega)-1)(\omega\delta)^2\varrho_{i,n}|=0,i=\mathbf{T},\mathbf{M},\mathbf{N}$, where $\varrho_{i,n}$ are as seen in \eqref{E:tau1n1}--\eqref{E:tau1n3}.
\end{defi}
\begin{rema}
More precisely, if a static polariton resonance happens, then
\begin{align}\label{SE:resonance}
\lambda_{i,n}=\frac{c(\omega)+1}{2(c(\omega)-1)},\quad i=\mathbf{T},\mathbf{M},\mathbf{N}.\tag{$\mathrm{SRE}_{i}$}
\end{align}
Furthermore, if a first-order corrected polariton resonance happens, then
\begin{align}\label{E:resonance}
-\frac{1}{2}(c(\omega)+1)+(c(\omega)-1)(\lambda_{i,n}+\omega^2\delta^2\varrho_{i,n})=0.\tag{$\mathrm{RE}_{i}$}
\end{align}
In what follows, we use the lower-case character $\omega$ for real frequencies and the upper-case character $\Omega$ for complex frequencies.
\end{rema}
Moreover, we can express $\varrho_{i,n},i=\mathbf{T},\mathbf{M},\mathbf{N}$ as
\begin{align*}
&\varrho_{\mathbf{T},n}=\big\langle\big(-\frac{1}{2}\mathbf{I}+\mathbf{K}^{*}_{\partial B}\big)\mathbf{S}^{-1}_{\partial B}\mathbf{P}_{\partial B}[\tilde{\mathbf{T}}^{m}_{n}],\tilde{\mathbf{T}}^{m}_{n}\big\rangle_{L^{2}(\partial B)^3},\\
&\varrho_{\mathbf{M},n}=\big\langle\big(-\frac{1}{2}\mathbf{I}+\mathbf{K}^{*}_{\partial B}\big)\mathbf{S}^{-1}_{\partial B}\mathbf{P}_{\partial B}[\tilde{\mathbf{M}}^{m}_{n}],\tilde{\mathbf{M}}^{m}_{n}\big\rangle_{L^{2}(\partial B)^3},\\
&\varrho_{\mathbf{N},n}=\big\langle\big(-\frac{1}{2}\mathbf{I}+\mathbf{K}^{*}_{\partial B}\big)\mathbf{S}^{-1}_{\partial B}\mathbf{P}_{\partial B}[\tilde{\mathbf{N}}^{m}_{n}],\tilde{\mathbf{N}}^{m}_{n}\big\rangle_{L^{2}(\partial B)^3}.
\end{align*}

\begin{lemm}
One has $\varrho_{i,n}\in \mathbb{R},i=\mathbf{T},\mathbf{M},\mathbf{N}$ satisfying
\begin{align*}
&\varrho_{\mathbf{T},n}=\big(\lambda_{\mathbf{T},n}-\frac{1}{2}\big)\big(\mathbf{S}^{-1}_{\partial B}\mathbf{P}_{\partial B}[\tilde{\mathbf{T}}^{m}_{n}],\tilde{\mathbf{T}}^{m}_{n}\big\rangle_{L^{2}(\partial B)^3},\\
&\varrho_{\mathbf{M},n}=\big(\lambda_{\mathbf{M},n}-\frac{1}{2}\big)\big(\mathbf{S}^{-1}_{\partial B}\mathbf{P}_{\partial B}[\tilde{\mathbf{M}}^{m}_{n}],\tilde{\mathbf{M}}^{m}_{n}\big\rangle_{L^{2}(\partial B)^3},\\
&\varrho_{\mathbf{N},n}=\big(\lambda_{\mathbf{N},n}-\frac{1}{2}\big)\big(\mathbf{S}^{-1}_{\partial B}\mathbf{P}_{\partial B}[\tilde{\mathbf{N}}^{m}_{n}],\tilde{\mathbf{N}}^{m}_{n}\big\rangle_{L^{2}(\partial B)^3}.
\end{align*}
\end{lemm}
\begin{proof}
Note that the operators $-\frac{1}{2}\mathbf{I}+\mathbf{K}^{*}_{\partial B},\mathbf{S}^{-1}_{\partial B},\mathbf{P}_{\partial B}$ are self-adjoint. In fact, one directly verifies that
\begin{align*}
&\langle \mathbf{S}_{\partial B} [\tilde{\mathbf{f}}],\tilde{\mathbf{g}}\rangle_{L^{2}(\partial B)^3}\\
&=\int_{\partial B}\int_{\partial B}\frac{1}{\delta}\big(-\frac{\gamma_1}{4\pi}\frac{1}{|\boldsymbol X-\boldsymbol Y|}\boldsymbol{\mathcal{I}}-\frac{\gamma_2}{4\pi}\frac{(\boldsymbol X-\boldsymbol Y)(\boldsymbol X-\boldsymbol Y)^\top}{|\boldsymbol X-\boldsymbol Y|^3}\big)\tilde{\mathbf{f}}(\boldsymbol Y)\mathrm{d}\sigma(\boldsymbol Y)\cdot\overline{\tilde{\mathbf{g}}}(\boldsymbol X)\mathrm{d}\sigma(\boldsymbol X)\\
&=\int_{\partial B}\tilde{\mathbf{f}}(\boldsymbol X)\cdot\int_{\partial B}\frac{1}{\delta}\big(-\frac{\gamma_1}{4\pi}\frac{1}{|\boldsymbol Y-\boldsymbol X|}\boldsymbol{\mathcal{I}}-\frac{\gamma_2}{4\pi}\frac{(\boldsymbol Y-\boldsymbol X)(\boldsymbol Y-\boldsymbol X)^\top}{|\boldsymbol Y-\boldsymbol X|^3}\big)\overline{\tilde{\mathbf{g}}}(\boldsymbol Y)\mathrm{d}\sigma(\boldsymbol Y)\mathrm{d}\sigma(\boldsymbol X)\\
&=\langle \tilde{\mathbf{f}},\mathbf{S}_{\partial B} [\tilde{\mathbf{g}}]\rangle_{L^{2}(\partial B)^3}
\end{align*}
and
\begin{align*}
&\langle \mathbf{P}_{\partial B} [\tilde{\mathbf{f}}],\overline{\tilde{\mathbf{g}}}\rangle_{L^{2}(\partial B)^3}\\
&=\int_{\partial B}\int_{\partial B}\delta\big(\frac{1}{16\pi}\gamma_4|\boldsymbol X-\boldsymbol Y|\boldsymbol{\mathcal{I}}-\frac{1}{16\pi}\gamma_5\frac{(\boldsymbol X-\boldsymbol Y)(\boldsymbol X-\boldsymbol Y)^\top}{|\boldsymbol X-\boldsymbol Y|}\big)\tilde{\mathbf{f}}(\boldsymbol Y)\mathrm{d}\sigma(\boldsymbol Y)\cdot\overline{\tilde{\mathbf{g}}}(\boldsymbol X)\mathrm{d}\sigma(\boldsymbol X)\\
&=\int_{\partial B}\tilde{\mathbf{f}}(\boldsymbol X)\cdot\int_{\partial B}\delta\big(\frac{1}{16\pi}\gamma_4|\boldsymbol Y-\boldsymbol X|\boldsymbol{\mathcal{I}}-\frac{1}{16\pi}\gamma_5\frac{(\boldsymbol Y-\boldsymbol X)(\boldsymbol Y-\boldsymbol X)^\top}{|\boldsymbol Y-\boldsymbol X|}\big)\overline{\tilde{\mathbf{g}}}(\boldsymbol Y)\mathrm{d}\sigma(\boldsymbol Y)\mathrm{d}\sigma(\boldsymbol X)\\
&=\langle \tilde{\mathbf{f}},\mathbf{P}_{\partial B} [\tilde{\mathbf{g}}]\rangle_{L^{2}(\partial B)^3}.
\end{align*}
Hence the proof of the lemma can be readily completed.
\end{proof}


From  Lemma \ref{le:known} and Theorem \ref{spectrum}, we observe
\begin{align}\label{P:spectrum}
\left\{\begin{aligned}
&\lambda_{\mathbf{T},1}={1}/{2},\quad\lambda_{\mathbf{M},1}={1}/{2},\\
&{1}/{2}-\lambda_{i,n}>0,\quad &&2\leq n\leq N,i=\mathbf{T},\mathbf{M},\\
&{1}/{2}-\lambda_{\mathbf{N},n}>0,\quad&&1\leq n\leq N.
\end{aligned}\right.
\end{align}
In addition, we define that for $i=\mathbf{T},\mathbf{M},\mathbf{N}$,
\begin{align*}
&\Delta_{i,1}=\big(\frac{3-2(\alpha+1)(\frac{1}{2}-\lambda_{i,n})}{48\beta\delta^2\varrho_{i,n}}-\frac{(\alpha+1)^3}{216\beta^3}\big)^2+
\big(\frac{\frac{1}{2}-\lambda_{i,n}}{12\delta^2\varrho_{i,n}}-\frac{(\alpha+1)^2}{36\beta^2}\big)^3,\\
&\Delta_{i,2}=\big(\frac{\alpha(\frac{1}{2}-\lambda_{i,n})-(1+\lambda_{i,n})}{3\beta\delta^2\varrho_{i,n}}+\frac{(\alpha+1)^3}{27\beta^3}\big)^2+
\big(\frac{\frac{1}{2}-\lambda_{i,n}}{3\delta^2\varrho_{i,n}}-\frac{(\alpha+1)^2}{9\beta^2}\big)^3.
\end{align*}

The following proposition corresponds to solutions to equations \eqref{SE:resonance} and \eqref{E:resonance}.

\begin{prop}
Let $c(\Omega)$ satisfy the form as \eqref{c:omega} and $N>1$ be an integer large enough. Suppose that
\begin{align}\label{E:alpha}
\alpha>\max_{n=1,\cdots,N}\max\Big\{\frac{\frac{1}{2}+\lambda_{i,n}}{\frac{1}{2}-\lambda_{i,n}},i=\mathbf{T},\mathbf{M},\mathbf{N},\lambda_{i,n}\neq\lambda_{\mathbf{T},1},\lambda_{\mathbf{M},1}\Big\}>0,\quad \beta>0.
\end{align}
Then the solutions $\Omega_{i,n}=\Omega'_{i,n}+\mathrm{i}\Omega''_{i,n}$,  with $\Omega'_{i,n}\in\mathbb{R},\Omega''_{i,n}\in\mathbb{R}$, of \eqref{SE:resonance}, $i=\mathbf{T},\mathbf{M},\mathbf{N}$ (i.e., the 3-dimensional static polariton resonances) satisfy that for $\lambda_{i,n}\neq\lambda_{\mathbf{T},1},\lambda_{\mathbf{M},1}$ and for all $1\leq n\leq N$,
\begin{align}
&\Omega'_{i,n}=0,\quad\Omega''_{i,n}=\frac{-\alpha\big(\frac{1}{2}-\lambda_{i,n}\big)+\big(\frac{1}{2}+\lambda_{i,n}\big)}{\beta(\frac{1}{2}-\lambda_{i,n})}<0,\quad i=\mathbf{T},\mathbf{M},\mathbf{N}.\label{S:static}
\end{align}
In addition, for $\delta\ll1$, the solutions $\Omega_{i,n}(\delta)=\Omega'_{i,n}+\mathrm{i}\Omega''_{i,n}$,  with $\Omega'_{i,n}\in\mathbb{R},\Omega''_{i,n}\in\mathbb{R}$, of \eqref{E:resonance}, $i=\mathbf{T},\mathbf{M},\mathbf{N}$ (i.e., the 3-dimensional first-order corrected polariton resonances) satisfy that for all $1\leq n\leq N,i=\mathbf{T},\mathbf{M},\mathbf{N}$,

$\mathrm{(i)}$ If $\varrho_{i,n}=0$, then for $\lambda_{i,n}\neq\lambda_{\mathbf{T},1},\lambda_{\mathbf{M},1}$,
\begin{align*}
\Omega'_{i,n}=0,\quad\Omega''_{i,n}=\frac{-\alpha\big(\frac{1}{2}-\lambda_{i,n}\big)+\big(\frac{1}{2}+\lambda_{i,n}\big)}{\beta(\frac{1}{2}-\lambda_{i,n})}.
\end{align*}

$\mathrm{(ii)}$ If $\varrho_{i,n}>0$ or $\varrho_{i,n}<0$ and $\lambda_{i,n}=\lambda_{\mathbf{T},1},\lambda_{\mathbf{M},1}$, then
\begin{align*}
\Omega_{i,n,k}(\delta)=\Omega'_{i,n,k}+\mathrm{i}\Omega''_{i,n,k},\quad k=1,2,3,
\end{align*}
where
\begin{align}
\Omega'_{i,n,1}=&\bigg(3\bigg(\sqrt[3]{-\frac{3-2(\alpha+1)(\frac{1}{2}-\lambda_{i,n})}{48\beta\delta^2\varrho_{i,n}}+\frac{(\alpha+1)^3}{216\beta^3}+\sqrt{\Delta_{i,1}}}\nonumber\\
&+\sqrt[3]{-\frac{3-2(\alpha+1)(\frac{1}{2}-\lambda_{i,n})}{48\beta\delta^2\varrho_{i,n}}+\frac{(\alpha+1)^3}{216\beta^3}-\sqrt{\Delta_{i,1}}}-\frac{\alpha+1}{3\beta}\bigg)^2\nonumber\\
&+\frac{2(\alpha+1)}{\beta}\bigg(\sqrt[3]{-\frac{3-2(\alpha+1)(\frac{1}{2}-\lambda_{i,n})}{48\beta\delta^2\varrho_{i,n}}+\frac{(\alpha+1)^3}{216\beta^3}+\sqrt{\Delta_{i,1}}}\nonumber\\
&+\sqrt[3]{-\frac{3-2(\alpha+1)(\frac{1}{2}-\lambda_{i,n})}{48\beta\delta^2\varrho_{i,n}}+\frac{(\alpha+1)^3}{216\beta^3}-\sqrt{\Delta_{i,1}}}-\frac{\alpha+1}{3\beta}\bigg)+\frac{\frac{1}{2}-\lambda_{i,n}}{\delta^2\varrho_{i,n}}\bigg)^{\frac{1}{2}},\label{E:omega11}\\
\Omega'_{i,n,2}=&-\bigg(3\bigg(\sqrt[3]{-\frac{3-2(\alpha+1)(\frac{1}{2}-\lambda_{i,n})}{48\beta\delta^2\varrho_{i,n}}+\frac{(\alpha+1)^3}{216\beta^3}+\sqrt{\Delta_{i,1}}}\nonumber\\
&+\sqrt[3]{-\frac{3-2(\alpha+1)(\frac{1}{2}-\lambda_{i,n})}{48\beta\delta^2\varrho_{i,n}}+\frac{(\alpha+1)^3}{216\beta^3}-\sqrt{\Delta_{i,1}}}-\frac{\alpha+1}{3\beta}\bigg)^2\nonumber\\
&+\frac{2(\alpha+1)}{\beta}\bigg(\sqrt[3]{-\frac{3-2(\alpha+1)(\frac{1}{2}-\lambda_{i,n})}{48\beta\delta^2\varrho_{i,n}}+\frac{(\alpha+1)^3}{216\beta^3}+\sqrt{\Delta_{i,1}}}\nonumber\\
&+\sqrt[3]{-\frac{3-2(\alpha+1)(\frac{1}{2}-\lambda_{i,n})}{48\beta\delta^2\varrho_{i,n}}+\frac{(\alpha+1)^3}{216\beta^3}-\sqrt{\Delta_{i,1}}}-\frac{\alpha+1}{3\beta}\bigg)+\frac{\frac{1}{2}-\lambda_{i,n}}{\delta^2\varrho_{i,n}}\bigg)^{\frac{1}{2}},\label{E:omega22}\\
\Omega''_{i,n,1}=&\Omega''_{i,n,2}=\sqrt[3]{-\frac{3-2(\alpha+1)(\frac{1}{2}-\lambda_{i,n})}{48\beta\delta^2\varrho_{i,n}}
+\frac{(\alpha+1)^3}{216\beta^3}+\sqrt{\Delta_{i,1}}}\nonumber\\
&\qquad\quad+\sqrt[3]{-\frac{3-2(\alpha+1)(\frac{1}{2}-\lambda_{i,n})}{48\beta\delta^2\varrho_{i,n}}+\frac{(\alpha+1)^3}{216\beta^3}-\sqrt{\Delta_{i,1}}}-\frac{\alpha+1}{3\beta},\label{E2:omega}\\
\Omega'_{i,n,3}=&0,\nonumber\\
\Omega''_{i,n,3}=&\sqrt[3]{-\frac{\alpha(\frac{1}{2}-\lambda_{i,n})-(1+\lambda_{i,n})}{3\beta\delta^2\varrho_{i,n}}-\frac{(\alpha+1)^3}{27\beta^3}+\sqrt{\Delta_{i,2}}}\nonumber\\
&+\sqrt[3]{-\frac{\alpha(\frac{1}{2}-\lambda_{i,n})-(1+\lambda_{i,n})}{3\beta\delta^2\varrho_{i,n}}-\frac{(\alpha+1)^3}{27\beta^3}-\sqrt{\Delta_{i,2}}}-\frac{\alpha+1}{3\beta}.\label{E22:omega}
\end{align}

$\mathrm{(iii)}$ If $\varrho_{i,n}<0$ and $\lambda_{i,n}\neq\lambda_{\mathbf{T},1},\lambda_{\mathbf{M},1}$,  then
\begin{align*}
\Omega_{i,n,k}(\delta)=\Omega'_{i,n,k}+\mathrm{i}\Omega''_{i,n,k},\quad k=1,2,3,
\end{align*}
where $\Omega'_{i,n,1}=\Omega'_{i,n,2}=\Omega'_{i,n,3}=0$ and
\begin{align}
\Omega''_{i,n,1}=&2\sqrt{-\frac{\frac{1}{2}-\lambda_{i,n}}{3\delta^2\varrho_{i,n}}+\frac{(\alpha+1)^2}{9\beta^2}}\cos\big(\frac{\theta'_i}{3}\big)-\frac{\alpha+1}{3\beta},\label{E:omega333}\\
\Omega''_{i,n,2}=&2\sqrt{-\frac{\frac{1}{2}-\lambda_{i,n}}{3\delta^2\varrho_{i,n}}+\frac{(\alpha+1)^2}{9\beta^2}}\cos\big(\frac{\theta'_i+2\pi}{3}\big)-\frac{\alpha+1}{3\beta},\label{E:omega444}\\
\Omega''_{i,n,3}=&2\sqrt{-\frac{\frac{1}{2}-\lambda_{i,n}}{3\delta^2\varrho_{i,n}}+\frac{(\alpha+1)^2}{9\beta^2}}\cos\big(\frac{\theta'_i+4\pi}{3}\big)-\frac{\alpha+1}{3\beta},\label{E:omega555}
\end{align}
where $\theta'_i=\arccos\bigg(-\frac{\frac{\alpha(\frac{1}{2}-\lambda_{i,n})-(1+\lambda_{i,n})}{3\beta\delta^2\varrho_{i,n}}+\frac{(\alpha+1)^3}{27\beta^3}}{\big(-\frac{\frac{1}{2}-\lambda_{i,n}}{3\delta^2\varrho_{i,n}}+\frac{(\alpha+1)^2}{9\beta^2}\big)^{\frac{3}{2}}}\bigg)$.
\end{prop}
\begin{proof}
Clearly, equations  \eqref{SE:resonance}, $i=\mathbf{T},\mathbf{M},\mathbf{N}$ are equivalent to
\begin{align*}
\left\{ \begin{aligned}
&\beta(\frac{1}{2}-\lambda_{i,n})\Omega''_{i,n}+\alpha(\frac{1}{2}-\lambda_{i,n})-(\frac{1}{2}+\lambda_{i,n})=0,\\
&\beta(\frac{1}{2}-\lambda_{i,n})\Omega'_{i,n}=0.
\end{aligned}
\right.
\end{align*}
Then it follows from \eqref{P:spectrum} that $\Omega_{i,n}$ satisfying \eqref{S:static} are the solutions of  the above system.

It remains to study equations  \eqref{E:resonance}, $i=\mathbf{T},\mathbf{M},\mathbf{N}$ which are equivalent to
\begin{align}
\left\{ \begin{aligned}\label{E:resnance2}
\beta\delta^2\varrho_{i,n}(\Omega''_{i,n})^3-3\beta\delta^2\varrho_{i,n}(\Omega'_{i,n})^2\Omega''_{i,n}
-(\alpha+1)\delta^2\varrho_{i,n}((\Omega'_{i,n})^2-(\Omega''_{i,n})^2)
+\beta(\frac{1}{2}-\lambda_{i,n})\Omega''_{i,n}\\+\alpha(\frac{1}{2}-\lambda_{i,n})-(\frac{1}{2}+\lambda_{i,n})=0,\\
\beta\delta^2\varrho_{i,n}(\Omega'_{i,n})^3-3\beta\delta^2\varrho_{i,n}(\Omega''_{i,n})^2\Omega'_{i,n}-2(\alpha+1)\delta^2\varrho_{i,n}\Omega''_{i,n}\Omega'_{i,n}-\beta(\frac{1}{2}-\lambda_{i,n})\Omega'_{i,n}=0.
\end{aligned}\right.\tag{$\mathrm{RE}'_{i}$}
\end{align}

The remainder of the discussion is divided into the following three cases.

\textbf{Case 1}: $\varrho_{i,n}=0$. It is clear that the corresponding solution of  \eqref{E:resnance2} has the form as seen in \eqref{S:static}.

\textbf{Case 2}: $\varrho_{i,n}>0$ or $\varrho_{i,n}<0$ and $\lambda_{i,n}=\lambda_{\mathbf{T},1}=\lambda_{\mathbf{M},1}$ (i.e. $\lambda_{i,n} =\frac{1}{2}$). If $\Omega'_{i,n}\neq 0$, then applying the second equality in \eqref{E:resnance2} yields that
\begin{align}\label{E:Omega1}
(\Omega'_{i,n})^2=3(\Omega''_{i,n})^2+\frac{2(\alpha+1)}{\beta}\Omega''_{i,n}+\frac{\frac{1}{2}-\lambda_{i,n}}{\delta^2\varrho_{i,n}}.
\end{align}
Hence the first one in \eqref{E:resnance2} becomes
 \begin{align}\label{E:Omega2}
8\beta\delta^2\varrho_{i,n}(\Omega''_{i,n})^3+8(\alpha+1)\delta^2\varrho_{i,n}(\Omega''_{i,n})^2+
\big(2\beta(\frac{1}{2}-\lambda_{i,n})+\frac{2(\alpha+1)^2}{\beta}\delta^2\varrho_{i,n}\big)\Omega''_{i,n}
+1=0.
\end{align}
For $\delta\ll1$, one has the discriminant  $\Delta_{i,1}>0$. Consequently, only one real root $\Omega''_{i,n}$ of equation \eqref{E:Omega2} has the form as seen in \eqref{E2:omega}. Thus two real roots $\Omega'_{i,n}$ of equation  \eqref{E:Omega1} have the forms given in  \eqref{E:omega11} and \eqref{E:omega22}.

If $\Omega'_{i,n}= 0$, then applying the first term in \eqref{E:resnance2} gives that
\begin{align}\label{E:RS-zero}
\beta\delta^2\varrho_{i,n}(\Omega''_{i,n})^3+(\alpha+1)\delta^2\varrho_{i,n}(\Omega''_{i,n})^2
+\beta(\frac{1}{2}-\lambda_{i,n})\Omega''_{i,n}+\alpha(\frac{1}{2}-\lambda_{i,n})-(\frac{1}{2}+\lambda_{i,n})=0.
\end{align}
For $\delta\ll1$, one arrives at the discriminant $\Delta_{i,2}>0$. Therefore, only one real root $\Omega''_{i,n}$ of equation  \eqref{E:RS-zero} has the form given in \eqref{E22:omega}.

\textbf{Case 3}: $\varrho_{i,n}<0$ and $\lambda_{i,n}\neq\lambda_{\mathbf{T},1},\lambda_{\mathbf{M},1}$ (i.e. $\lambda_{i,n}\neq\frac{1}{2}$). Observe that for $\delta\ll1$,
\begin{align*}
\frac{1}{\delta^2\varrho_{i,n}}(\frac{1}{2}-\lambda_{i,n})-\frac{(\alpha+1)^2}{3\beta^2}<0.
\end{align*}
For $\delta\ll1$, one has $\Delta_{i,1}<0$.
Hence the three distinct real roots of equation \eqref{E:Omega2} are
\begin{align*}
\Omega''_{i,n,1}=&2\sqrt{-\frac{\frac{1}{2}-\lambda_{i,n}}{12\delta^2\varrho_{i,n}}+\frac{(\alpha+1)^2}{36\beta^2}}\cos\big(\frac{\theta_i}{3}\big)-\frac{\alpha+1}{3\beta},\\
\Omega''_{i,n,2}=&2\sqrt{-\frac{\frac{1}{2}-\lambda_{i,n}}{12\delta^2\varrho_{i,n}}+\frac{(\alpha+1)^2}{36\beta^2}}\cos\big(\frac{\theta_i+2\pi}{3}\big)-\frac{\alpha+1}{3\beta},\\
\Omega''_{i,n,3}=&2\sqrt{-\frac{\frac{1}{2}-\lambda_{i,n}}{12\delta^2\varrho_{i,n}}+\frac{(\alpha+1)^2}{36\beta^2}}\cos\big(\frac{\theta_i+4\pi}{3}\big)-\frac{\alpha+1}{3\beta},
\end{align*}
where $\theta_i=\arccos\bigg(-\frac{\frac{3-2(\alpha+1)(\frac{1}{2}-\lambda_{i,n})}{48\beta\delta^2\varrho_{i,n}}
-\frac{(\alpha+1)^3}{216\beta^3}}{\big(-\frac{\frac{1}{2}-\lambda_{i,n}}{12\delta^2\varrho_{i,n}}
+\frac{(\alpha+1)^2}{36\beta^2}\big)^{\frac{3}{2}}}\bigg)\rightarrow\frac{\pi}{2}(~{\rm as}~\delta\rightarrow0)$, which can lead to
\begin{align*}
&\frac{\theta_i}{3}\rightarrow\frac{\pi}{6},\quad\frac{\theta_i+2\pi}{3}\rightarrow\frac{5\pi}{6},\quad\frac{\theta_i+4\pi}{3}\rightarrow\frac{3\pi}{2}.
\end{align*}
Substituting $\Omega''_{i,n,k},k=1,2,3$ into \eqref{E:Omega1} yields that for $\delta\ll1$,
\begin{align*}
&(\Omega'_{i,n,1})^2
=\big(\frac{\frac{1}{2}-\lambda_{i,n}}{\delta^2\varrho_{i,n}}-\frac{(\alpha+1)^2}{3\beta^2}\big)\sin^2\big(\frac{\theta_i}{3}\big)<0,\\
&(\Omega'_{i,n,2})^2
=\big(\frac{\frac{1}{2}-\lambda_{i,n}}{\delta^2\varrho_{i,n}}-\frac{(\alpha+1)^2}{3\beta^2}\big)\sin^2\big(\frac{\theta_i+2\pi}{3}\big)<0,\\
&(\Omega'_{i,n,3})^2
=\big(\frac{\frac{1}{2}-\lambda_{i,n}}{\delta^2\varrho_{i,n}}-\frac{(\alpha+1)^2}{3\beta^2}\big)\sin^2\big(\frac{\theta_i+4\pi}{3}\big)<0.
\end{align*}
This leads to a contradiction to $(\Omega'_{i,n,k})^2\geq0,k=1,2,3$.

If $\Omega'_{i,n}= 0$, then the discriminant $\Delta_{i,2}<0$ for $\delta\ll1$.
As a result, solutions $\Omega''_{i,n,k},k=1,2,3$ to equation \eqref{E:RS-zero} satisfy \eqref{E:omega333}--\eqref{E:omega555}.

The proof is complete.
\end{proof}

\begin{defi}
Define the resonance radius as
\begin{align*}
\mathscr{R}(\delta):=\max_{n=1,2,\cdots,N}\max\left\{|\Omega'_{i,n,1}|,|\Omega''_{i,n,1}|,|\Omega'_{i,n,2}|,|\Omega''_{i,n,2}|,
|\Omega'_{i,n,3}|,|\Omega''_{i,n,3}|,i=\mathbf{T},\mathbf{M},\mathbf{N}\right\},
\end{align*}
where, for simplicity, we write for $\varrho_{i,n}=0$,
\begin{align*}
&\Omega_{i,1,1}=\Omega_{i,1,2}=\Omega_{i,1,3}:=0,&&\text{if}\quad i=\mathbf{T},\mathbf{M},\\
&\Omega_{i,n,1}=\Omega_{i,n,2}=\Omega_{i,n,3}:=\Omega_{i,n},&&\text{if}\quad2\leq n\leq N,i=\mathbf{T},\mathbf{M},\\
&\Omega_{\mathbf{N},n,1}=\Omega_{\mathbf{N},n,2}=\Omega_{\mathbf{N},n,3}:=\Omega_{\mathbf{N},n},&&\text{if}\quad1\leq n\leq N.
\end{align*}
\end{defi}

\begin{rema}
In the proof, we require that $\omega\delta\ll1$.
This resonance radius provides our method a range of validity.  We compute resonant frequencies in a static regime and in a perturbative regime, respectively. In order to ensure that the largest polariton frequency lies in a region that is still considered as low-frequency for a quasiparticle of size $\delta$, we need to verify
\begin{align*}
\mathscr {R}(\delta)\delta<{1}/{2},
\end{align*}
If we pick the size $\mathscr {R}(\delta)\delta>1$, then the largest resonant frequency might not satisfy  $\omega \delta<1/2$. This leads to the fact that the method is not self-consistent.

Moreover, in view of \eqref{E:omega11}--\eqref{E:omega555}, we can see that for $\rho_{i,n}\neq0$,
\begin{align*}
\Omega'_{i,n,k}=\mathcal{O}(\frac{1}{\delta})\quad\text{or}\quad\Omega''_{i,n,k}=\mathcal{O}(\frac{1}{\delta}).
\end{align*}
This means that $\mathscr {R}(\delta)=\mathcal{O}(\frac{1}{\delta})$. Based on the above argument, if $\varrho_{i,n}\neq0$, then the method may be not self-consistent. In other words, if $c(\omega)$ satisfies the form as seen  in \eqref{c:omega}, then we just consider static polariton resonances, i.e., equalities \eqref{SE:resonance}, $\,i=\mathbf{T},\mathbf{M},\mathbf{N} $ hold. In particular, the resonance radius corresponding to static polariton resonances is given by
\begin{align*}
\mathscr{R}=&\max_{n=1,2,\cdots,N}\max\Big\{|\Omega'_{i,n}|,|\Omega''_{i,n}|,i=\mathbf{T},\mathbf{M},\mathbf{N},\lambda_{i,n}\neq\lambda_{\mathbf{T},1},\lambda_{\mathbf{M},1}\Big\}\\
=&\max_{n=1,2,\cdots,N}\max\Big\{\frac{\alpha\big(\frac{1}{2}-\lambda_{i,n}\big)-\big(\frac{1}{2}+\lambda_{i,n}\big)}{\beta(\frac{1}{2}-\lambda_{i,n})},i=\mathbf{T},\mathbf{M},\mathbf{N},\lambda_{i,n}\neq\lambda_{\mathbf{T},1},\lambda_{\mathbf{M},1}\Big\}.
\end{align*}

In addition, it is noticed that for $i=\mathbf{T},\mathbf{M},\mathbf{N},n=1,\cdots,N$,  every $\Omega_{i,n}$ is a simple pole of
\begin{align*}
\Omega\mapsto\frac{1}{-\frac{1}{2}(c(\Omega)+1)+(c(\Omega)-1)\lambda_{i,n}}.
\end{align*}
\end{rema}


\subsection{Scattered field for the Lam\'{e} system in the time domain}

The aim of this subsection is to establish a resonance expansion for the low-frequency part of the corresponding scattered field associated with system \eqref{S:Lame1} in the time domain.

For a fixed $\delta$, we can pick an excitation signal such that most of the frequency content is in the low frequencies but large enough to excite  polariton resonances. Recall that  $f:\omega\rightarrow f(\omega)$ is the Fourier transform of $\hat{f}$.
 We can pick $\eta_1\ll1,\eta_2\ll1$ and $\rho=\mathscr{R}>0$ such that
\begin{align}\label{E:delta}
\int_{\mathbb{R}\setminus[-\rho,\rho]}|f(\omega)|^2\mathrm{d}\omega\leq\eta_1,\quad\rho\delta\leq\eta_2.
\end{align}
\begin{rema}
In view of \eqref{E:delta}, the maximum size $\delta_{max}$ of the quasiparticle is taken as $\eta_2/\rho$.
\end{rema}

By means of \eqref{Fundamental2} and \eqref{E:incident}, the incident field in time domain is given by
\begin{align*}
\hat{\mathbf{u}}^{\mathrm{in}}(\boldsymbol x,t)=&\int_{\mathbb{R}}f(\omega)\boldsymbol\Gamma^{\omega}(\boldsymbol x,\mathbf{s})\mathbf{p}e^{-\mathrm{i}\omega t}\mathrm{d}\omega\\
=&-\int_{\mathbb{R}}f(\omega)\frac{e^{-\mathrm{i}\omega (t-\frac{1}{c_s}|\boldsymbol x-\mathbf{s}|)}}{4\pi\mu|\boldsymbol x-\mathbf{s}|}\mathrm{d}\omega\mathbf{p}
-\nabla\nabla\big(\int_{\mathbb{R}}\frac{f(\omega)}{\omega^2}\frac{e^{-{\mathrm{i}\omega(t-\frac{1}{c_s}|\boldsymbol x-\mathbf{s}|)}}}{4\pi|\boldsymbol x-\mathbf{s}|}\mathrm{d}\omega\big)\mathbf{p}\\
&+\nabla\nabla\big(\int_{\mathbb{R}}\frac{f(\omega)}{\omega^2}\frac{e^{-\mathrm{i}\omega(t-\frac{1}{c_p}|\boldsymbol x-\mathbf{s}|)}
}{4\pi|\boldsymbol x-\mathbf{s}|}\mathrm{d}\omega\big)\mathbf{p}\\
=&-\frac{\hat{f}(t-\frac{1}{c_s}|\boldsymbol x-\mathbf{s}|)}{4\pi\mu|\boldsymbol x-\mathbf{s}|}\mathbf{p}-\nabla\nabla\big(\int_{\mathbb{R}}\frac{f(\omega)}{\omega^2}\frac{e^{-{\mathrm{i}\omega(t-\frac{1}{c_s}|\boldsymbol x-\mathbf{s}|)}}}{4\pi|\boldsymbol x-\mathbf{s}|}\mathrm{d}\omega\big)\mathbf{p}\\
&+\nabla\nabla\big(\int_{\mathbb{R}}\frac{f(\omega)}{\omega^2}\frac{e^{-\mathrm{i}\omega(t-\frac{1}{c_p}|\boldsymbol x-\mathbf{s}|)}
}{4\pi|\boldsymbol x-\mathbf{s}|}\mathrm{d}\omega\big)\mathbf{p}.
\end{align*}

In the next lemma we establish  the following asymptotic expansion for the incident field.

\begin{lemm}\label{le:Ftaylor1}
If the quasiparticle's size $\delta\leq\eta_2/\rho$, 
then $\mathbf{F}^{\omega\delta}$ defined on $\partial D$ (see \eqref{tildeF}) is expressed as
\begin{align*}
\mathbf{F}^{\omega\delta}(\boldsymbol x)=&\delta\lambda f(\omega)(1-c(\omega))\sum^3_{j=1}\sum^3_{i=1}\partial_{i}\Gamma^{\omega}_{ij}(\boldsymbol z,\mathbf{s})p_{j}\boldsymbol\nu_{\boldsymbol x}\nonumber\\
&+\delta\mu f(\omega)(1-c(\omega))\big( \boldsymbol{\mathcal{D}}(\boldsymbol z,\mathbf{s})+\boldsymbol{\mathcal{D}}(\boldsymbol z,\mathbf{s})^\top\big)\boldsymbol\nu_{\boldsymbol x}+\mathcal{O}((\omega\delta)^2)+\mathcal{O}(\delta^2),
\end{align*}
where
$\boldsymbol{\mathcal{D}}(\boldsymbol z,\mathbf{s})=(\mathcal{D}_{ij})^3_{i,j=1}$ is a $3\times3$ matrix with entries:
\begin{align*}
\mathcal{D}_{ij}=\partial_{j}\Gamma^{\omega}_{i1}(\boldsymbol z,\mathbf{s})p_1+\partial_{j}\Gamma^{\omega}_{i2}(\boldsymbol z,\mathbf{s})p_2+\partial_{j}\Gamma^{\omega}_{i3}(\boldsymbol z,\mathbf{s})p_3.
\end{align*}
Here, $p_j$ is the $j$-th component of $\mathbf{p}$.
\end{lemm}
\begin{proof}
It is clear that
\begin{align*}
\tilde{\mathbf{F}}^{\omega\delta}(\boldsymbol X)=&\frac{\partial}{\partial\boldsymbol\nu_{\boldsymbol X}}\mathbf{u}^{\mathrm{in}}(\boldsymbol z+\delta \boldsymbol X)-c(\omega)\big(-\frac{1}{2}\mathbf{I}+\mathbf{K}^{\omega_1\delta,*}_{\partial B}\big)(\mathbf{S}^{\omega_1\delta}_{\partial B})^{-1}\mathbf{u}^{\mathrm{in}}(\boldsymbol z+\delta \boldsymbol X)\\
=&f(\omega)\frac{\partial}{\partial\boldsymbol\nu_{\boldsymbol X}}\boldsymbol\Gamma^{\omega}(\boldsymbol z+\delta \boldsymbol X,\mathbf{s})\mathbf{p}-f(\omega)c(\omega)\big(-\frac{1}{2}\mathbf{I}+\mathbf{K}^{\omega_1\delta,*}_{\partial B}\big)(\mathbf{S}^{\omega_1\delta}_{\partial B})^{-1}\boldsymbol\Gamma^{\omega}(\boldsymbol z+\delta \boldsymbol X,\mathbf{s})\mathbf{p}.
\end{align*}
Applying Taylor's expansion yields that
\begin{align}\label{E:Taylor}
\boldsymbol\Gamma^{\omega}(\boldsymbol z+\delta \boldsymbol X,\mathbf{s})\mathbf{p}\mid_{\boldsymbol X\in\partial B}=&\boldsymbol\Gamma^{\omega}(\boldsymbol z,\mathbf{s})\mathbf{p}+\delta X_1 \partial_{1}\boldsymbol\Gamma^{\omega}(\boldsymbol z,\mathbf{s})\mathbf{p}+\delta X_2 \partial_{2}\boldsymbol\Gamma^{\omega}(\boldsymbol z,\mathbf{s})\mathbf{p}\nonumber\\
&+\delta X_3 \partial_{3}\boldsymbol\Gamma^{\omega}(\boldsymbol z,\mathbf{s})\mathbf{p}+\mathcal{O}(\delta^2).
\end{align}
Therefore, for $\boldsymbol\Gamma^{\omega}(\boldsymbol z+\delta \boldsymbol X,\mathbf{s})=(\Gamma^{\omega}_{ij})^3_{i,j=1}$,
\begin{align*}
\frac{\partial}{\partial\boldsymbol\nu_{\boldsymbol X}}\boldsymbol\Gamma^{\omega}(\boldsymbol z+\delta \boldsymbol X,\mathbf{s})\mathbf{p}=&\lambda(\nabla\cdot(\boldsymbol\Gamma^{\omega}(\boldsymbol z+\delta \boldsymbol X,\mathbf{s})\mathbf{p}))\boldsymbol\nu_{\boldsymbol X}+2\mu(\nabla^s(\boldsymbol\Gamma^{\omega}(\boldsymbol z+\delta \boldsymbol X,\mathbf{s})\mathbf{p}))\boldsymbol\nu_{\boldsymbol X}\\
=&\lambda\big(\delta\sum^3_{j=1}\sum^3_{i=1}\partial_{{i}}\Gamma^{\omega}_{ij}(\boldsymbol z,\mathbf{s})p_{j}+\mathcal{O}(\delta^2)\big)\boldsymbol\nu_{\boldsymbol X}\\
&+\mu\big(\delta( \boldsymbol{\mathcal{D}}(\boldsymbol z,\mathbf{s})+\boldsymbol{\mathcal{D}}(\boldsymbol z,\mathbf{s})^\top)+\mathcal{O}(\delta^2)\big)\boldsymbol\nu_{\boldsymbol X}\\
=&\delta\lambda\sum^3_{j=1}\sum^3_{i=1}\partial_{{i}}\Gamma^{\omega}_{ij}(\boldsymbol z,\mathbf{s})p_{j}\boldsymbol\nu_{\boldsymbol X}+\delta\mu\big( \boldsymbol{\mathcal{D}}(\boldsymbol z,\mathbf{s})+\boldsymbol{\mathcal{D}}(\boldsymbol z,\mathbf{s})^\top\big)\boldsymbol\nu_{\boldsymbol X}+\mathcal{O}(\delta^2).
\end{align*}

On the other hand, by Lemmas \ref{le:K-series} and \ref{le:IS-series}, one has
\begin{align*}
&(\mathbf{S}^{\omega_1\delta}_{\partial B})^{-1}
=\mathbf{S}^{-1}_{\partial B}+\omega_1\delta\mathbf{R}_{\partial B,1}+(\omega_1\delta)^2\mathbf{P}_{\partial B,1}+\mathcal{O}((\omega\delta)^3),\\
&\mathbf{K}^{\omega_1\delta,*}_{\partial B}=\mathbf{K}^*_{\partial B}+(\omega_1\delta)^2\mathbf{Q}_{\partial B}+\mathcal{O}((\omega\delta)^3),
\end{align*}
which in turn yield that
\begin{align*}
\big(-\frac{1}{2}\mathbf{I}+\mathbf{K}^{\omega_1\delta,*}_{\partial B}\big)(\mathbf{S}^{\omega_1\delta}_{\partial B})^{-1}=\big(-\frac{1}{2}\mathbf{I}+\mathbf{K}^{*}_{\partial B}\big)\mathbf{S}^{-1}_{\partial B}+\omega_1\delta\big(-\frac{1}{2}\mathbf{I}+\mathbf{K}^{*}_{\partial B}\big)\mathbf{R}_{\partial B,1}+\mathcal{O}((\omega\delta)^2).
\end{align*}
Note that $\left(-\frac{1}{2}\mathbf{I}+\mathbf{K}^{*}_{\partial B}\right)\mathbf{R}_{\partial B,1}[\boldsymbol\Gamma^{\omega}(\boldsymbol z+\delta \boldsymbol X,\mathbf{s})\mathbf{p}]=0$ for $\boldsymbol X\in\partial B$. Hence it follows from \eqref{E:Taylor} that
\begin{align*}
&c(\omega)\big(-\frac{1}{2}\mathbf{I}+\mathbf{K}^{\omega_1\delta,*}_{\partial B}\big)(\mathbf{S}^{\omega_1\delta}_{\partial B})^{-1}[\boldsymbol\Gamma^{\omega}(\boldsymbol z+\delta \boldsymbol X,\mathbf{s})\mathbf{p}]\\
&=c(\omega)\big(-\frac{1}{2}\mathbf{I}+\mathbf{K}^{*}_{\partial B}\big)\mathbf{S}^{-1}_{\partial B}[\boldsymbol\Gamma^{\omega}(\boldsymbol z+\delta \boldsymbol X,\mathbf{s})\mathbf{p}]+\mathcal{O}((\omega\delta)^2)\\
&=c(\omega)\big(-\frac{1}{2}\mathbf{I}+\mathbf{K}^{*}_{\partial B}\big)\mathbf{S}^{-1}_{\partial B}[\boldsymbol\Gamma^{\omega}(\boldsymbol z,\mathbf{s})\mathbf{p}]\\
&+c(\omega)\big(-\frac{1}{2}\mathbf{I}+\mathbf{K}^{*}_{\partial B}\big)\mathbf{S}^{-1}_{\partial B}[\delta X_1 \partial_{1}\boldsymbol\Gamma^{\omega}(\boldsymbol z,\mathbf{s})\mathbf{p}+\delta X_2 \partial_{2}\boldsymbol\Gamma^{\omega}(\boldsymbol z,\mathbf{s})\mathbf{p}+\delta X_3 \partial_{3}\boldsymbol\Gamma^{\omega}(\boldsymbol z,\mathbf{s})\mathbf{p}]\\
&+c(\omega)\big(-\frac{1}{2}\mathbf{I}+\mathbf{K}^{*}_{\partial B}\big)\mathbf{S}^{-1}_{\partial B}[\mathcal{O}(\delta^2)]+\mathcal{O}((\omega\delta)^2).
\end{align*}
It is straightforward to see that
\begin{align*}
c(\omega)\big(-\frac{1}{2}\mathbf{I}+\mathbf{K}^{*}_{\partial B}\big)\mathbf{S}^{-1}_{\partial B}[\boldsymbol\Gamma^{\omega}(\boldsymbol z,\mathbf{s})\mathbf{p}]=0.
\end{align*}
In addition, if we assume
\begin{align*}
\mathbf{S}^{-1}_{\partial B}[\delta X_1 \partial_{1}\boldsymbol\Gamma^{\omega}(\boldsymbol z,\mathbf{s})\mathbf{p}+\delta X_2 \partial_{2}\boldsymbol\Gamma^{\omega}(\boldsymbol z,\mathbf{s})\mathbf{p}+\delta X_3 \partial_{3}\boldsymbol\Gamma^{\omega}(\boldsymbol z,\mathbf{s})\mathbf{p}]=\boldsymbol\phi(\boldsymbol X),
\end{align*}
which leads to
\begin{align*}
\mathbf{S}_{\partial B}[\boldsymbol\phi](\boldsymbol X)=\delta X_1 \partial_{1}\boldsymbol\Gamma^{\omega}(\boldsymbol z,\mathbf{s})\mathbf{p}+\delta X_2 \partial_{2}\boldsymbol\Gamma^{\omega}(\boldsymbol z,\mathbf{s})\mathbf{p}+\delta X_3 \partial_{3}\boldsymbol\Gamma^{\omega}(\boldsymbol z,\mathbf{s})\mathbf{p},
\end{align*}
and then
\begin{align*}
\big(-\frac{1}{2}\mathbf{I}+\mathbf{K}^{*}_{\partial B}\big)[\boldsymbol\phi](\boldsymbol X)=&\frac{\partial}{\partial \boldsymbol\nu_{\boldsymbol X}}\mathbf{S}_{\partial B}[\boldsymbol\phi]\big|_{-}(\boldsymbol X)\\
=&\frac{\partial}{\partial \boldsymbol\nu_{\boldsymbol X}}(\delta X_1 \partial_{1}\boldsymbol\Gamma^{\omega}(\boldsymbol z,\mathbf{s})\mathbf{p}+\delta X_2 \partial_{2}\boldsymbol\Gamma^{\omega}(\boldsymbol z,\mathbf{s})\mathbf{p}+\delta X_3 \partial_{3}\boldsymbol\Gamma^{\omega}(\boldsymbol z,\mathbf{s})\mathbf{p})\\
=&\delta\lambda\sum^3_{j=1}\sum^3_{i=1}\partial_{{i}}\Gamma^{\omega}_{ij}(\boldsymbol z,\mathbf{s})p_{j}\boldsymbol\nu_{\boldsymbol X}+\delta\mu\big( \boldsymbol{\mathcal{D}}(\boldsymbol z,\mathbf{s})+\boldsymbol{\mathcal{D}}(\boldsymbol z,\mathbf{s})^\top\big)\boldsymbol\nu_{\boldsymbol X}+\mathcal{O}(\delta^2).
\end{align*}
As a consequence,
\begin{align*}
\tilde{\mathbf{F}}^{\omega\delta}(\boldsymbol X)=&\delta\lambda f(\omega)(1-c(\omega))\sum^3_{j=1}\sum^3_{i=1}\partial_{{i}}\Gamma^{\omega}_{ij}(\boldsymbol z,\mathbf{s})p_{j}\boldsymbol\nu_{\boldsymbol X}\\
&+\delta\mu f(\omega)(1-c(\omega))\big( \boldsymbol{\mathcal{D}}(\boldsymbol z,\mathbf{s})+\boldsymbol{\mathcal{D}}(\boldsymbol z,\mathbf{s})^\top\big)\boldsymbol\nu_{\boldsymbol X}++\mathcal{O}((\omega\delta)^2)+\mathcal{O}(\delta^2),
\end{align*}
which implies the conclusion of the lemma.

The proof is complete.
\end{proof}

From Lemma \ref{le:Ftaylor1}, we consider that for $\delta\leq\eta_2/\rho$,
\begin{align}\label{E:FFF}
\mathbf{F}^{\omega\delta}(\boldsymbol x)=&\delta\lambda f(\omega)(1-c(\omega))\sum^3_{j=1}\sum^3_{i=1}\partial_{i}\Gamma^{\omega}_{ij}(\boldsymbol z,\mathbf{s})p_{j}\boldsymbol\nu_{\boldsymbol x}\nonumber\\
&+\delta\mu f(\omega)(1-c(\omega))\big( \boldsymbol{\mathcal{D}}(\boldsymbol z,\mathbf{s})+\boldsymbol{\mathcal{D}}(\boldsymbol z,\mathbf{s})^\top\big)\boldsymbol\nu_{\boldsymbol x}.
\end{align}
Recall that $\boldsymbol z$ is the center of the resonator and $\delta $ is its radius. We define
\begin{align*}
t^{-}_0:=&\frac{|\boldsymbol z-\mathbf{s}|}{c_p}+\frac{|\boldsymbol x-\boldsymbol z|}{c_p}-\frac{\delta}{c_p}
-C_1,\\
t^{+}_0:=&\frac{|\boldsymbol z-\mathbf{s}|}{c_s}+\frac{|\boldsymbol x-\boldsymbol z|}{c_s}+\frac{\delta}{c_s}
+C_1,
\end{align*}
where $\frac{|\boldsymbol z-\mathbf{s}|+|\boldsymbol x-\boldsymbol z|}{c_i},i=s,p$ stand for the time it takes the wide-band signal to reach first the scatterer and then to the observation point $\boldsymbol x$. The term $- \frac{\delta}{c_p},\frac{\delta}{c_s}$ account for the maximal timespan spent inside the quasiparticle.

 For $\rho>0$, the truncated inverse Fourier transform of the scattered field $\mathbf{u}^{\mathrm{sca}}$ is given by
\begin{align*}
\mathbf{P}_{\rho}[\mathbf{u}^{\mathrm{sca}}](\boldsymbol x,t):=\int^{\rho}_{-\rho}\mathbf{u}^{\mathrm{sca}}(\boldsymbol x,\omega)e^{-\mathrm{i}\omega t}\mathrm{d}\omega.
\end{align*}
Moreover, define
\begin{align*}
\mathbf{e}_{\mathbf{T},n}(\boldsymbol x)=\mathbf{S}^{\mathrm{i}\Omega''_{\mathbf{T},n}}_{\partial D}[\breve{\mathbf{T}}^{m}_{n}](\boldsymbol x),\quad\mathbf{e}_{\mathbf{M},n}(\boldsymbol x)=\mathbf{S}^{\mathrm{i}\Omega''_{\mathbf{M},n}}_{\partial D}[\breve{\mathbf{M}}^{m}_{n}](\boldsymbol x),\quad\mathbf{e}_{\mathbf{N},n}(\boldsymbol x)=\mathbf{S}^{\mathrm{i}\Omega''_{\mathbf{N},n}}_{\partial D}[\breve{\mathbf{N}}^{m}_{n}](\boldsymbol x),
\end{align*}
where $\Omega''_{i,n},i=\mathbf{T},\mathbf{M},\mathbf{N}$ are given in \eqref{S:static}.

The following theorem corresponds to the expression of the scattered field in the time domain.

\begin{theo}\label{th:scattered}
Let $\alpha,\beta$ satisfy conditions \eqref{E:alpha}. For an integer $N>1$ large enough, if the incident wave $\mathbf{F}^{\omega\delta}$ satisfies \eqref{E:FFF}, then for $\delta\leq\eta_2/\rho$, there exists an integer $M\geq1$ such that  the truncated scattered field has the following form in the time domain for $\boldsymbol x\in\mathbb{R}^3\backslash \overline{D},t\leq t^{-}_0$,
\begin{align}\label{E:scattered}
\mathbf{P}_{\rho}[\mathbf{u}^{\mathrm{sca}}](\boldsymbol x,t)=
&\mathcal{O}\left(\delta^4\rho^{-M}\right),
\end{align}
and for $\boldsymbol x\in\mathbb{R}^3\backslash \overline{D},t\geq t^{+}_0$,
\begin{align}\label{EE:scattered}
\mathbf{P}_{\rho}[\mathbf{u}^{\mathrm{sca}}](\boldsymbol x,t)=&\sum_{n=2}^{N}\sum_{m=-n}^{n}\frac{-2\pi}{\delta\beta(\frac{1}{2}-\lambda_{\mathbf{T},n})}\langle \mathbf{F}^{\mathrm{i}\Omega''_{\mathbf{T},n}\delta},\breve{\mathbf{T}}^{m}_{n}\rangle_{L^{2}(\partial D)^3}\mathbf{e}_{\mathbf{T},n}(\boldsymbol x)e^{\Omega''_{\mathbf{T},n}t}\nonumber\\
&+\sum_{n=2}^{N}\sum_{m=-n}^{n}\frac{-2\pi}{\delta\beta(\frac{1}{2}-\lambda_{\mathbf{M},n})}\langle \mathbf{F}^{\mathrm{i}\Omega''_{\mathbf{M},n}\delta},\breve{\mathbf{M}}^{m}_{n}\rangle_{L^{2}(\partial D)^3}\mathbf{e}_{\mathbf{M},n}(\boldsymbol x)e^{\Omega''_{\mathbf{M},n}t}\nonumber\\
&+\sum_{n=1}^{N}\sum_{m=-(n-1)}^{n-1}\frac{-2\pi}{\delta\beta(\frac{1}{2}-\lambda_{\mathbf{N},n})}\langle \mathbf{F}^{\mathrm{i}\Omega''_{\mathbf{N},n}\delta},\breve{\mathbf{N}}^{m}_{n}\rangle_{L^{2}(\partial D)^3}\mathbf{e}_{\mathbf{N},n}(\boldsymbol x)e^{\Omega''_{\mathbf{N},n}t}\nonumber\\
&+\mathcal{O}\left({\delta^4\rho^{-M}}/{t}\right).
\end{align}
\end{theo}

\begin{rema}
By straightforward modifications,  we can deal with the other resonance case with the parameters $\alpha<0$ and $\beta<0$.
\end{rema}

Notice that
\begin{align}\label{Re:gamma}
\boldsymbol\Gamma^{\omega}(\boldsymbol x,\boldsymbol y) =&-e^{\mathrm{i}\frac{\omega}{c_s}|\boldsymbol x-\boldsymbol y|}\frac{\boldsymbol{\mathcal{A}}(\boldsymbol x,\boldsymbol y,\frac{\omega}{c_s})}{4\pi\omega^2|\boldsymbol x-\boldsymbol y|}
-e^{\mathrm{i}\frac{\omega}{c_p}|\boldsymbol x-\boldsymbol y|}\frac{\boldsymbol{\mathcal{B}}(\boldsymbol x,\boldsymbol y,\frac{\omega}{c_p})}{4\pi\omega^2|\boldsymbol x-\boldsymbol y|},
\end{align}
where $\boldsymbol{\mathcal{A}}(\boldsymbol x,\boldsymbol y,\frac{\omega}{c_s})$ and $\boldsymbol{\mathcal{B}}(\boldsymbol x,\boldsymbol y,\frac{\omega}{c_p})$ will be given later; see Lemma \ref{le:gamma} for details.
One obtains
\begin{align*}
&\mathbf{e}_{\mathbf{T},n}(\boldsymbol x)=\int_{\partial D}\big(e^{\frac{-\Omega''_{\mathbf{T},n}}{c_s}|\boldsymbol x-\boldsymbol y|}\frac{\boldsymbol{\mathcal{A}}(\boldsymbol x,\boldsymbol y,\frac{\mathrm{i}\Omega''_{\mathbf{T}}}{c_s})}{4\pi(\Omega''_{\mathbf{T},n})^2|\boldsymbol x-\boldsymbol y|}
+e^{-\frac{\Omega''_{\mathbf{T},n}}{c_p}|\boldsymbol x-\boldsymbol y|}\frac{\boldsymbol{\mathcal{B}}(\boldsymbol x,\boldsymbol y,\frac{\mathrm{i}\Omega''_{\mathbf{T}}}{c_p})}{4\pi(\Omega''_{\mathbf{T},n})^2|\boldsymbol x-\boldsymbol y|}\big)\breve{\mathbf{T}}^{m}_{n}(\boldsymbol y)\mathrm{d}\sigma(\boldsymbol y),\\
&\mathbf{e}_{\mathbf{M},n}(\boldsymbol x)=\int_{\partial D}\big(e^{\frac{-\Omega''_{\mathbf{M},n}}{c_s}|\boldsymbol x-\boldsymbol y|}\frac{\boldsymbol{\mathcal{A}}(\boldsymbol x,\boldsymbol y,\frac{\mathrm{i}\Omega''_{\mathbf{M}}}{c_s})}{4\pi(\Omega''_{\mathbf{M}})^2|\boldsymbol x-\boldsymbol y|}
+e^{\frac{-\Omega''_{\mathbf{M},n}}{c_p}|\boldsymbol x-\boldsymbol y|}\frac{\boldsymbol{\mathcal{B}}(\boldsymbol x,\boldsymbol y,\frac{\mathrm{i}\Omega''_{\mathbf{M}}}{c_p})}{4\pi(\Omega''_{\mathbf{M}})^2|\boldsymbol x-\boldsymbol y|}\big)\breve{\mathbf{M}}^{m}_{n}(\boldsymbol y)\mathrm{d}\sigma(\boldsymbol y),\\
&\mathbf{e}_{\mathbf{N},n}(\boldsymbol x)=\int_{\partial D}\big(e^{\frac{-\Omega''_{\mathbf{N},n}}{c_s}|\boldsymbol x-\boldsymbol y|}\frac{\boldsymbol{\mathcal{A}}(\boldsymbol x,\boldsymbol y,\frac{\mathrm{i}\Omega''_{\mathbf{N}}}{c_s})}{4\pi(\Omega''_{\mathbf{N},n})^2|\boldsymbol x-\boldsymbol y|}
+e^{\frac{-\Omega''_{\mathbf{N},n}}{c_p}|\boldsymbol x-\boldsymbol y|}\frac{\boldsymbol{\mathcal{B}}(\boldsymbol x,\boldsymbol y,\frac{\mathrm{i}\Omega''_{\mathbf{N}}}{c_p})}{4\pi(\Omega''_{\mathbf{N},n})^2|\boldsymbol x-\boldsymbol y|}\big)\breve{\mathbf{N}}^{m}_{n}(\boldsymbol y)\mathrm{d}\sigma(\boldsymbol y).
\end{align*}
Since $\alpha,\beta$ satisfy \eqref{E:alpha}, we can see $\Omega''_{i,n}<0,i=\mathbf{T},\mathbf{M},\mathbf{N}$ thanks  to   \eqref{S:static}. Using the expressions of $\boldsymbol{\mathcal{A}}(\boldsymbol x,\boldsymbol y,\frac{\omega}{c_s})$ and $\boldsymbol{\mathcal{B}}(\boldsymbol x,\boldsymbol y,\frac{\omega}{c_p})$, these terms $\mathbf{e}_{i,n}(\boldsymbol x),i=\mathbf{T},\mathbf{M},\mathbf{N}$ may tend to $\infty$ as $|x|$ tends to $\infty$. Even so, let us show that no terms  in \eqref{EE:scattered} diverge.

\begin{theo}
Let $\alpha,\beta$ satisfy conditions \eqref{E:alpha} and $M,N$ be two integers. For $N>1$ large enough, if we define
\begin{align*}
\mathbf{E}_{i,n}(\boldsymbol x):=\mathbf{e}_{i,n}(\boldsymbol x)e^{\frac{\Omega''_{i,n}|\boldsymbol x-\boldsymbol z|}{c_s}},\quad i=\mathbf{T},\mathbf{M},\mathbf{N},
\end{align*}
then there exists $M\geq1$ such that for $\delta\leq\eta_2/\rho$,  the scattered field has the following form in the time domain for $\boldsymbol x\in\mathbb{R}^3\backslash \overline{D},t\leq t^{-}_0$,
\begin{align*}
\mathbf{P}_{\rho}[\mathbf{u}^{\mathrm{sca}}](\boldsymbol x,t)=\mathcal{O}\left(\delta^4\rho^{-M}\right),
\end{align*}
and for $\boldsymbol x\in\mathbb{R}^3\backslash \overline{D},t\geq t^{+}_0$,
\begin{align*}
\mathbf{P}_{\rho}[\mathbf{u}^{\mathrm{sca}}](\boldsymbol x,t)=&\sum_{n=2}^{N}\sum_{m=-n}^{n}\mathscr{C}_{\mathbf{T},n}^{m}(\mathbf{u}^{\mathbf{in}},\delta)\mathbf{E}_{\mathbf{T},n}(\boldsymbol x)e^{\Omega''_{\mathbf{T},n} (t-t_0^+)}\\
&+\sum_{n=2}^{N}\sum_{m=-n}^{n}\mathscr{C}_{\mathbf{M},n}^{m}(\mathbf{u}^{\mathbf{in}},\delta)\mathbf{E}_{\mathbf{M},n}(\boldsymbol x)e^{\Omega''_{\mathbf{M},n} (t-t_0^+)}\\
&+\sum_{n=1}^{N}\sum_{m=-(n-1)}^{n-1}\mathscr{C}_{\mathbf{N},n}^{m}(\mathbf{u}^{\mathbf{in}},\delta)\mathbf{E}_{\mathbf{N},n}(\boldsymbol x)e^{\Omega''_{\mathbf{N},n} (t-t_0^+)}+\mathcal{O}\left({\delta^4\rho^{-M}}/{t}\right),
\end{align*}
where
\begin{align*}
&\mathscr{C}_{\mathbf{T},n}^{m}(\mathbf{u}^{\mathbf{in}},\delta)=\frac{-2\pi C_{\mathbf{T},\delta}}{\delta\beta(\frac{1}{2}-\lambda_{\mathbf{T},n})}\langle \mathbf{F}^{\mathrm{i}\Omega''_{\mathbf{T},n}\delta},\breve{\mathbf{T}}^{m}_{n}\rangle_{L^{2}(\partial D)^3},\\
&\mathscr{C}_{\mathbf{M},n}^{m}(\mathbf{u}^{\mathbf{in}},\delta)=\frac{-2\pi C_{\mathbf{M},\delta}}{\delta\beta(\frac{1}{2}-\lambda_{\mathbf{M},n})}\langle \mathbf{F}^{\mathrm{i}\Omega''_{\mathbf{M},n}\delta},\breve{\mathbf{M}}^{m}_{n}\rangle_{L^{2}(\partial D)^3},\\
&\mathscr{C}_{\mathbf{N},n}^{m}(\mathbf{u}^{\mathbf{in}},\delta)=\frac{-2\pi C_{\mathbf{N},\delta}}{\delta\beta(\frac{1}{2}-\lambda_{\mathbf{N},n})}\langle \mathbf{F}^{\mathrm{i}\Omega''_{\mathbf{N},n}\delta},\breve{\mathbf{N}}^{m}_{n}\rangle_{L^{2}(\partial D)^3}.
\end{align*}
\end{theo}
\begin{proof}
Observe that for $i=\mathbf{T},\mathbf{M},\mathbf{N}$,
\begin{align*}
\mathbf{e}_{i,n}(\boldsymbol x)e^{\Omega''_{i,n} t}=&\mathbf{e}_{i,n}(\boldsymbol x)e^{\Omega''_{i,n} t^{+}_0}e^{\Omega''_{i,n} (t-t_0^+)}\\
=&\mathbf{e}_{i,n}(\boldsymbol x)e^{\Omega''_{i,n} (\frac{|\boldsymbol z-\mathbf{s}|}{c_s}+\frac{|\boldsymbol x-\boldsymbol z|}{c_s}+\frac{\delta}{c_s}
+C_1)}e^{\Omega''_{i,n} (t-t_0^+)}\\
=&C_{i,\delta}\mathbf{e}_{i,n}(\boldsymbol x)e^{\frac{\Omega''_{i,n}|\boldsymbol x-\boldsymbol z|}{c_s}}e^{\Omega''_{i,n} (t-t_0^+)},
\end{align*}
where $C_{i,\delta}=e^{\Omega''_{i,n} (\frac{|\boldsymbol z-\mathbf{s}|}{c_s}+\frac{\delta}{c_s}
+C_1)}$. Combining this with Theorem \ref{th:scattered}, we can derive the conclusion of the theorem. The proof is complete.
\end{proof}

\subsection{Proof of Theorem \ref{th:scattered}}
The rest of this section aims at completing the proof of of Theorem \ref{th:scattered}. Before  starting  the proof, we need the following lemma.
\begin{lemm}\label{le:gamma}
One has
\begin{align*}
\boldsymbol\Gamma^{\omega}(\boldsymbol x,\mathbf{s}) =&-e^{\mathrm{i}\frac{\omega}{c_s}|\boldsymbol x-\mathbf{s}|}\frac{\boldsymbol{\mathcal{A}}(\boldsymbol x,\mathbf{s},\frac{\omega}{c_s})}{4\pi\omega^2|\boldsymbol x-\mathbf{s}|}
-e^{\mathrm{i}\frac{\omega}{c_p}|\boldsymbol x-\mathbf{s}|}\frac{\boldsymbol{\mathcal{B}}(\boldsymbol x,\mathbf{s},\frac{\omega}{c_p})}{4\pi\omega^2|\boldsymbol x-\mathbf{s}|},
\end{align*}
where
$\boldsymbol{\mathcal{A}}(\boldsymbol x,\mathbf{s},\frac{\omega}{c_s})=(\mathcal{A}_{ij})^3_{i,j=1},\boldsymbol{\mathcal{B}}(\boldsymbol x,\mathbf{s},\frac{\omega}{c_p})=(\mathcal{B}_{ij})^3_{i,j=1}$ are two $3\times3$ matrices, and behave like a polynomial in $\omega$, with entries
\begin{align*}
\mathcal{A}_{ii}=&\frac{1}{|\boldsymbol x-\mathbf{s}|^4}\big(3(x_i-s_i)^2-3\mathrm{i}\frac{\omega}{c_s}(x_i-s_i)^2|\boldsymbol x-\mathbf{s}|-(\frac{\omega}{c_s})^2(x_i-s_i)^2|\boldsymbol x-\mathbf{s}|^2\\
&-|\boldsymbol x-\mathbf{s}|^2+\mathrm{i}\frac{\omega}{c_s}|\boldsymbol x-\mathbf{s}|^3
+(\frac{\omega}{c_s})^2|\boldsymbol x-\mathbf{s}|^4\big),\\
\mathcal{A}_{ij}=&\frac{(x_i-s_i)(x_j-s_j)}{|\boldsymbol x-\mathbf{s}|^4}\big(3-3\mathrm{i}\frac{\omega}{c_s}|\boldsymbol x-\mathbf{s}|
-(\frac{\omega}{c_s})^2|\boldsymbol x-\mathbf{s}|^2\big),i\neq j,
\end{align*}
and
\begin{align*}
\mathcal{B}_{ii}=&-\frac{1}{|\boldsymbol x-\mathbf{s}|^4}\big(3(x_i-s_i)^2-3\mathrm{i}\frac{\omega}{c_p}(x_i-s_i)^2|\boldsymbol x-\mathbf{s}|-(\frac{\omega}{c_p})^2(x_i-s_i)^2|\boldsymbol x-\mathbf{s}|^2\\
&-|\boldsymbol x-\mathbf{s}|^2+\mathrm{i}\frac{\omega}{c_p}|\boldsymbol x-\mathbf{s}|^3\big),\\
\mathcal{B}_{ij}=&-\frac{(x_i-s_i)(x_j-s_j)}{|\boldsymbol x-\mathbf{s}|^4}\big(3-3\mathrm{i}\frac{\omega}{c_p}|\boldsymbol x-\mathbf{s}|
-(\frac{\omega}{c_p})^2|\boldsymbol x-\mathbf{s}|^2\big),i\neq j.
\end{align*}

Moreover, one has that for $k=1,2,3$,
\begin{align*}
\partial_{k}\boldsymbol\Gamma^{\omega}(\boldsymbol x,\mathbf{s})=-e^{\mathrm{i}\frac{\omega}{c_s}|\boldsymbol x-\mathbf{s}|}\frac{\boldsymbol{\mathscr{A}}(\boldsymbol x,\mathbf{s},\frac{\omega}{c_s})}{4\pi\omega^2|\boldsymbol x-\mathbf{s}|}
-e^{\mathrm{i}\frac{\omega}{c_p}|\boldsymbol x-\mathbf{s}|}\frac{\boldsymbol{\mathscr{B}}(\boldsymbol x,\mathbf{s},\frac{\omega}{c_p})}{4\pi\omega^2|\boldsymbol x-\mathbf{s}|},
\end{align*}
where $\boldsymbol{\mathscr{A}}(\boldsymbol x,\mathbf{s},\frac{\omega}{c_s})=(\mathscr{A}_{ij})^3_{i,j=1},\boldsymbol{\mathscr{B}}(\boldsymbol x,\mathbf{s},\frac{\omega}{c_p})=(\mathscr{B}_{ij})^3_{i,j=1}$ are two $3\times3$ matrices, and behave like a polynomial in $\omega$,  with entries
\begin{align*}
\mathscr{A}_{kk}=&\frac{x_k-s_k}{|\boldsymbol x-\mathbf{s}|^6}(-15(x_k-s_k)^2+15\mathrm{i}\frac{\omega}{c_s}(x_k-s_k)^2|\boldsymbol x-\mathbf{s}|+6(\frac{\omega}{c_s})^2(x_k-s_k)^2|\boldsymbol x-\mathbf{s}|^2+9|\boldsymbol x-\mathbf{s}|^2\\
&-\mathrm{i}(\frac{\omega}{c_s})^3(x_k-s_k)^2|\boldsymbol x-\mathbf{s}|^3-9\mathrm{i}\frac{\omega}{c_s}|\boldsymbol x-\mathbf{s}|^3
-4(\frac{\omega}{c_s})^2|\boldsymbol x-\mathbf{s}|^4+\mathrm{i}(\frac{\omega}{c_s})^3|\boldsymbol x-\mathbf{s}|^5),\\
\mathscr{A}_{ii}=&\frac{x_k-s_k}{|\boldsymbol x-\mathbf{s}|^6}(-15(x_k-s_k)^2+15\mathrm{i}\frac{\omega}{c_s}(x_i-s_i)^2|\boldsymbol x-\mathbf{s}|+6(\frac{\omega}{c_s})^2(x_i-s_i)^2|\boldsymbol x-\mathbf{s}|^2+3|\boldsymbol x-\mathbf{s}|^2\\
&-\mathrm{i}(\frac{\omega}{c_s})^3(x_i-s_i)^2|\boldsymbol x-\mathbf{s}|^3-3\mathrm{i}\frac{\omega}{c_s}|\boldsymbol x-\mathbf{s}|^3
-2(\frac{\omega}{c_s})^2|\boldsymbol x-\mathbf{s}|^4+\mathrm{i}(\frac{\omega}{c_s})^3|\boldsymbol x-\mathbf{s}|^5),i\neq k,\\
\mathscr{A}_{ik}=&\frac{x_i-s_i}{|\boldsymbol x-\mathbf{s}|^4}(3-3\mathrm{i}\frac{\omega}{c_s}|\boldsymbol x-\mathbf{s}|
-(\frac{\omega}{c_s})^2|\boldsymbol x-\mathbf{s}|^2)\\
&+\frac{(x_i-s_i)(x_k-s_k)^2}{|\boldsymbol x-\mathbf{s}|^6}(-15+15\mathrm{i}\frac{\omega}{c_s}|\boldsymbol x-\mathbf{s}|
+6(\frac{\omega}{c_s})^2|\boldsymbol x-\mathbf{s}|^2-\mathrm{i}(\frac{\omega}{c_s})^3|\boldsymbol x-\mathbf{s}|^3),i\neq k,\\
\mathscr{A}_{kj}=&\frac{x_j-s_j}{|\boldsymbol x-\mathbf{s}|^4}(3-3\mathrm{i}\frac{\omega}{c_s}|\boldsymbol x-\mathbf{s}|
-(\frac{\omega}{c_s})^2|\boldsymbol x-\mathbf{s}|^2)\\
&+\frac{(x_j-s_j)(x_k-s_k)^2}{|\boldsymbol x-\mathbf{s}|^6}(-15+15\mathrm{i}\frac{\omega}{c_s}|\boldsymbol x-\mathbf{s}|
+6(\frac{\omega}{c_s})^2|\boldsymbol x-\mathbf{s}|^2-\mathrm{i}(\frac{\omega}{c_s})^3|\boldsymbol x-\mathbf{s}|^3),j\neq k, \\
\mathscr{A}_{ij}=&\frac{(x_i-s_i)(x_j-s_j)(x_k-s_k)}{|\boldsymbol x-\mathbf{s}|^6}(-15+15\mathrm{i}\frac{\omega}{c_s}|\boldsymbol x-\mathbf{s}|
+6(\frac{\omega}{c_s})^2|\boldsymbol x-\mathbf{s}|^2-\mathrm{i}(\frac{\omega}{c_s})^3|\boldsymbol x-\mathbf{s}|^3),\\
&i\neq j,i\neq k,j\neq k,
\end{align*}
and
\begin{align*}
\mathscr{B}_{kk}=&-\frac{x_k-s_k}{|\boldsymbol x-\mathbf{s}|^6}(-15(x_k-s_k)^2+15\mathrm{i}\frac{\omega}{c_p}(x_k-s_k)^2|\boldsymbol x-\mathbf{s}|+6(\frac{\omega}{c_p})^2(x_k-s_k)^2|\boldsymbol x-\mathbf{s}|^2+9|\boldsymbol x-\mathbf{s}|^2\\
&-\mathrm{i}(\frac{\omega}{c_p})^3(x_k-s_k)^2|\boldsymbol x-\mathbf{s}|^3-9\mathrm{i}\frac{\omega}{c_p}|\boldsymbol x-\mathbf{s}|^3
-3(\frac{\omega}{c_p})^2|\boldsymbol x-\mathbf{s}|^4),\\
\mathscr{B}_{ii}=&-\frac{x_k-s_k}{|\boldsymbol x-\mathbf{s}|^6}(-15(x_k-s_k)^2+15\mathrm{i}\frac{\omega}{c_p}(x_i-s_i)^2|\boldsymbol x-\mathbf{s}|+6(\frac{\omega}{c_p})^2(x_i-s_i)^2|\boldsymbol x-\mathbf{s}|^2+3|\boldsymbol x-\mathbf{s}|^2\\
&-\mathrm{i}(\frac{\omega}{c_p})^3(x_i-s_i)^2|\boldsymbol x-\mathbf{s}|^3-3\mathrm{i}\frac{\omega}{c_p}|\boldsymbol x-\mathbf{s}|^3
-(\frac{\omega}{c_p})^2|\boldsymbol x-\mathbf{s}|^4),i\neq k,\\
\mathscr{B}_{ik}=&-\frac{x_i-s_i}{|\boldsymbol x-\mathbf{s}|^4}(3-3\mathrm{i}\frac{\omega}{c_p}|\boldsymbol x-\mathbf{s}|
-(\frac{\omega}{c_p})^2|\boldsymbol x-\mathbf{s}|^2)\\
&-\frac{(x_i-s_i)(x_k-s_k)^2}{|\boldsymbol x-\mathbf{s}|^6}(-15+15\mathrm{i}\frac{\omega}{c_p}|\boldsymbol x-\mathbf{s}|
+6(\frac{\omega}{c_p})^2|\boldsymbol x-\mathbf{s}|^2-\mathrm{i}(\frac{\omega}{c_p})^3|\boldsymbol x-\mathbf{s}|^3),i\neq k, \\
\mathscr{B}_{kj}=&-\frac{x_j-s_j}{|\boldsymbol x-\mathbf{s}|^4}(3-3\mathrm{i}\frac{\omega}{c_p}|\boldsymbol x-\mathbf{s}|
-(\frac{\omega}{c_p})^2|\boldsymbol x-\mathbf{s}|^2)\\
&-\frac{(x_j-s_j)(x_k-s_k)^2}{|\boldsymbol x-\mathbf{s}|^6}(-15+15\mathrm{i}\frac{\omega}{c_p}|\boldsymbol x-\mathbf{s}|
+6(\frac{\omega}{c_p})^2|\boldsymbol x-\mathbf{s}|^2-\mathrm{i}(\frac{\omega}{c_p})^3|\boldsymbol x-\mathbf{s}|^3),j\neq k, \\
\mathscr{B}_{ij}=&-\frac{(x_i-s_i)(x_j-s_j)(x_k-s_k)}{|\boldsymbol x-\mathbf{s}|^6}(-15+15\mathrm{i}\frac{\omega}{c_p}|\boldsymbol x-\mathbf{s}|
+6(\frac{\omega}{c_p})^2|\boldsymbol x-\mathbf{s}|^2-\mathrm{i}(\frac{\omega}{c_p})^3|\boldsymbol x-\mathbf{s}|^3),\\
&i\neq j,i\neq k,j\neq k.
\end{align*}
\end{lemm}
\begin{proof}
The proof of the lemma is a direct result of formula \eqref{Fundamental2}.
\end{proof}
\begin{rema}
It is apparent from Lemma \ref{le:gamma} that \eqref{Re:gamma} holds.
\end{rema}
Let us define
\begin{align*}
&\boldsymbol{\Xi}^{m}_{\mathbf{T},n}(\boldsymbol x,\omega):=\frac{1}{\delta}
\frac{1}{\tau_{\mathbf{T},n}(\omega\delta)}\langle \mathbf{F}^{\omega\delta},\breve{\mathbf{T}}^{m}_{n}\rangle_{L^{2}(\partial D)^3}\mathbf{S}^{\omega}_{\partial D}[\breve{\mathbf{T}}^{m}_{n}](\boldsymbol x),\\
&\boldsymbol{\Xi}^{m}_{\mathbf{M},n}(\boldsymbol x,\omega):=\frac{1}{\delta}
\frac{1}{\tau_{\mathbf{M},n}(\omega\delta)}\langle \mathbf{F}^{\omega\delta},\breve{\mathbf{M}}^{m}_{n}\rangle_{L^{2}(\partial D)^3}\mathbf{S}^{\omega}_{\partial D}[\breve{\mathbf{M}}^{m}_{n}](\boldsymbol x),\\
&\boldsymbol{\Xi}^{m}_{\mathbf{N},n}(\boldsymbol x,\omega):=\frac{1}{\delta}
\frac{1}{\tau_{\mathbf{N},n}(\omega\delta)}\langle \mathbf{F}^{\omega\delta},\breve{\mathbf{N}}^{m}_{n}\rangle_{L^{2}(\partial D)^3}\mathbf{S}^{\omega}_{\partial D}[\breve{\mathbf{N}}^{m}_{n}](\boldsymbol x).
\end{align*}
Now we are devoted to verifying Theorem \ref{th:scattered}.
\begin{proof}[Proof of Theorem \ref{th:scattered}]
Since $\omega_2=\omega$ (see \eqref{Omega2}), by Proposition \ref{prop:approximation}, we recall the following spectral decomposition in the frequency domain
\begin{align*}
\mathbf{u}^{\mathrm{sca}}(\boldsymbol x,\omega)=\sum_{n=1}^{N}\sum_{m=-n}^{n}(\boldsymbol{\Xi}^{m}_{\mathbf{T},n}(\boldsymbol x,\omega)+\boldsymbol{\Xi}^{m}_{\mathbf{M},n}(\boldsymbol x,\omega))+\sum_{n=1}^{N}\sum_{m=-(n-1)}^{n-1}\boldsymbol{\Xi}^{m}_{\mathbf{N},n}(\boldsymbol x,\omega),\quad \boldsymbol x\in\mathbb{R}^3\backslash \overline{D}.
\end{align*}
It is straightforward to verify that
\begin{align*}
\int^{\rho}_{-\rho}\mathbf{u}^{\mathrm{sca}}(\boldsymbol x,\omega)e^{-\mathrm{i}\omega t}\mathrm{d}\omega
=&\sum_{n=1}^{N}\sum_{m=-n}^{n}\big(\int^{\rho}_{-\rho}\boldsymbol{\Xi}^{m}_{\mathbf{T},n}(\boldsymbol x,\omega)e^{-\mathrm{i}\omega t}\mathrm{d}\omega+\int^{\rho}_{-\rho}\boldsymbol{\Xi}^{m}_{\mathbf{M},n}(\boldsymbol x,\omega)e^{-\mathrm{i}\omega t}\mathrm{d}\omega\big)\\
&+\sum_{n=1}^{N}\sum_{m=-(n-1)}^{n-1}\int^{\rho}_{-\rho}\boldsymbol{\Xi}^{m}_{\mathbf{N},n}(\boldsymbol x,\omega)e^{-\mathrm{i}\omega t}\mathrm{d}\omega.
\end{align*}
The key is to estimate $\int^{\rho}_{-\rho}\boldsymbol{\Xi}^{m}_{i,n}(\boldsymbol x,\omega)e^{-\mathrm{i}\omega t}\mathrm{d}\omega,i=\mathbf{T},\mathbf{M},\mathbf{N}.$

We next apply the residue theorem to obtain an asymptotic expansion in the time domain. Let the integration contour $\mathcal{C}^{\pm}_{\rho}$ be a semicircular arc of radius $\rho$ in the upper $(+)$ or lower $(-)$ half-plane and $\mathcal{C}^{\pm}$ be the closed contour $\mathcal{C}^{\pm}_{\rho}\cup[-\rho,\rho]$. The integral on the closed contour is the main contribution to the scattered field by the mode $(\breve{\mathbf{T}}^{m}_{n},\breve{\mathbf{M}}^{m}_{n},\breve{\mathbf{N}}^{m}_{n})$. Clearly, for $i=\mathbf{T},\mathbf{M},\mathbf{N}$,
\begin{align*}
\int^{\rho}_{-\rho}\boldsymbol{\Xi}^{m}_{i,n}(\boldsymbol x,\omega)e^{-\mathrm{i}\omega t}\mathrm{d}\omega=\oint_{\mathcal{C}^\pm}\boldsymbol{\Xi}^{m}_{i,n}(\boldsymbol x,\Omega)e^{-\mathrm{i}\Omega t}\mathrm{d}\Omega-\int_{\mathcal{C}^\pm_\rho}\boldsymbol{\Xi}^{m}_{i,n}(\boldsymbol x,\Omega)e^{-\mathrm{i}\Omega t}\mathrm{d}\Omega.
\end{align*}
Since $\Omega''_{i,n}<0,i=\mathbf{T},\mathbf{M},\mathbf{N}$ by \eqref{S:static}, one can show that
\begin{align*}
&\oint_{\mathcal{C}^+}\boldsymbol{\Xi}^{m}_{i,n}(\boldsymbol x,\Omega)e^{-\mathrm{i}\Omega t}\mathrm{d}\Omega=0,\quad i=\mathbf{T},\mathbf{M},\mathbf{N},\\
&\oint_{\mathcal{C}^-}\boldsymbol{\Xi}^{m}_{\mathbf{T},1}(\boldsymbol x,\Omega)e^{-\mathrm{i}\Omega t}\mathrm{d}\Omega=0,\quad\oint_{\mathcal{C}^-}\boldsymbol{\Xi}^{m}_{\mathbf{M},1}(\boldsymbol x,\Omega)e^{-\mathrm{i}\Omega t}\mathrm{d}\Omega=0,\\
&\oint_{\mathcal{C}^-}\boldsymbol{\Xi}^{m}_{\mathbf{T},n}(\boldsymbol x,\Omega)e^{-\mathrm{i}\Omega t}\mathrm{d}\Omega=2\pi\mathrm{i}\mathrm{Res}(\boldsymbol{\Xi}^{m}_{\mathbf{T},n}(\boldsymbol x,\Omega)e^{-\mathrm{i}\Omega t},\Omega_{\mathbf{T},n}),\quad&&\forall 2\leq n\leq N,\forall| m|\leq n, \\
&\oint_{\mathcal{C}^-}\boldsymbol{\Xi}^{m}_{\mathbf{M},n}(\boldsymbol x,\Omega)e^{-\mathrm{i}\Omega t}\mathrm{d}\Omega=2\pi\mathrm{i}\mathrm{Res}(\boldsymbol{\Xi}^{m}_{\mathbf{M},n}(\boldsymbol x,\Omega)e^{-\mathrm{i}\Omega t},\Omega_{\mathbf{M},n}),\quad&&\forall 2\leq n\leq N,\forall| m|\leq n,\\
&\oint_{\mathcal{C}^-}\boldsymbol{\Xi}^{m}_{\mathbf{N},n}(\boldsymbol x,\Omega)e^{-\mathrm{i}\Omega t}\mathrm{d}\Omega=2\pi\mathrm{i}\mathrm{Res}(\boldsymbol{\Xi}^{m}_{\mathbf{N},n}(\boldsymbol x,\Omega)e^{-\mathrm{i}\Omega t},\Omega_{\mathbf{N},n}),\quad&&\forall 1\leq n\leq N,\forall| m|\leq n-1.
\end{align*}
As $\Omega_{i,n}$ is a simple pole of $\Omega\mapsto\frac{1}{-\frac{1}{2}(c(\Omega)+1)+(c(\Omega)-1)\lambda_{i,n}}$, we can derive that
\begin{align*}
\oint_{\mathcal{C}^-}\boldsymbol{\Xi}^{m}_{\mathbf{T},n}(\boldsymbol x,\Omega)e^{-\mathrm{i}\Omega t}\mathrm{d}\Omega=&2\pi\mathrm{i}\mathrm{Res}(\boldsymbol{\Xi}^{m}_{\mathbf{T},n}(\boldsymbol x,\Omega),\Omega_{\mathbf{T},n})e^{-\mathrm{i}\Omega_{\mathbf{T},n} t}\\
=&2\pi\mathrm{i}\frac{1}{\delta}\frac{1}{-\mathrm{i}\beta(\frac{1}{2}-\lambda_{\mathbf{T},n})}\langle \mathbf{F}^{\Omega_{\mathbf{T},n}\delta},\breve{\mathbf{T}}^{m}_{n}\rangle_{L^{2}(\partial D)^3}\mathbf{S}^{\Omega_{\mathbf{T},n}}_{\partial D}[\breve{\mathbf{T}}^{m}_{n}](\boldsymbol x)e^{-\mathrm{i}\Omega_{\mathbf{T},n} t}\\
=&-\frac{2\pi}{\delta\beta(\frac{1}{2}-\lambda_{\mathbf{T},n})}\langle \mathbf{F}^{\mathrm{i}\Omega''_{\mathbf{T},n}\delta},\breve{\mathbf{T}}^{m}_{n}\rangle_{L^{2}(\partial D)^3}\mathbf{S}^{\mathrm{i}\Omega''_{\mathbf{T},n}}_{\partial D}[\breve{\mathbf{T}}^{m}_{n}](\boldsymbol x)e^{\Omega''_{\mathbf{T},n}t}.
\end{align*}
Applying the similar technique as above yields that
\begin{align*}
&\oint_{\mathcal{C}^-}\boldsymbol{\Xi}^{m}_{\mathbf{M},n}(\boldsymbol x,\Omega)e^{-\mathrm{i}\Omega t}\mathrm{d}\Omega=-\frac{2\pi}{\delta\beta(\frac{1}{2}-\lambda_{\mathbf{M},n})}\langle \mathbf{F}^{\mathrm{i}\Omega''_{\mathbf{M},n}\delta},\breve{\mathbf{M}}^{m}_{n}\rangle_{L^{2}(\partial D)^3}\mathbf{S}^{\mathrm{i}\Omega''_{\mathbf{M},n}}_{\partial D}[\breve{\mathbf{M}}^{m}_{n}](\boldsymbol x)e^{\Omega''_{\mathbf{M},n} t},\\
&\oint_{\mathcal{C}^-}\boldsymbol{\Xi}^{m}_{\mathbf{N},n}(\boldsymbol x,\Omega)e^{-\mathrm{i}\Omega t}\mathrm{d}\Omega=-\frac{2\pi}{\delta\beta(\frac{1}{2}-\lambda_{\mathbf{N},n})}\langle \mathbf{F}^{\mathrm{i}\Omega''_{\mathbf{N},n}\delta},\breve{\mathbf{N}}^{m}_{n}\rangle_{L^{2}(\partial D)^3}\mathbf{S}^{\mathrm{i}\Omega''_{\mathbf{N},n}}_{\partial D}[\breve{\mathbf{N}}^{m}_{n}](\boldsymbol x)e^{\Omega''_{\mathbf{N},n} t}.
\end{align*}
The remainder of the proof is to estimate
\begin{align*}
\int_{\mathcal{C}^{\pm}_{\rho}}\boldsymbol{\Xi}^{m}_{i,n}(\boldsymbol x,\Omega)e^{-\mathrm{i}\Omega t}\mathrm{d}\Omega,\quad i=\mathbf{T},\mathbf{M},\mathbf{N}.
\end{align*}
We only consider  $\int_{\mathcal{C}^{\pm}_{\rho}}\boldsymbol{\Xi}^{m}_{\mathbf{T},n}(\boldsymbol x,\Omega)e^{-\mathrm{i}\Omega t}\mathrm{d}\Omega$.  Then two  integrals $\int_{\mathcal{C}^{\pm}_{\rho}}\boldsymbol{\Xi}^{m}_{i,n}(\boldsymbol x,\Omega)e^{-\mathrm{i}\Omega t}\mathrm{d}\Omega,i=\mathbf{M},\mathbf{N}$ can be estimated in a completely similar way.

From Lemma \ref{le:gamma}, we have
\begin{align}
&\sum^3_{j=1}\sum^3_{i=1}\partial_{{i}}\Gamma^{\Omega}_{ij}(\boldsymbol z,\mathbf{s})p_{j}\boldsymbol\nu_{\boldsymbol x}=-e^{\mathrm{i}\frac{\Omega}{c_s}|\boldsymbol z-\mathbf{s}|}
\frac{\mathscr{P}(\boldsymbol z,\mathbf{s},\frac{\Omega}{c_s})\boldsymbol\nu_{\boldsymbol x}}{4\pi\Omega^2|\boldsymbol z-\mathbf{s}|}-
e^{\mathrm{i}\frac{\Omega}{c_p}|\boldsymbol z-\mathbf{s}|}\frac{\mathscr{Q}(\boldsymbol z,\mathbf{s},\frac{\Omega}{c_p})\boldsymbol\nu_{\boldsymbol x}}{{4\pi\Omega^2|\boldsymbol z-\mathbf{s}|}},\label{E:F1}\\
&\big( \boldsymbol{\mathcal{D}}(\boldsymbol z,\mathbf{s})+\boldsymbol{\mathcal{D}}(\boldsymbol z,\mathbf{s})^\top\big)\boldsymbol\nu_{\boldsymbol x}=-e^{\mathrm{i}\frac{\Omega}{c_s}|\boldsymbol z-\mathbf{s}|}\frac{\boldsymbol{\mathscr{M}}(\boldsymbol z,\mathbf{s},\frac{\Omega}{c_s})\boldsymbol\nu_{\boldsymbol x}}{4\pi\Omega^2|\boldsymbol z-\mathbf{s}|}
-e^{\mathrm{i}\frac{\Omega}{c_p}|\boldsymbol z-\mathbf{s}|}\frac{\boldsymbol{\mathscr{N}}(\boldsymbol z,\mathbf{s},\frac{\Omega}{c_p})\boldsymbol\nu_{\boldsymbol x}}{4\pi\Omega^2|\boldsymbol z-\mathbf{s}|},\label{E:F2}
\end{align}
where $\mathscr{P}(\boldsymbol z,\mathbf{s},\frac{\Omega}{c_s})=\sum^3_{j=1}\sum^3_{i=1}\mathscr{P}_{ij}p_{j},\mathscr{Q}(\boldsymbol z,\mathbf{s},\frac{\Omega}{c_p})= \sum^3_{j=1}\sum^3_{i=1}\mathscr{Q}_{ij}p_{j}$ with
\begin{align*}
\mathscr{P}_{ii}=\mathscr{A}_{kk}|_{k=i,\boldsymbol x=\boldsymbol z},\quad\mathscr{P}_{ij}=\mathscr{A}_{kj}|_{k=i, \boldsymbol x=\boldsymbol z},i\neq j,\\
\mathscr{Q}_{ii}=\mathscr{B}_{kk}|_{k=i,\boldsymbol x=\boldsymbol z},\quad \mathscr{Q}_{ij}=\mathscr{B}_{kj}|_{k=i, \boldsymbol x=\boldsymbol z},i\neq j,
\end{align*}
and $\boldsymbol{\mathscr{M}}(\boldsymbol z,\mathbf{s},\frac{\Omega}{c_s})=(\mathscr{M}_{ij})^3_{i,j=1},\boldsymbol{\mathscr{N}}(\boldsymbol z,\mathbf{s},\frac{\Omega}{c_p})=(\mathscr{N}_{ij})^3_{i,j=1}$ are $3\times3$ matrix whose entries are composed of $\mathscr{A}_{ij}|_{\boldsymbol x=\boldsymbol z},\mathscr{B}_{ij}|_{\boldsymbol x=\boldsymbol z}$ and $p_j$, $i,j=1,2,3$.
Then it follows from  \eqref{E:FFF}, \eqref{E:F1} and \eqref{E:F2} that
\begin{align*}
\boldsymbol{\Xi}^{m}_{\mathbf{T},n}(\boldsymbol x,\Omega)=&\lambda f(\Omega)\frac{1-c(\Omega)}{\tau_{\mathbf{T},n}(\Omega\delta)}\int_{\partial D}\big(e^{\mathrm{i}\frac{\Omega}{c_s}|\boldsymbol z-\mathbf{s}|}\frac{\mathscr{P}(\boldsymbol z,\mathbf{s},\frac{\Omega}{c_s})\boldsymbol\nu_{\boldsymbol x}}{4\pi\Omega^2|\boldsymbol z-\mathbf{s}|}
+e^{\mathrm{i}\frac{\Omega}{c_p}|\boldsymbol z-\mathbf{s}|}\frac{\mathscr{Q}(\boldsymbol z,\mathbf{s},\frac{\Omega}{c_p})\boldsymbol\nu_{\boldsymbol x}}{4\pi\Omega^2|\boldsymbol z-\mathbf{s}|}\big)\cdot \breve{\mathbf{T}}^{m}_{n}(\boldsymbol x)\mathbf{d}\sigma(\boldsymbol x) \\
&\quad\quad\quad\quad\quad\quad\times\int_{\partial D}\big(e^{\mathrm{i}\frac{\Omega}{c_s}|\boldsymbol x-\boldsymbol y|}\frac{\boldsymbol{\mathcal{A}}(\boldsymbol x,\boldsymbol y,\frac{\Omega}{c_s})}{4\pi\Omega^2|\boldsymbol x-\boldsymbol y|}
+e^{\mathrm{i}\frac{\Omega}{c_p}|\boldsymbol x-\boldsymbol y|}\frac{\boldsymbol{\mathcal{B}}(\boldsymbol x,\boldsymbol y,\frac{\Omega}{c_p})}{4\pi\Omega^2|\boldsymbol x-\boldsymbol y|}\big)\breve{\mathbf{T}}^{m}_{n}(\boldsymbol y)\mathbf{d}\sigma(\boldsymbol y)\\
&+\mu f(\Omega)\frac{1-c(\Omega)}{\tau_{\mathbf{T},n}(\Omega\delta)}\int_{\partial D}\big(e^{\mathrm{i}\frac{\Omega}{c_s}|\boldsymbol z-\mathbf{s}|}\frac{\boldsymbol{\mathscr{M}}(\boldsymbol z,\mathbf{s},\frac{\Omega}{c_s})\boldsymbol\nu_{\boldsymbol x}}{4\pi\Omega^2|\boldsymbol z-\mathbf{s}|}
+e^{\mathrm{i}\frac{\Omega}{c_p}|\boldsymbol z-\mathbf{s}|}\frac{\boldsymbol{\mathscr{N}}(\boldsymbol z,\mathbf{s},\frac{\Omega}{c_p})\boldsymbol\nu_{\boldsymbol x}}{4\pi\Omega^2|\boldsymbol z-\mathbf{s}|}\big)\cdot \breve{\mathbf{T}}^{m}_{n}(\boldsymbol x)\mathbf{d}\sigma(\boldsymbol x) \\
&\quad\quad\quad\quad\quad\quad\times\int_{\partial D}\big(e^{\mathrm{i}\frac{\Omega}{c_s}|\boldsymbol x-\boldsymbol y|}\frac{\boldsymbol{\mathcal{A}}(\boldsymbol x,\boldsymbol y,\frac{\Omega}{c_s})}{4\pi\Omega^2|\boldsymbol x-\boldsymbol y|}
+e^{\mathrm{i}\frac{\Omega}{c_p}|\boldsymbol x-\boldsymbol y|}\frac{\boldsymbol{\mathcal{B}}(\boldsymbol x,\boldsymbol y,\frac{\Omega}{c_p})}{4\pi\Omega^2|\boldsymbol x-\boldsymbol y|}\big)\breve{\mathbf{T}}^{m}_{n}(\boldsymbol y)\mathbf{d}\sigma(\boldsymbol y)\\
=&\boldsymbol{\Xi}^{m}_{\lambda,\mathbf{T},n}(\boldsymbol x,\Omega)+\boldsymbol{\Xi}^{m}_{\mu,\mathbf{T},n}(\boldsymbol x,\Omega)
\end{align*}
with
\begin{align*}
\boldsymbol{\Xi}^{m}_{\lambda,\mathbf{T},n}(\boldsymbol x,\Omega):=&\lambda f(\Omega)\frac{1-c(\Omega)}{\tau_{\mathbf{T},n}(\Omega\delta)}\\
&\times\big(\int_{\partial D\times\partial D}e^{\mathrm{i}\frac{\Omega}{c_s}|\boldsymbol z-\mathbf{s}|+\mathrm{i}\frac{\Omega}{c_s}|\boldsymbol x-\boldsymbol y|}
\frac{\mathscr{P}(\boldsymbol z,\mathbf{s},\frac{\Omega}{c_s})\boldsymbol\nu_{\mathbf{v}}\cdot\breve{\mathbf{T}}^{m}_{n}(\mathbf{v})}{4\pi\Omega^2|\boldsymbol z-\mathbf{s}|}
\frac{\boldsymbol{\mathcal{A}}(\boldsymbol x,\boldsymbol y,\frac{\Omega}{c_s})\breve{\mathbf{T}}^{m}_{n}(\boldsymbol y)}{4\pi\Omega^2|\boldsymbol x-\boldsymbol y|}\mathrm{d}\sigma(\mathbf{v})\mathrm{d}\sigma(\boldsymbol y)\\
&+\int_{\partial D\times\partial D}e^{\mathrm{i}\frac{\Omega}{c_s}|\boldsymbol z-\mathbf{s}|+\mathrm{i}\frac{\Omega}{c_p}|\boldsymbol x-\boldsymbol y|}
\frac{\mathscr{P}(\boldsymbol z,\mathbf{s},\frac{\Omega}{c_s})\boldsymbol\nu_{\mathbf{v}}\cdot\breve{\mathbf{T}}^{m}_{n}(\mathbf{v})}{4\pi\Omega^2|\boldsymbol z-\mathbf{s}|}
\frac{\boldsymbol{\mathcal{B}}(\boldsymbol x,\boldsymbol y,\frac{\Omega}{c_p})\breve{\mathbf{T}}^{m}_{n}(\boldsymbol y)}{4\pi\Omega^2|\boldsymbol x-\boldsymbol y|}\mathrm{d}\sigma(\mathbf{v})\mathrm{d}\sigma(\boldsymbol y)\\
&+\int_{\partial D\times\partial D}e^{\mathrm{i}\frac{\Omega}{c_p}|\boldsymbol z-\mathbf{s}|+\mathrm{i}\frac{\Omega}{c_s}|\boldsymbol x-\boldsymbol y|}
\frac{\mathscr{Q}(\boldsymbol z,\mathbf{s},\frac{\Omega}{c_p})\boldsymbol\nu_{\mathbf{v}}\cdot\breve{\mathbf{T}}^{m}_{n}(\mathbf{v})}{4\pi\Omega^2|\boldsymbol z-\mathbf{s}|}
\frac{\boldsymbol{\mathcal{A}}(\boldsymbol x,\boldsymbol y,\frac{\Omega}{c_s})\breve{\mathbf{T}}^{m}_{n}(\boldsymbol y)}{4\pi\Omega^2|\boldsymbol x-\boldsymbol y|}\mathrm{d}\sigma(\mathbf{v})\mathrm{d}\sigma(\boldsymbol y)\\
&+\int_{\partial D\times\partial D}e^{\mathrm{i}\frac{\Omega}{c_p}|\boldsymbol z-\mathbf{s}|+\mathrm{i}\frac{\Omega}{c_p}|\boldsymbol x-\boldsymbol y|}
\frac{\mathscr{Q}(\boldsymbol z,\mathbf{s},\frac{\Omega}{c_p})\boldsymbol\nu_{\mathbf{v}}\cdot\breve{\mathbf{T}}^{m}_{n}(\mathbf{v})}{4\pi\Omega^2|\boldsymbol z-\mathbf{s}|}
\frac{\boldsymbol{\mathcal{B}}(\boldsymbol x,\boldsymbol y,\frac{\Omega}{c_p})\breve{\mathbf{T}}^{m}_{n}(\boldsymbol y)}{4\pi\Omega^2|\boldsymbol x-\boldsymbol y|}\mathrm{d}\sigma(\mathbf{v})\mathrm{d}\sigma(\boldsymbol y)\big),
\end{align*}
and
\begin{align*}
\boldsymbol{\Xi}^{m}_{\mu,\mathbf{T},n}(\boldsymbol x,\Omega):=&\mu f(\Omega)\frac{1-c(\Omega)}{\tau_{\mathbf{T},n}(\Omega\delta)}\\
&\times\big(\int_{\partial D\times\partial D}e^{\mathrm{i}\frac{\Omega}{c_s}|\boldsymbol z-\mathbf{s}|+\mathrm{i}\frac{\Omega}{c_s}|\boldsymbol x-\boldsymbol y|}
\frac{\boldsymbol{\mathscr{M}}(\boldsymbol z,\mathbf{s},\frac{\Omega}{c_s})\boldsymbol\nu_{\mathbf{v}}\cdot\breve{\mathbf{T}}^{m}_{n}(\mathbf{v})}{4\pi\Omega^2|\boldsymbol z-\mathbf{s}|}
\frac{\boldsymbol{\mathcal{A}}(\boldsymbol x,\boldsymbol y,\frac{\Omega}{c_s})\breve{\mathbf{T}}^{m}_{n}(\boldsymbol y)}{4\pi\Omega^2|\boldsymbol x-\boldsymbol y|}\mathrm{d}\sigma(\mathbf{v})\mathrm{d}\sigma(\boldsymbol y)\\
&+\int_{\partial D\times\partial D}e^{\mathrm{i}\frac{\Omega}{c_s}|\boldsymbol z-\mathbf{s}|+\mathrm{i}\frac{\Omega}{c_p}|\boldsymbol x-\boldsymbol y|}
\frac{\boldsymbol{\mathscr{M}}(\boldsymbol z,\mathbf{s},\frac{\Omega}{c_s})\boldsymbol\nu_{\mathbf{v}}\cdot\breve{\mathbf{T}}^{m}_{n}(\mathbf{v})}{4\pi\Omega^2|\boldsymbol z-\mathbf{s}|}
\frac{\boldsymbol{\mathcal{B}}(\boldsymbol x,\boldsymbol y,\frac{\Omega}{c_p})\breve{\mathbf{T}}^{m}_{n}(\boldsymbol y)}{4\pi\Omega^2|\boldsymbol x-\boldsymbol y|}\mathrm{d}\sigma(\mathbf{v})\mathrm{d}\sigma(\boldsymbol y)\\
&+\int_{\partial D\times\partial D}e^{\mathrm{i}\frac{\Omega}{c_p}|\boldsymbol z-\mathbf{s}|+\mathrm{i}\frac{\Omega}{c_s}|\boldsymbol x-\boldsymbol y|}
\frac{\boldsymbol{\mathscr{N}}(\boldsymbol z,\mathbf{s},\frac{\Omega}{c_p})\boldsymbol\nu_{\mathbf{v}}\cdot\breve{\mathbf{T}}^{m}_{n}(\mathbf{v})}{4\pi\Omega^2|\boldsymbol z-\mathbf{s}|}
\frac{\boldsymbol{\mathcal{A}}(\boldsymbol x,\boldsymbol y,\frac{\Omega}{c_s})\breve{\mathbf{T}}^{m}_{n}(\boldsymbol y)}{4\pi\Omega^2|\boldsymbol x-\boldsymbol y|}\mathrm{d}\sigma(\mathbf{v})\mathrm{d}\sigma(\boldsymbol y)\\
&+\int_{\partial D\times\partial D}e^{\mathrm{i}\frac{\Omega}{c_p}|\boldsymbol z-\mathbf{s}|+\mathrm{i}\frac{\Omega}{c_p}|\boldsymbol x-\boldsymbol y|}
\frac{\boldsymbol{\mathscr{N}}(\boldsymbol z,\mathbf{s},\frac{\Omega}{c_p})\boldsymbol\nu_{\mathbf{v}}\cdot\breve{\mathbf{T}}^{m}_{n}(\mathbf{v})}{4\pi\Omega^2|\boldsymbol z-\mathbf{s}|}
\frac{\boldsymbol{\mathcal{B}}(\boldsymbol x,\boldsymbol y,\frac{\Omega}{c_p})\breve{\mathbf{T}}^{m}_{n}(\boldsymbol y)}{4\pi\Omega^2|\boldsymbol x-\boldsymbol y|}\mathrm{d}\sigma(\mathbf{v})\mathrm{d}\sigma(\boldsymbol y)\big).
\end{align*}
As a consequence,
\begin{align*}
\int_{\mathcal{C}^{\pm}_{\rho}}\boldsymbol{\Xi}^{m}_{\mathbf{T},n}(\boldsymbol x,\Omega)e^{-\mathrm{i}\Omega t}\mathrm{d}\Omega=&
\int_{\mathcal{C}^{\pm}_{\rho}}\boldsymbol{\Xi}^{m}_{\lambda,\mathbf{T},n}(\boldsymbol x,\Omega)e^{-\mathrm{i}\Omega t}\mathrm{d}\Omega+\int_{\mathcal{C}^{\pm}_{\rho}}\boldsymbol{\Xi}^{m}_{\mu,\mathbf{T},n}(\boldsymbol x,\Omega)e^{-\mathrm{i}\Omega t}\mathrm{d}\Omega\\=&\boldsymbol{\Phi}^m_{1,n}+\boldsymbol{\Phi}^m_{2,n}+\boldsymbol{\Phi}^m_{3,n}+\boldsymbol{\Phi}^m_{4,n},
\end{align*}
with
\begin{align*}
\boldsymbol{\Phi}^m_{1,n}:=&\int_{\mathcal{C}^{\pm}_{\rho}}f(\Omega)\int_{\partial D\times\partial D}\boldsymbol{\mathscr{A}}^m_{n}(\mathbf{v},\boldsymbol y,\Omega)e^{\mathrm{i}\Omega
(\frac{|\boldsymbol z-\mathbf{s}|}{c_s}+\frac{|\boldsymbol x-\boldsymbol y|}{c_s}-t)}\mathrm{d}\sigma(\mathbf{v})\mathrm{d}\sigma(\boldsymbol y)\mathrm{d}\Omega,\\
\boldsymbol{\Phi}^m_{2,n}:=&\int_{\mathcal{C}^{\pm}_{\rho}}f(\Omega)\int_{\partial D\times\partial D}\boldsymbol{\mathscr{B}}^m_{n}(\mathbf{v},\boldsymbol y,\Omega)e^{\mathrm{i}\Omega
(\frac{|\boldsymbol z-\mathbf{s}|}{c_s}+\frac{|\boldsymbol x-\boldsymbol y|}{c_p}-t)}\mathrm{d}\sigma(\mathbf{v})\mathrm{d}\sigma(\boldsymbol y)\mathrm{d}\Omega,\\
\boldsymbol{\Phi}^m_{3,n}:=&\int_{\mathcal{C}^{\pm}_{\rho}}f(\Omega)\int_{\partial D\times\partial D}\boldsymbol{\mathscr{C}}^m_{n}(\mathbf{v},\boldsymbol y,\Omega)e^{\mathrm{i}\Omega
(\frac{|\boldsymbol z-\mathbf{s}|}{c_p}+\frac{|\boldsymbol x-\boldsymbol y|}{c_s}-t)}\mathrm{d}\sigma(\mathbf{v})\mathrm{d}\sigma(\boldsymbol y)\mathrm{d}\Omega,\\
\boldsymbol{\Phi}^m_{4,n}:=&\int_{\mathcal{C}^{\pm}_{\rho}}f(\Omega)\int_{\partial D\times\partial D}\boldsymbol{\mathscr{D}}^m_{n}(\mathbf{v},\boldsymbol y,\Omega)e^{\mathrm{i}\Omega
(\frac{|\boldsymbol z-\mathbf{s}|}{c_p}+\frac{|\boldsymbol x-\boldsymbol y|}{c_p}-t)}\mathrm{d}\sigma(\mathbf{v})\mathrm{d}\sigma(\boldsymbol y)\mathrm{d}\Omega,
\end{align*}
where
\begin{align*}
\boldsymbol{\mathscr{A}}^m_{n}(\mathbf{v},\boldsymbol y,\Omega)=&\frac{\lambda (1-c(\Omega))}{\tau_{\mathbf{T},n}(\Omega\delta)}\frac{\mathscr{P}(\boldsymbol z,\mathbf{s},\frac{\Omega}{c_s})\boldsymbol{\nu}_{\mathbf{v}}\cdot\breve{\mathbf{T}}^{m}_{n}(\mathbf{v})
\boldsymbol{\mathcal{A}}(\boldsymbol x,\boldsymbol y,\frac{\Omega}{c_s})\breve{\mathbf{T}}^{m}_{n}(\boldsymbol y)}{16\pi^2\Omega^4|\boldsymbol z-\mathbf{s}||\boldsymbol x-\boldsymbol y|}\\
&+\frac{\mu (1-c(\Omega))}{\tau_{\mathbf{T},n}(\Omega\delta)}\frac{\boldsymbol{\mathscr{M}}(\boldsymbol z,\mathbf{s},\frac{\Omega}{c_s})\boldsymbol{\nu}_{\mathbf{v}}\cdot\breve{\mathbf{T}}^{m}_{n}(\mathbf{v})
\boldsymbol{\mathcal{A}}(\boldsymbol x,\boldsymbol y,\frac{\Omega}{c_s})\breve{\mathbf{T}}^{m}_{n}(\boldsymbol y)}{16\pi^2\Omega^4|\boldsymbol z-\mathbf{s}||\boldsymbol x-\boldsymbol y|},\\
\boldsymbol{\mathscr{B}}^m_{n}(\mathbf{v},\boldsymbol y,\Omega)=&\frac{\lambda (1-c(\Omega))}{\tau_{\mathbf{T},n}(\Omega\delta)}\frac{\mathscr{P}(\boldsymbol z,\mathbf{s},\frac{\Omega}{c_s})\boldsymbol{\nu}_{\mathbf{v}}\cdot\breve{\mathbf{T}}^{m}_{n}(\mathbf{v})
\boldsymbol{\mathcal{B}}(\boldsymbol x,\boldsymbol y,\frac{\Omega}{c_p})\breve{\mathbf{T}}^{m}_{n}(\boldsymbol y)}{16\pi^2\Omega^4|\boldsymbol z-\mathbf{s}||\boldsymbol x-\boldsymbol y|}\\
&+\frac{\mu (1-c(\Omega))}{\tau_{\mathbf{T},n}(\Omega\delta)}\frac{\boldsymbol{\mathscr{M}}(\boldsymbol z,\mathbf{s},\frac{\Omega}{c_s})\boldsymbol{\nu}_{\mathbf{v}}\cdot\breve{\mathbf{T}}^{m}_{n}(\mathbf{v})
\boldsymbol{\mathcal{B}}(\boldsymbol x,\boldsymbol y,\frac{\Omega}{c_p})\breve{\mathbf{T}}^{m}_{n}(\boldsymbol y)}{16\pi^2\Omega^4|\boldsymbol z-\mathbf{s}||\boldsymbol x-\boldsymbol y|},\\
\boldsymbol{\mathscr{C}}^m_{n}(\mathbf{v},\boldsymbol y,\Omega)=&\frac{\lambda (1-c(\Omega))}{\tau_{\mathbf{T},n}(\Omega\delta)}\frac{\mathscr{Q}(\boldsymbol z,\mathbf{s},\frac{\Omega}{c_p})\boldsymbol{\nu}_{\mathbf{v}}\cdot\breve{\mathbf{T}}^{m}_{n}(\mathbf{v})
\boldsymbol{\mathcal{A}}(\boldsymbol x,\boldsymbol y,\frac{\Omega}{c_s})\breve{\mathbf{T}}^{m}_{n}(\boldsymbol y)}{16\pi^2\Omega^4|\boldsymbol z-\mathbf{s}||\boldsymbol x-\boldsymbol y|}\\
&+\frac{\mu (1-c(\Omega))}{\tau_{\mathbf{T},n}(\Omega\delta)}\frac{\boldsymbol{\mathscr{N}}(\boldsymbol z,\mathbf{s},\frac{\Omega}{c_p})\boldsymbol{\nu}_{\mathbf{v}}\cdot\breve{\mathbf{T}}^{m}_{n}(\mathbf{v})
\boldsymbol{\mathcal{A}}(\boldsymbol x,\boldsymbol y,\frac{\Omega}{c_s})\breve{\mathbf{T}}^{m}_{n}(\boldsymbol y)}{16\pi^2\Omega^4|\boldsymbol z-\mathbf{s}||\boldsymbol x-\boldsymbol y|},\\
\boldsymbol{\mathscr{D}}^m_{n}(\mathbf{v},\boldsymbol y,\Omega)=&\frac{\lambda (1-c(\Omega))}{\tau_{\mathbf{T},n}(\Omega\delta)}\frac{\mathscr{Q}(\boldsymbol z,\mathbf{s},\frac{\Omega}{c_p})\boldsymbol{\nu}_{\mathbf{v}}\cdot\breve{\mathbf{T}}^{m}_{n}(\mathbf{v})
\boldsymbol{\mathcal{B}}(\boldsymbol x,\boldsymbol y,\frac{\Omega}{c_p})\breve{\mathbf{T}}^{m}_{n}(\boldsymbol y)}{16\pi^2\Omega^4|\boldsymbol z-\mathbf{s}||\boldsymbol x-\boldsymbol y|}\\
&+\frac{\mu (1-c(\Omega))}{\tau_{\mathbf{T},n}(\Omega\delta)}\frac{\boldsymbol{\mathscr{N}}(\boldsymbol z,\mathbf{s},\frac{\Omega}{c_p})\boldsymbol{\nu}_{\mathbf{v}}\cdot\breve{\mathbf{T}}^{m}_{n}(\mathbf{v})
\boldsymbol{\mathcal{B}}(\boldsymbol x,\boldsymbol y,\frac{\Omega}{c_p})\breve{\mathbf{T}}^{m}_{n}(\boldsymbol y)}{16\pi^2\Omega^4|\boldsymbol z-\mathbf{s}||\boldsymbol x-\boldsymbol y|}.
 \end{align*}
Observe that $\boldsymbol{\mathscr{A}}^m_{n}(\cdot,\cdot,\Omega),\boldsymbol{\mathscr{B}}^m_{n}(\cdot,\cdot,\Omega),\boldsymbol{\mathscr{C}}^m_{n}(\cdot,\cdot,\Omega),\boldsymbol{\mathscr{D}}^m_{n}(\cdot,\cdot,\Omega)$ behave like a polynomial in $\Omega$ when $|\Omega|\rightarrow\infty$. Provided the regularity of the wideband signal $\hat{f}:t\mapsto \hat{f}(t)\in C^{\infty}_0([0,C_1])$, the Paley-Wiener theorem ensures the decay property of its Fourier transform at infinity. For all $M\in\mathbb{N}^{+}$, there is some constant $C_{M}>0$ such that for all $\Omega\in\mathbb{C}$,
\begin{align*}
|f(\Omega)|\leq C_{M}(1+|\Omega|)^{-M}e^{C_1|\Im(\Omega)|}.
\end{align*}

The rest of the discussion is divided into the following two cases.

Case {\rm(i)}: $t<t^{-}_0$. Let us consider the upper-half integration contour $\mathcal{C}^+$. Provided the polar coordinate transform 
\begin{align*}
\Omega=\rho e^{\mathrm{i}\theta},\quad \theta\in[0,\pi],
\end{align*}
we can derive that
\begin{align*}
\boldsymbol\Phi^m_{1,n}=&\int_{0}^{\pi}\mathrm{i}\rho e^{\mathrm{i}\theta}f(\rho e^{\mathrm{i}\theta})\int_{\partial D\times\partial D}\boldsymbol{\mathscr{A}}^m_{n}(\mathbf{v},\boldsymbol y,\rho e^{\mathrm{i}\theta})e^{\mathrm{i}\rho e^{\mathrm{i}\theta}
(\frac{|\boldsymbol z-\mathbf{s}|}{c_s}+\frac{|\boldsymbol x-\boldsymbol y|}{c_s}-t)}\mathrm{d}\sigma(\mathbf{v})\mathrm{d}\sigma(\boldsymbol y)\mathrm{d}\theta,\\
\boldsymbol\Phi^m_{2,n}=&\int_{0}^{\pi}\mathrm{i}\rho e^{\mathrm{i}\theta}f(\rho e^{\mathrm{i}\theta})\int_{\partial D\times\partial D}\boldsymbol{\mathscr{B}}^m_{n}(\mathbf{v},\boldsymbol y,\rho e^{\mathrm{i}\theta})e^{\mathrm{i}\rho e^{\mathrm{i}\theta}
(\frac{|\boldsymbol z-\mathbf{s}|}{c_s}+\frac{|\boldsymbol x-\boldsymbol y|}{c_p}-t)}\mathrm{d}\sigma(\mathbf{v})\mathrm{d}\sigma(\boldsymbol y)\mathrm{d}\theta,\\
\boldsymbol\Phi^m_{3,n}=&\int_{0}^{\pi}\mathrm{i}\rho e^{\mathrm{i}\theta}f(\rho e^{\mathrm{i}\theta})\int_{\partial D\times\partial D}\boldsymbol{\mathscr{C}}^m_{n}(\mathbf{v},\boldsymbol y,\rho e^{\mathrm{i}\theta})e^{\mathrm{i}\rho e^{\mathrm{i}\theta}
(\frac{|\boldsymbol z-\mathbf{s}|}{c_p}+\frac{|\boldsymbol x-\boldsymbol y|}{c_s}-t)}\mathrm{d}\sigma(\mathbf{v})\mathrm{d}\sigma(\boldsymbol y)\mathrm{d}\theta,\\
\boldsymbol\Phi^m_{4,n}=&\int_{0}^{\pi}\mathrm{i}\rho e^{\mathrm{i}\theta}f(\rho e^{\mathrm{i}\theta})\int_{\partial D\times\partial D}\boldsymbol{\mathscr{D}}^m_{n}(\mathbf{v},\boldsymbol y,\rho e^{\mathrm{i}\theta})e^{\mathrm{i}\rho e^{\mathrm{i}\theta}
(\frac{|\boldsymbol z-\mathbf{s}|}{c_p}+\frac{|\boldsymbol x-\boldsymbol y|}{c_p}-t)}\mathrm{d}\sigma(\mathbf{v})\mathrm{d}\sigma(\boldsymbol y)\mathrm{d}\theta.
\end{align*}

Let us first estimate $\boldsymbol\Phi^m_{1,n}$. For $t<\frac{|\boldsymbol z-\mathbf{s}|}{c_s}+\frac{|\boldsymbol x-\boldsymbol z|}{c_s}-\frac{\delta}{c_s}-C_1$, we obtain
\begin{align*}
|e^{\mathrm{i}\rho e^{\mathrm{i}\theta}
(\frac{|\boldsymbol z-\mathbf{s}|}{c_s}+\frac{|\boldsymbol x-\boldsymbol y|}{c_s}-C_1-t)}|\leq e^{-\rho\sin \theta(\frac{|\boldsymbol z-\mathbf{s}|}{c_s}+\frac{|\boldsymbol x-\boldsymbol z|}{c_s}-\frac{\delta}{c_s}-C_1-t)},\quad \forall\boldsymbol y\in\partial D.
\end{align*}
This arrives at
\begin{align*}
|\boldsymbol\Phi^m_{1,n}|
\leq&2 C_M\delta^4\rho(1+\rho)^{-M}\max_{\theta\in[0,\pi]}\|\boldsymbol{\mathscr{A}}^m_{n}(\cdot,\cdot,\rho  e^{\mathrm{i}\theta})\|_{L^{\infty}(\partial D\times \partial D)}\int^{\frac{\pi}{2}}_{0}e^{-\rho\sin \theta(\frac{|\boldsymbol z-\mathbf{s}|}{c_s}+\frac{|\boldsymbol x-\boldsymbol z|}{c_s}-\frac{\delta}{c_s}-C_1-t)}\mathrm{d}\theta.
\end{align*}
Moreover, it is clear that $\sin\theta\geq\frac{2}{\pi}\theta,\forall \theta\in[0,\frac{\pi}{2}].$
Hence, for $t<\frac{|\boldsymbol z-\mathbf{s}|}{c_s}+\frac{|\boldsymbol x-\boldsymbol z|}{c_s}-\frac{\delta}{c_s}-C_1$,
\begin{align*}
e^{-\rho\sin \theta(\frac{|\boldsymbol z-\mathbf{s}|}{c_s}+\frac{|\boldsymbol x-\boldsymbol z|}{c_s}-\frac{\delta}{c_s}-C_1-t)}\leq e^{-\frac{2\rho}{\pi}\theta(\frac{|\boldsymbol z-\mathbf{s}|}{c_s}+\frac{|\boldsymbol x-\boldsymbol z|}{c_s}-\frac{\delta}{c_s}-C_1-t)},
\end{align*}
which then gives
\begin{align*}
\int^{\frac{\pi}{2}}_0e^{-\rho\sin \theta(\frac{|\boldsymbol z-\mathbf{s}|}{c_s}+\frac{|\boldsymbol x-\boldsymbol z|}{c_s}-\frac{\delta}{c_s}-C_1-t)}\mathrm{d}\theta\leq & \int^{\frac{\pi}{2}}_0e^{-\frac{2\rho}{\pi}\theta(\frac{|\boldsymbol z-\mathbf{s}|}{c_s}+\frac{|\boldsymbol x-\boldsymbol z|}{c_s}-\frac{\delta}{c_s}-C_1-t)}\mathrm{d}\theta\\
=&\frac{\pi\big(1-e^{-\rho(\frac{|\boldsymbol z-\mathbf{s}|}{c_s}+\frac{|\boldsymbol x-\boldsymbol z|}{c_s}-\frac{\delta}{c_s}-C_1-t)}\big)}{2\rho\big(\frac{|\boldsymbol z-\mathbf{s}|}{c_s}+\frac{|\boldsymbol x-\boldsymbol z|}{c_s}-\frac{\delta}{c_s}-C_1-t\big)}.
\end{align*}
Hence,
\begin{align*}
|\boldsymbol\Phi^m_{1,n}|\leq C_M\delta^4(1+\rho)^{-M}\max_{\theta\in[0,\pi]}\|\boldsymbol{\mathscr{A}}^m_{n}(\cdot,\cdot,\rho  e^{\mathrm{i}\theta})\|_{L^{\infty}(\partial D\times \partial D)}\frac{\pi\big(1-e^{-\rho(\frac{|\boldsymbol z-\mathbf{s}|}{c_s}+\frac{|\boldsymbol x-\boldsymbol z|}{c_s}-\frac{\delta}{c_s}-C_1-t)}\big)}{\frac{|\boldsymbol z-\mathbf{s}|}{c_s}+\frac{|\boldsymbol x-\boldsymbol z|}{c_s}-\frac{\delta}{c_s}-C_1-t}.
\end{align*}
As a consequence,
\begin{align*}
|\boldsymbol\Phi^m_{1,n}|=\mathcal{O}\big(\frac{\delta^4\rho^{-M}}{\frac{|\boldsymbol z-\mathbf{s}|}{c_s}+\frac{|\boldsymbol x-\boldsymbol z|}{c_s}-\frac{\delta}{c_s}-C_1-t}\big).
\end{align*}
By proceeding the similar technique as above, one has
\begin{align*}
&|\boldsymbol\Phi^m_{2,n}|=\mathcal{O}\big(\frac{\delta^4\rho^{-M}}{\frac{|\boldsymbol z-\mathbf{s}|}{c_s}+\frac{|\boldsymbol x-\boldsymbol z|}{c_p}-\frac{\delta}{c_p}-C_1-t}\big),
\quad\text{if }t<\frac{|\boldsymbol z-\mathbf{s}|}{c_s}+\frac{|\boldsymbol x-\boldsymbol z|}{c_p}-\frac{\delta}{c_p}-C_1,\\
&|\boldsymbol\Phi^m_{3,n}|=\mathcal{O}\big(\frac{\delta^4\rho^{-M}}{\frac{|\boldsymbol z-\mathbf{s}|}{c_p}+\frac{|\boldsymbol x-\boldsymbol z|}{c_s}-\frac{\delta}{c_s}-C_1-t}\big),
\quad\text{if }t<\frac{|\boldsymbol z-\mathbf{s}|}{c_p}+\frac{|\boldsymbol x-\boldsymbol z|}{c_s}-\frac{\delta}{c_s}-C_1,\\
&|\boldsymbol\Phi^m_{4,n}|=\mathcal{O}\big(\frac{\delta^4\rho^{-M}}{\frac{|\boldsymbol z-\mathbf{s}|}{c_p}+\frac{|\boldsymbol x-\boldsymbol z|}{c_p}-\frac{\delta}{c_p}-C_1-t}\big),
\quad\text{if }t<\frac{|\boldsymbol z-\mathbf{s}|}{c_p}+\frac{|\boldsymbol x-\boldsymbol z|}{c_p}-\frac{\delta}{c_p}-C_1.
\end{align*}
Furthermore, according to formula \eqref{E:speed}, we get
\begin{align}\label{relation}
c_p=\sqrt{(\lambda+\mu)+\mu}=\sqrt{\frac{1}{3}\cdot3(\lambda+\mu)+\mu}\stackrel{\eqref{convexity}}{\geq}\sqrt{\frac{1}{3}\mu+\mu}\geq c_s.
\end{align}
Hence  the first term in \eqref{E:scattered} holds for all $t\leq t^{-}_0$.

Case {\rm(ii)}: $t>t^{+}_0$. We study the lower-half integration contour $\mathcal{C}^-$. By using the polar coordinate transform
\begin{align*}
\Omega=\rho e^{\mathrm{i}\theta},\quad \theta\in[\pi,2\pi],
\end{align*}
we can deduce
\begin{align*}
\boldsymbol\Phi^m_{1,n}=&\int_{\pi}^{2\pi}\mathrm{i}\rho e^{\mathrm{i}\theta}f(\rho e^{\mathrm{i}\theta})\int_{\partial D\times\partial D}\boldsymbol{\mathscr{A}}^m_{n}(\mathbf{v},\boldsymbol y,\rho e^{\mathrm{i}\theta})e^{\mathrm{i}\rho e^{\mathrm{i}\theta}(\frac{|\boldsymbol z-\mathbf{s}|}{c_s}+\frac{|\boldsymbol x-\boldsymbol y|}{c_s}-t)}\mathrm{d}\sigma(\mathbf{v})\mathrm{d}\sigma(\boldsymbol y)\mathrm{d}\theta,\\
\boldsymbol\Phi^m_{2,n}=&\int_{\pi}^{2\pi}\mathrm{i}\rho e^{\mathrm{i}\theta}f(\rho e^{\mathrm{i}\theta})\int_{\partial D\times\partial D}\boldsymbol{\mathscr{B}}^m_{n}(\mathbf{v},\boldsymbol y,\rho e^{\mathrm{i}\theta})e^{\mathrm{i}\rho e^{\mathrm{i}\theta}
(\frac{|\boldsymbol z-\mathbf{s}|}{c_s}+\frac{|\boldsymbol x-\boldsymbol y|}{c_p}-t)}\mathrm{d}\sigma(\mathbf{v})\mathrm{d}\sigma(\boldsymbol y)\mathrm{d}\theta,\\
\boldsymbol\Phi^m_{3,n}=&\int_{\pi}^{2\pi}\mathrm{i}\rho e^{\mathrm{i}\theta}f(\rho e^{\mathrm{i}\theta})\int_{\partial D\times\partial D}\boldsymbol{\mathscr{C}}^m_{n}(\mathbf{v},\boldsymbol y,\rho e^{\mathrm{i}\theta})e^{\mathrm{i}\rho e^{\mathrm{i}\theta}
(\frac{|\boldsymbol z-\mathbf{s}|}{c_p}+\frac{|\boldsymbol x-\boldsymbol y|}{c_s}-t)}\mathrm{d}\sigma(\mathbf{v})\mathrm{d}\sigma(\boldsymbol y)\mathrm{d}\theta,\\
\boldsymbol\Phi^m_{4,n}=&\int_{\pi}^{2\pi}\mathrm{i}\rho e^{\mathrm{i}\theta}f(\rho e^{\mathrm{i}\theta})\int_{\partial D\times\partial D}\boldsymbol{\mathscr{D}}^m_{n}(\mathbf{v},\boldsymbol y,\rho e^{\mathrm{i}\theta})e^{\mathrm{i}\rho e^{\mathrm{i}\theta}
(\frac{|\boldsymbol z-\mathbf{s}|}{c_p}+\frac{|\boldsymbol x-\boldsymbol y|}{c_p}-t)}\mathrm{d}\sigma(\mathbf{v})\mathrm{d}\sigma(\boldsymbol y)\mathrm{d}\theta.
\end{align*}
For $t>\frac{|\boldsymbol z-\mathbf{s}|}{c_s}+\frac{|\boldsymbol x-\boldsymbol z|}{c_s}+\frac{\delta}{c_s}+C_1$, we have that for all $\boldsymbol y\in\partial D$,
\begin{align*}
|e^{\mathrm{i}\rho e^{\mathrm{i}\theta}
(\frac{|\boldsymbol z-\mathbf{s}|}{c_s}+\frac{|\boldsymbol x-\boldsymbol y|}{c_s}+C_1-t)}|
\leq& e^{\rho\sin \theta(t-(\frac{|\boldsymbol z-\mathbf{s}|}{c_s}+\frac{|\boldsymbol x-\boldsymbol z|}{c_s}+\frac{\delta}{c_s}+C_1))},
\end{align*}
which leads to
\begin{align*}
|\boldsymbol\Phi^m_{1,n}|
\leq&2 C_M\delta^4\rho(1+\rho)^{-M}\max_{\theta\in[\pi,2\pi]}\|\boldsymbol{\mathscr{A}}^m_{n}(\cdot,\cdot,\rho  e^{\mathrm{i}\theta})\|_{L^{\infty}(\partial D\times \partial D)}\int^{\frac{3\pi}{2}}_{\pi}e^{\rho\sin \theta(t-(\frac{|\boldsymbol z-\mathbf{s}|}{c_s}+\frac{|\boldsymbol x-\boldsymbol z|}{c_s}+\frac{\delta}{c_s}+C_1))}\mathrm{d}\theta.
\end{align*}
Moreover, it is straightforward that $-\frac{2}{\pi}\theta+2\geq\sin\theta, \forall \theta\in[\pi,\frac{3\pi}{2}].$
One has that for $t>\frac{|\boldsymbol z-\mathbf{s}|}{c_s}+\frac{|\boldsymbol x-\boldsymbol z|}{c_s}+\frac{\delta}{c_s}+C_1$,
\begin{align*}
e^{\rho\sin \theta(t-(\frac{|\boldsymbol z-\mathbf{s}|}{c_s}+\frac{|\boldsymbol x-\boldsymbol z|}{c_s}+\frac{\delta}{c_s}+C_1))}\leq e^{\rho(-\frac{2}{\pi}\theta+2)(t-(\frac{|\boldsymbol z-\mathbf{s}|}{c_s}+\frac{|\boldsymbol x-\boldsymbol z|}{c_s}+\frac{\delta}{c_s}+C_1))},
\end{align*}
which then gives
\begin{align*}
\int^{\frac{3\pi}{2}}_{\pi}e^{\rho\sin \theta(t-(\frac{|\boldsymbol z-\mathbf{s}|}{c_s}+\frac{|\boldsymbol x-\boldsymbol z|}{c_s}+\frac{\delta}{c_s}+C_1))}\mathrm{d}\theta\leq & \int^{\frac{3\pi}{2}}_{\pi}e^{\rho(-\frac{2}{\pi}\theta+2)(t-(\frac{|\boldsymbol z-\mathbf{s}|}{c_s}+\frac{|\boldsymbol x-\boldsymbol z|}{c_s}+\frac{\delta}{c_s}+C_1))}\mathrm{d}\theta\\
=&\frac{\pi\big(1-e^{-\rho(t-(\frac{|\boldsymbol z-\mathbf{s}|}{c_s}+\frac{|\boldsymbol x-\boldsymbol z|}{c_s}+\frac{\delta}{c_s}+C_1))}\big)}{2\rho\big(t-(\frac{|\boldsymbol z-\mathbf{s}|}{c_s}+\frac{|\boldsymbol x-\boldsymbol z|}{c_s}+\frac{\delta}{c_s}+C_1)\big)}.
\end{align*}
Therefore,
\begin{align*}
|\boldsymbol\Phi^m_{1,n}|\leq C_M\delta^4(1+\rho)^{-M}\max_{\theta\in[\pi,2\pi]}\|\boldsymbol{\mathscr{A}}^m_{n}(\cdot,\cdot,\rho  e^{\mathrm{i}\theta})\|_{L^{\infty}(\partial D\times \partial D)}\frac{\pi\big(1-e^{-\rho(t-(\frac{|\boldsymbol z-\mathbf{s}|}{c_s}+\frac{|\boldsymbol x-\boldsymbol z|}{c_s}+\frac{\delta}{c_s}+C_1))}\big)}{t-(\frac{|\boldsymbol z-\mathbf{s}|}{c_s}+\frac{|\boldsymbol x-\boldsymbol z|}{c_s}+\frac{\delta}{c_s}+C_1)},
\end{align*}
that is,
\begin{align*}
|\boldsymbol\Phi^m_{1,n}|=\mathcal{O}\big(\frac{\delta^4\rho^{-M}}{t-(\frac{|\boldsymbol z-\mathbf{s}|}{c_s}+\frac{|\boldsymbol x-\boldsymbol z|}{c_s}+\frac{\delta}{c_s}+C_1)}\big).
\end{align*}
Proceeding the similar technique as above yields that
\begin{align*}
&|\boldsymbol\Phi^m_{2,n}|=\mathcal{O}\big(\frac{\delta^4\rho^{-M}}{t-(\frac{|\boldsymbol z-\mathbf{s}|}{c_s}+\frac{|\boldsymbol x-\boldsymbol z|}{c_p}+\frac{\delta}{c_p}+C_1)}\big),
\quad\text{if }t>\frac{|\boldsymbol z-\mathbf{s}|}{c_s}+\frac{|\boldsymbol x-\boldsymbol z|}{c_p}+\frac{\delta}{c_p}+C_1,\\
&|\boldsymbol\Phi^m_{3,n}|=\mathcal{O}\big(\frac{\delta^4\rho^{-M}}{t-(\frac{|\boldsymbol z-\mathbf{s}|}{c_p}+\frac{|\boldsymbol x-\boldsymbol z|}{c_s}+\frac{\delta}{c_s}+C_1)}\big),
\quad\text{if }t>\frac{|\boldsymbol z-\mathbf{s}|}{c_p}+\frac{|\boldsymbol x-\boldsymbol z|}{c_s}+\frac{\delta}{c_s}+C_1,\\
&|\boldsymbol\Phi^m_{4,n}|=\mathcal{O}\big(\frac{\delta^4\rho^{-M}}{t-(\frac{|\boldsymbol z-\mathbf{s}|}{c_p}+\frac{|\boldsymbol x-\boldsymbol z|}{c_p}+\frac{\delta}{c_p}+C_1)}\big),
\quad\text{if }t>\frac{|\boldsymbol z-\mathbf{s}|}{c_p}+\frac{|\boldsymbol x-\boldsymbol z|}{c_p}+\frac{\delta}{c_p}+C_1.
\end{align*}
Then it follows from \eqref{relation} that formula \eqref{EE:scattered} holds for all $t\geq t^{+}_0$.


The proof of the theorem is now complete.
\end{proof}

\section{Concluding remarks}\label{sec:conclusion}

In this paper, we provided the modal analysis for the time-dependent field scattered by metamaterial quasiparticles in elastodynamics. By Fourier analysis and layer potential techniques, we first established the modal expansion of the elastic field in the frequency domain. Then through delicate spectral and asymptotic analysis, we show that the low-frequency part of the scattered field in the time domain can be well approximated by using the polariton resonant modal expansion with sharp error estimates. To our best knowledge, this work presents the first result in the literature on the modal analysis of time-dependent scattering due to negative elastic metamaterial structures. It paves the way for many potential applications of analysing the performance of elastic metamaterials  interacting with wide-band signals. Finally, we would like to remark that as also pointed in Section \ref{sec:1}, our study is mainly confined within the radial geometry and the metamaterial structure of the form \eqref{E:A}--\eqref{c:omega}, which constitutes a dispersion relation of the viscoelastic metamaterial. Nevertheless, we would like to emphasize that the theoretical framework laid in this paper can be readily extended to analysing more general elastic metamaterial structures, say e.g. the one investigated in \cite{LLL18} or elastic metamaterials with more general but specific dispersion relations. Meanwhile, the extension to non-radial geometry will lead to significant challenges due to the non-compactness and non-self-adjointness of the layer potential operators involved. We shall investigate those extensions in our forthcoming work.


 \end{document}